\documentclass[reqno]{amsart}
\usepackage{graphicx} 
 \usepackage[foot]{amsaddr}

\title{On the Profile of Singularity Formation for the Incompressible Hydrostatic Boussinesq system}

\date{\today}

\author[1]{S. Ibrahim$^{1,2}$}
\address[1]{Department of Mathematics and Statistics
University of Victoria, 3800 Finnerty Road, Victoria
BC V8P 5C2, Canada}
\address[2]{
Pacific Institute for Mathematical Sciences, 4176-2207 Main Mall, Vancouver, BC V6T 1Z4,
Canada}
\email[1]{ibrahims@uvic.ca}

\author[3]{Q. Lin$^3$}
\address[3]{School of Mathematical and Statistical Sciences, Clemson University, Clemson, SC 29634, USA}
\email[3]{quyuanl@clemson.edu}

\author[4]{L. Qian$^4$}
\address[4]{Department of Mathematics and Statistics
University of Victoria, 3800 Finnerty Road, Victoria
BC V8P 5C2, Canada}
\email[4]{ljqian@uvic.ca}

\author[5]{E. S. Titi$^{5,6,7}$}
\address[5]{Department of Applied Mathematics and Theoretical Physics, University of Cambridge, Cambridge CB3 0WA UK}
\email[5]{Edriss.Titi@maths.cam.ac.uk}

\address[6]{Department of Mathematics, Texas A\&M University, College Station, TX 77843-3368, USA}
\email[6]{titi@math.tamu.edu}
\address[7]{Department of Computer Science and Applied Mathematics, Weizmann Institute of Science, Rehovot 76100, Israel. }



\usepackage[utf8]{inputenc}
\usepackage{dsfont}
\usepackage{float}
\usepackage{amsmath}
\usepackage{amsthm}
\usepackage{floatpag,fancyhdr}
\usepackage{commath} 
\usepackage{mathtools,xparse}
\usepackage[includefoot,bottom=20pt]{geometry}
\usepackage{amssymb}
\usepackage{exercise}
\usepackage{float}
\usepackage{xcolor}
\usepackage{cases}
\usepackage{soul}
\usepackage[normalem]{ulem}
\usepackage{cite}
\newcommand{\stkout}[1]{\ifmmode\text{\sout{\ensuremath{#1}}}\else\sout{#1}\fi}

\numberwithin{equation}{section}
\usepackage{hyperref}

\hypersetup{
    colorlinks,
    citecolor=black,
    filecolor=black,
    linkcolor=black,
    urlcolor=black
}
\theoremstyle{break}
\numberwithin{equation}{section}

\newcommand{\R}{\mathbb{R}}

\newcommand{\antipartial}{\partial^{-1}}

\newcommand{\N}{\mathds{N}}

\makeatletter
\newcommand*{\rom}[1]{\expandafter\@slowromancap\romannumeral #1@}
\makeatother

\def\ee{\varepsilon}

\def\pa{\partial}
\def\ee{\end{equation}}
\def\be{\begin{equation}}

\newtheorem{theorem}{Theorem}
\newtheorem{definition}{Definition}
\newtheorem{lemma}{Lemma}

\newtheorem{remark}{Remark}
\newtheorem{corollary}{Corollary}

\newtheorem{proposition}{Proposition}

\newboolean{isSimplified}
\setboolean{isSimplified}{true}

\calclayout
\begin{document}
\begin{abstract}
    The primitive equations (PEs) model {planetary} large-scale {oceanic and atmospheric dynamics}. While it has been shown that there are smooth solutions to the inviscid PEs (also called the hydrostatic Euler equations) with constant temperature (isothermal) that develop stable singularities in finite time, the effect of {non-constant} temperature {on} the singularity formation has not been established yet. This paper studies the stability of singularity formation {for non-constant temperature} in two scenarios: when there is no diffusion in the temperature, {or} when {a} vertical diffusivity is added to the {temperature dynamics}. For both {scenarios}, our results indicate that the variation of temperature {affects neither} the formation of singularity, nor its stability, in the velocity field, respectively.
\end{abstract}
\maketitle

\section{Introduction}
For {planetary} large-scale {oceanic and atmospheric dynamics models}, the primitive equations provide a useful approximation to Navier-Stokes equations, {while} retaining many of its essential key features. The approximation is given through the hydrostatic balance that neglects the vertical acceleration of the fluid motion. Such an approximation/derivation was {first} rigorously justified {for the case of incompressible fluid} in \cite{Azrad2001} in the sense that the limit of a weak solution obtained from anisotropic Navier-Stokes equations agrees with the corresponding weak solution of the primitive equations, as the aspect ratio (ratio between typical vertical and horizontal scales in the atmosphere or the ocean) approaches zero. A few years ago, following a similar rescaling, Li and Titi \cite{Li2019} improved this result to the strong convergence {and provided explicit estimate on the rate of convergence in terms of the small aspect ratio. (See also \cite{Li2022,Liu2024} for the compressible case.)}\par
The mathematical foundation of primitive equations {was} first established by Lions, Temam and Wang in the {1990s} \cite{Lions1992}. In their ground-breaking paper, Lions et al. further proved the existence of weak solution for initial datum in a subset of $L^2$. Numerous endeavors have henceforth been made to improve their results in terms of well-posedness of such equations. Notably, Cao and Titi \cite{Cao2007} showed the existence, uniqueness and continuous dependence of global strong solution to the full viscous primitive equations with  temperature, for $H^1$ data satisfying the boundary conditions. \par

The well-posedness or ill-posedness of the primitive equations appears to depend heavily on {the presence of anisotropic viscosity}. \ifthenelse{\boolean{isSimplified}}{With full or only horizontal viscosity, well-posedness results have been established in \cite{Kukavica2007,Hieber2016,Cao2020}. }{With viscosity in both horizontal and vertical directions, Kukavica et al. \cite{Kukavica2007} showed the global existence and uniqueness of strong solutions for initial condition in a subset of $H^1$ with mixed Dirichlet and Neumann boundary conditions on the top and the bottom. Hieber and Kashiwabara \cite{Hieber2016} obtained a similar well-posedness result with initial condition in a subset of $L^p$ for $p\in [\frac{6}{5},\infty]$. Similar well-posedness results (\cite{Cao2020}) have also been achieved when the viscosity is only horizontal. On the contrary, with only vertical viscosity, the hydrostatic Navier-Stokes equations are known to be ill-posed in Sobolev Spaces (\cite{Renardy2009}).}The inviscid version of the primitive equations, also known as the hydrostatic Euler equations, appears also as an important model for ideal geophysical fluids. 
 {In the absence of viscosity, it is well-known that solutions of {the hydrostatic Euler equations experience a loss of regularity. {(See, for example, \cite{Jeong2021,Constantin1986,Roberta2024}.)}} Hence, establishing a local well-posedness results in the scale of Sobolev spaces for an arbitrary initial datum is usually deemed to be impossible for the hydrostatic {Euler equation}. {For the hydrostatic Euler equations,  Renardy \cite{Renardy2009} constructed a linearly unstable parallel shear flow and then showed that the equations are ill-posed, by extending linear instability to nonlinear instability. The ill-posedness of the hydrostatic Euler equation was later strengened by Han-Kwan and Nguyen \cite{HanKwan2016}, who constructed a family of solutions for which the flow map has unbounded {Hölder} norm in arbitrarily short time.}\par 

Most of the aforementioned literature focuses exclusively on the scenario where the temperature is isothermal. That is, the temperature is imposed and remains constant over time. While such assumption reduces the complexity of the system,  the variation of temperature often {plays a key role in influencing the} solution dynamics. For instance, in {the study of }atmospheric science, temperature gradient {is observed to induce instabilities}, which is essential to the formation of wind patterns, giving rise to the chaotic and turbulent behaviors of flow patterns. \par

{It is also worth mentioning literature that considers Euler equations coupled with {fluid density} {(also known as inhomogeneous Euler equations)}, as the resulting {partial differential equations} do mathematically share a certain degree of similarity. One of the most relatable results has been recently established by Bianchini et al. \cite{Roberta2024}: They showed that the 2D incompressible Euler-Boussinesq equations in a thin periodic channel, in general, do not converge to their limiting hydrostatic version. Therefore, the hydrostatic limit process cannot be rigorously justified. {One of the key ingredients towards the invalidity is the construction of a linearly unstable stationary initial condition $(U_s(z),0,\rho_s(z))$ that violates the Miles-Howard criterion \cite{Howard1961,Miles1961}: $\frac14\leq -\frac{\rho_s'}{\abs{U_s'}^2}$---a sufficient condition of linear stability of such equations. }

{Intuitively, if one views the density as temperature, then Miles-Howard Criterion suggests that the introduction of non-constant coupling term may have a stabilizing effect. On the other hand, numerous literatures (see below) have suggested that it is possible to find self-similar profiles that lead to singularity in finite time for the primitive equations. It is natural to wonder whether such singulairty may persist as the variation of the coupling term is added to the dynamics.}

Incorporating the gravitational potential into the pressure, when {the} temperature is taken into account, the normalized {three-dimensional} incompressible inviscid external-force-free primitive equations take the form of
\begin{equation}\label{primitive equation PDE only}
    \begin{cases}
        u_t+uu_X+vu_Y+wu_Z+p_X-\Omega v=0\\
        v_t+uv_X+vv_Y+wv_Z+p_Y+\Omega u=0\\
        p_Z+\theta=0\\
        \theta_t+u \theta_X+v\theta_Y+w\theta_Z-\sigma\theta_{ZZ}=0\\
        u_X+v_Y+w_Z=0,
    \end{cases}
\end{equation}
where the velocity field $\mathbf{u}=(u,v,w)$, the pressure $p$ and the temperature $\theta$ are $(X,Y,Z,t)$-dependent variables, while  the angular velocity of the earth $\Omega$, and the {vertical diffusivity} $\sigma$ are constants. {System \eqref{primitive equation PDE only}} is {typically} supplemented with initial conditions
\begin{equation}
    (u,v,\theta)(t=0)=(u_0,v_0,\theta_0),
\end{equation}
and relevant geophysical \cite{Kukavica2011} boundary conditions for the velocity field
\begin{equation}\label{boundary condition}
    \begin{cases}
        w(X,Y,Z=0,t)=w(X,Y,Z=1,t)=0\\
        \displaystyle\hat{n}\cdot\int_0^1 (u,v)(X,Y,Z,t)\, dZ=0
    \end{cases}
    \qquad\text{for all }(X,Y)\in \mathcal \partial \mathcal M\quad\text{and  }t\geq0,
\end{equation}
where the domain is defined to be
\begin{equation}\label{boundary condition domain}
\begin{split}
  &  \mathcal{D}=\mathcal{M}\times[0,1]=\set{(X,Y,Z):(X,Y)\in \mathcal{M},\, 0\leq Z\leq 1}\\
  &\mathcal{M} {\text{ bounded and Lipschitz},}
    \end{split}
\end{equation}
and $\hat n$ is the outward unit normal vector of $\partial \mathcal M$.\par

In case where a singularity does form, it is often natural to determine its characteristics, i.e., blowup speed, location, and profile, and also investigate its stability under small perturbation of  the initial conditions. 
{One of the earliest blowup results was established in \cite{Wong2014}: For 2D hydrostatic Euler equations {in the setting of a horizontally periodic channel}, Wong showed that there exists a class of initial conditions, which are not necessarily spatially symmetric, blows up in finite time. Regarding the existence of blowup profiles of hydrostatic equation, by symmetry reduction, Cao et al. \cite{Cao2015} proved that for the 3D isothermal Coriolis-force-free primitive equation, such profile indeed exists and the blowup result immediately follows. When rotation (Coriolis force) is taken into account, a similar result was established in \cite{Ibrahim2021}.
Regarding the stability and blowup speed, Collot et al. \cite{Collot2023} showed that for isothermal 3D primitive equations, subject to the boundary conditions \eqref{boundary condition}, there exists a family of initial conditions for which the corresponding solutions experience shocks on the $x$-axis in finite time. Notably, the blowup is stable under certain perturbation and the blowup speed is explicitly given.}


Here we follow similar steps. Under the assumptions of zero Coriolis force, particular solution-symmetry, and $Y$-independent initial data, {we will show in Section \ref{reduction}, below, that} the PDEs \eqref{primitive equation PDE only} can be reduced to an evolution equation on {vertical axis $[0,1]$ of the $Z$-axis}, when {the unknown variables are }restricted on the axis $X=0$: {(For simplicity, below, we denote $\partial_Z^{-1} f=\int_0^Z f(\tilde Z)d\tilde Z$.)}

{
\begin{subequations}\label{trace-system-introduction}
\begin{align}
    & a_t - a^2 + (\partial_Z^{-1}a) a_Z  + \partial_Z^{-1} c  + \int_0^1 (2a^2 - \partial_Z^{-1} c ) dZ = 0,
    \\
    & c_t - 2ac + (\partial_Z^{-1} a) c_Z-\sigma c_{ZZ} = 0,\qquad \sigma\in\set{0,1},
\end{align}
\end{subequations}
supplemented with boundary conditions 
}
{
\begin{equation}\label{boundary conditions-introduction}
    \left.\int_0^1 a(t,Z) dZ=0\qquad\text{and}\qquad \sigma c(t,Z=0)=\sigma c(t,Z=1)=0,\right.
\end{equation}}where  $a(Z,t)=-u_X(X=0,Z,t),$ $c(Z,t)=\theta_{XX}(X=0,Z,t),$ and we have added Dirichlet boundary conditions for the diffusive case. \par
This paper shall demonstrate that the variation of temperature, in general, {affects neither} the formation of {singularity} of the velocity field, nor its stability. More precisely, we shall provide the proofs of the following two theorems, respectively addressing the scenarios where temperature is governed by a transport or an advection-diffusion equation.

\begin{theorem}[{Non-diffusive Temperature ($\sigma=0$)}]\label{theorem:critical}
Let $\phi(x)=\exp(-x)$. There exists  $\lambda^{*}_0>0$, sufficiently small, such that for all $0<\lambda_0\leq {\lambda_0^*}$, there exists $\kappa(\lambda_0)>0$  such that the following holds. If the initial conditions {to System \eqref{trace-system-introduction} are of the form of}
\be \label{smooth:id:initial:non-diffusive}
a_0(Z)=\frac{1}{\lambda_0} \phi\left(\frac{Z}{\nu_0}\right)+\tilde a_0(Z),\, \quad {c_0}(Z) \ \ \text{on } 0\leq Z \leq 1,
\ee
{satisfying the zero-average condition and spatial scaling $\nu_0$ are in the range of}:
\begin{equation}
    \int_0^1 a_0(Z)=0,\quad \frac{1}{2\log (\lambda_0^{-1})}\leq\nu_0\leq\frac{3}{2\log (\lambda_0^{-1})},\quad
\end{equation}
and the perturbations $\tilde a_0$ and $c_0$ satisfy
\begin{equation}\label{boundary-conditions-initial-data-non-diffusive}
\begin{split}
&{{\partial_Z \tilde a_0(t,Z=0)=0},\,  c_0(Z=0)=\partial_Z  c_0(Z=0)=0,}\\
        &{\norm{\tilde a_{0}}}_{C^{1,\frac{3}{4}} ([0,1])}   + \left\| c_{0}\right\|_{C^{1,\frac{4}{5}} ([0,1])}   \leq \kappa,\quad {\tilde a_0(Z), c_0(Z) \text{ analytic on }[0,1]},
\end{split}
\end{equation}
then there exists $T>0$ such that the solution $(a,c)$ to \eqref{trace-system-introduction} and \eqref{boundary conditions-introduction} with $\sigma=0$ and initial data $(a,c)(t=0)=(a_0,c_0)$ blows up at time $T>0$ according to 
\be \label{smooth:1}
\begin{split}
    a(t,Z)&=\frac{1}{\lambda(t)} \left(\phi\left(\frac{Z}{\nu(t)}\right)+\tilde a(t,Z)\right) \qquad \mbox{with}\quad \lambda(t)\approx T-t,\quad \nu(t)\approx\frac{1}{|\log(T-t)|},\\
c(t,Z)&=\frac{1}{\lambda(t)}\tilde c(t,Z)
\end{split}
\ee
where for all $t\in [0,T)$:
\be \label{smooth:2}
\begin{split}
&  \tilde a(t,Z=0)={\partial_Z \tilde a(t,Z=0)=0},\quad \tilde c(t,Z=0)=\partial_Z \tilde c(t,Z=0)=0,\\
    &\|\tilde a(t,\cdot)\|_{L^{\infty}([0,1])}=O(  |\log (T-t)|^{-\frac{2}{3}}),\quad \|\tilde c(t,\cdot)\|_{L^{\infty}([0,1])}=O(   (T-t)^{\frac{1}{4}}).
\end{split}
\ee
\end{theorem}

{
\begin{remark}
    For non-diffusive temperature {($\sigma=0$)}, near blowup  time $T$, $a$ and $c$ can be further simplified asymptotically as
    \begin{equation}
    \begin{split}
         &a(t,Z)=(T-t)^{-1+Z}+O\left((T-t)^{-1}\abs{\log (T-t)}^{-\frac{1}{3}}\right)\\
         &c(t,Z)=\begin{cases}
             0 \quad&\text{ if }Z=0\\
             O((T-t)^{-\frac{3}{4}}) \quad &\text{ if }Z\in(0,1].\\
         \end{cases}
    \end{split}
    \end{equation}
\end{remark}
}

{This theorem indicates that, for non-diffusive temperature {($\sigma=0$)}, {the singularity forms as shock} at $Z=0$, at a speed precisely equal to $\frac{1}{T-t}$. Moreover, in the case where $c$ also blows up at locations near, but other than $Z=0$ ({since the estimate for $c$ on locations other than zero is only an upper bound, and thus whether it blows up is not known}), its blowup speed cannot be faster than that of $a$ at $Z=0$. This, however, is not the case of diffusive temperature.}

\begin{theorem}[Diffusive Temperature {($\sigma=1$)}]\label{theorem:critical:diffusive}
Let $\phi(x)=\exp(-x)$. There exists  $\lambda^{*}_0>0$, sufficiently small, such that for all $0<\lambda_0\leq {\lambda_0^*}$, there exists $\kappa(\lambda_0)>0$  such that the following holds. If the {initial conditions are of the form of }
\be \label{smooth:id:initial}
a_0(Z)=\frac{1}{\lambda_0} \phi\left(\frac{Z}{\nu_0}\right)+\tilde a_0(Z),\, \quad c_0(Z)\text{ on }  \ \ 0\leq Z \leq 1,
\ee
{satisfying} the zero-average condition and spatial scaling {$\nu_0$} are in the range of:
\begin{equation}
    \int_0^1 a_0(Z)=0,\quad \frac{1}{2\log (\lambda_0^{-1})}\leq\nu_0\leq\frac{3}{2\log (\lambda_0^{-1})},\quad
\end{equation}
and the perturbations $\tilde a_0$ and $\tilde c_0$ satisfy
\begin{equation}\label{boundary-conditions-initial-data-diffusive}
\begin{split}
    &{{\partial_Z\tilde a_0(Z=0)=0},\,c_0(Z=0)=c_0(Z=1)=0,}\\
     &{\norm{\tilde a_{0}}}_{C^{1,\frac{3}{4}} ([0,1])}   + \left\| c_{0}\right\|_{C^{0,\frac{31}{32}} ([0,1])}   \leq \kappa,\quad {\tilde a_0(Z),c_0(Z) \text{ analytic on }[0,1]},
\end{split}
\end{equation}
then there exists $T>0$ such that the solution $(a,c)$ to \eqref{trace-system-introduction} and \eqref{boundary conditions-introduction}  with $\sigma=1$ and initial data $(a,c)(t=0)=(a_0,c_0)$ blows up at time $T>0$ according to
\be \label{smooth:1}
\begin{split}
    a(t,Z)&=\frac{1}{\lambda(t)} \left(\phi\left(\frac{Z}{\nu(t)}\right)+\tilde a(t,Z)\right), \qquad \mbox{with}\quad \lambda(t)\approx T-t,\quad \nu(t)\approx\frac{1}{|\log(T-t)|},\\
c(t,Z)&=\frac{1}{\lambda(t)^2}\tilde c(t,Z),
\end{split}
\ee
where for all $t\in [0,T)$:
\be \label{smooth:2}
\begin{split}
&  \tilde a(t,Z=0)={\partial_Z \tilde a(t,Z=0)=0},\quad\tilde c(t,Z=0)=\tilde c(t,Z=1)=0,\\
    &\|\tilde a(t,\cdot)\|_{L^{\infty}([0,1])}=O(  |\log (T-t)|^{-\frac{3}{4}}),\quad \|\tilde c(t,\cdot)\|_{L^{\infty}([0,1])}=O(   (T-t)^{\frac{3}{16}}|\log (T-t)|^{\frac{3}{4}}).
\end{split}
\ee
\end{theorem}
{
\begin{remark}
    For {diffusive} temperature {($\sigma=1$)}, near {the} blowup time $T$, $a$ and $c$ can be further simplified asymptotically as
    \begin{equation}
    \begin{split}
        &a(t,Z)=(T-t)^{-1+Z}+O\left((T-t)^{-1}\abs{\log (T-t)}^{-\frac{1}{2}}\right)\\
         &c(t,Z)=\begin{cases}
         0\quad &\text{ if } Z=0,{1}\\
             O((T-t)^{-\frac{29}{16}}|\log (T-t)|^{\frac{3}{4}})\quad&\text{ if }Z\in{(0,1)}.\\
         \end{cases}
         \end{split}
    \end{equation}
\end{remark}

{
\begin{remark}
    Notice that some of the exponents that appear in both theorems are not unique. In fact, we have determined a range of these exponents for which both theorems hold, corresponding to the non-diffusive (Subsection \ref{choice of parameters non-diffusive}) and the diffusive case (Subsection \ref{choice of parameters diffusive}).
\end{remark}
}
}

{For diffusive temperature, similarly, a shock forms for the velocity field at the same speed at $Z=0$. However, for locations { not equal to zero}, if $c$ blows up as well, the blow-up speed can no longer be comparable to the speed of $a$ at $Z=0$, as we have chosen a different rescaling for $c$. {Regardless}, as there is no universal rule to compare {$a$ and $c$}, we conclude that, near $Z=0$, the blow-up speed of $\abs{c}^\frac{1}{2}$ cannot be faster than that of $a$ at $Z=0$.}

Concerning the stability of the blowup, for both types of temperatures, we observe that if the initial condition is sufficiently large (and hence sufficiently close to blowup), then blowup is stable under a particular way of perturbation. That is, any analytic small perturbations, with vanishing value and derivative at $Z=0$, for which the initial condition still satisfies {the respective boundary conditions. (\eqref{boundary-conditions-initial-data-non-diffusive} or \eqref{boundary-conditions-initial-data-diffusive}})

{The explicit profile $\phi=\exp(-x)$ we adopt in Theorem \ref{theorem:critical} and Theorem \ref{theorem:critical:diffusive} traces its origin back to \cite{Collot2023}: the study of isothermal case. By imposing a self-similar Ansatz at the blowup time $T$: $a(t,Z)\approx\frac{1}{T-t}\phi_\beta (\frac{Z}{(T-t)^\beta})$ {(Each $\phi_\beta$  is a self-similar profile and $\beta$ is related to its regularity on $\mathds{R}^+$.)} and neglecting the lower order term appearing in \eqref{trace-system-introduction}, they derived and analyzed the ODEs of the profile equation in self-similar coordinates $z=\frac{Z}{(T-t)^\beta}$, for $\beta\geq 0$. For $\beta=0$, the profile can be explicitly solved as $\exp(-x)$, while for $\beta>0$, the profiles are not smooth and do not have an explicit form. This effectively introduces {more computations} later on for the non-smooth case, while both types of profiles achieve the same desired blowup results in the end. To avoid unnecessary and tedious computations, especially when the temperature is no longer constant, we {shift} our attention entirely to the smooth {case}. We observe that the ``profile'' we refer to does not satisfy the definition of ``self-similar profile'' in the traditional sense. In fact, it is actually an ``approximate profile'', as there is still a remainder term from the solution, which takes into account the previously ignored {terms}.}\par

By imposing {the following} self-similar Ansatz on the solution structure

\begin{equation}\label{decomposition-introduction}
    \begin{split}
        a(s,z)&=a(t,Z)=\frac{1}{\lambda(s(t))}\left(\phi\left(\frac{Z}{\nu(s(t))}\right)+\tilde{a}\left(s(t),\frac{Z}{\nu(s(t))}\right)\right)=\frac{1}{\lambda(s)}(\phi(z)+\tilde{a}(s,z)),\\
        c(s,z)&=c(t,Z)=\frac{1}{\lambda(s)^{1+\sigma}}\tilde{c}(s,z),\,{(\text{Recall that } \sigma\in\set{0,1} \text{ is the diffusivity.})}
    \end{split}
\end{equation}
as well as introducing a new time and spatial-rescaling 
\begin{equation}
     z=\frac{Z}{\nu(t)},\,\frac{ds}{dt}=\frac{1}{\lambda(t)},\,s(0)=s_0,\, \quad \nu,\lambda>0\in C^1,\,\quad z\in[0,\nu^{-1}(s)],
\end{equation}
the PDEs {governing} the perturbation {terms} $(\tilde a,\tilde c)$ and the modulation parameters $(\lambda,\nu)$ are readily derived in new coordinate system. {The strategies we shall deploy are outlined as follows. To begin with, we introduce a tool to measure the significance of the perturbation with respect to the profile. The notion of initial closeness and trapped solution \cite{Collot2023}, depends on the partition of the whole domain and requires some weighted norms of the perturbations to be small and comparable to the self-similar time $s$ on some interval $[s_0,s_1]$. We shall show that for any initial perturbation that satisfies the definition of initial closeness, remains trapped for all time $[s_0,\infty)$, as long as the framework parameters in the definitions above are chosen correctly. As directly establishing such result can be quite challenging, a bootstrap argument is {then} implemented. We shall find a decomposition of the spatial domain into an interior and exterior regions, such that energy method is applied in the interior to obtain differential inequalities, while maximum principle is applied to obtain better upper bounds in the exterior. The last step involves establishing a link between conclusions made in {the self-similar coordinates and  the original physical coordinates}. We shall show that initial conditions in Theorem \ref{theorem:critical} and Theorem \ref{theorem:critical:diffusive} can be re-decomposed such that the resulting initial perturbation in self-similar coordinates satisfies the definition of initial closeness and thus remains trapped for all time. Finally, both theorems are concluded by unwinding or translating back the results into {the original} physical coordinates.}\par

In essence, this paper can be considered as a continuation or generalization of the results in \cite{Collot2023} to non-isothermal temperatures. Nevertheless, when variation of temperature is taken into account, the methodology needs to be modified accordingly, especially when the temperatures {are either advective or diffusive}. In our case, this difference is first reflected in the choices of scaling of the temperature. For non-diffusive temperature {($\sigma=0$)}, the fact that the temperature is governed by a transport equation, whose speed is a function of velocity field, suggests that the same scaling would be appropriate, as we expect {that the variable} $a$ to blow up in finite time, whereas for non-diffusive temperature {($\sigma=1$)}, we choose the scaling of heat equation.\par

The other major difference, which directly affects the proof strategies {for the two different cases of }temperatures, is manifested in the bootstrap arguments. As some of the weighted integral norms may involve singularity at $z=0$, this effectively restricts the range of initial perturbations into a family of functions that vanish sufficiently fast compared to the weight. For these weighted-norms to remain well-defined for all time, {one needs to first verify whether certain vanishing-speed conditions at $z=0$ are time-invariant to the PDEs of perturbations}. Moreover, there are a few instances when we integrate by parts and require some quantities to vanish. For non-diffusive temperature, as the temperature perturbation $\tilde c_z(s,z)=o(z^{1-\epsilon})$ (for any $\epsilon>0$) can be imposed as a time-invariant {condition at $z=0$}, results in a substantial simplification of the methodology that shares similarity with \cite{Collot2023}. On the other hand, for diffusive temperature, we are only able to show the temperature perturbation $\tilde c(s,z)=o(z^{1-\epsilon})$ is time-invariant, and thus some weighted-norm needs to be modified. We replace the weighted $H^1$-norm by a weighted $L^{2\eta_0}$-norm (for some $\eta_0 \in \N$) and add $\eta_0$ into the list of framework parameters. Thanks to the weighted Hardy's inequality and Dirichlet boundary conditions, to deal with the higher order derivative $\tilde c_{zz}$, we apply the energy method in the whole region to conclude that the influence of this higher order term can be bounded above by zero. Once we find all the framework parameters to close the bootstrap argument, we take $\eta_0\leq\eta\rightarrow \infty$ to obtain the $L^\infty$ estimate.\par

This paper is structured as follows. In {Section \ref{reduction},} we provide a detailed reduction from the 3D inviscid primitive equation \eqref{primitive equation PDE only} to the system of our focus \eqref{trace-system-introduction}, as well as explain how the boundary conditions \eqref{boundary conditions-introduction} are derived. In the first three subsections of Section 3, we apply the self-similar Ansatz to the solution of \eqref{trace-system-introduction}, and derive the PDEs of perturbations, as well as the ODEs of modular parameters, for both diffusive and non-diffusive case. The last subsection corresponds to the study of time-invariant vanishing-speed-related boundary conditions, for which can be further imposed. Section 4 and Section 5 focus on the bootstrap arguments for non-diffusive and diffusive temperature, respectively, each of which contains a definition subsection, where we define the initial closeness and trapped solution with respect to framework parameters, subsections that involve the derivation of integral inequalities (energy method), a subsection that solves the inequalities, a subsection that studies the maximum principle, and a subsection that give the range of framework parameters for which the bootstrap argument works. Finally, in Section 6, we relate the conclusions we have arrived in self-similar coordinates to the original PDE {System} \eqref{trace-system-introduction}. In the first subsection, we show that for certain initial conditions, they can be re-decomposed in a different way such that they satisfy the definition of initial closeness. In the last subsection, we translate our results in physical time and derive the blowup speed.

\section{Derivation and Reduction}\label{reduction}
\ifthenelse{\boolean{isSimplified}}{{Following the footsteps of \cite{Cao2015,Ibrahim2021}}, we assume that the solutions of \eqref{primitive equation PDE only} and \eqref{boundary conditions-introduction} are smooth and do not depend on $Y$:
{
{
\begin{equation}
\resizebox{.95 \textwidth}{!}{$
    u(X,Y,Z,t=0)=u_0(X,Z),\,v(X,Y,Z,t=0)=v_0(X,Z)=0,\,\theta(X,Y,Z,t=0)=\theta_0(X,Z),\, w(X,Y,Z,t=0)=w_0(X,Z).
    $}
\end{equation}
}
}
Upon setting $\Omega=0$, we obtain the 2D inviscid primitive equations 

{
\begin{subequations}\label{PE-system-intro-2d}
\begin{align}
    &u_t + u\, u_X + w u_Z +p_X = 0 , \label{EQ2-1}
    \\
    &p_Z +\theta=0 ,   \label{EQ2-3}  
    \\
    &u_X + w_Z =0,   \label{EQ2-4}
    \\
    &\theta_t + u\theta_X + w\theta_Z-\sigma \theta_{ZZ} = 0, \label{EQ2-5}
\end{align}
\end{subequations}
where {$v(X,Y,Z,t)\equiv 0$ is the solution and hence can be excluded from the system above.}
}

We consider the domain to be 
{
\begin{equation}
    \mathcal U=\{(X,Z):X\in[-L,L], Z\in[0,1]\},
\end{equation}
}
and that {the boundary conditions from \eqref{boundary condition} become:
\begin{equation}\label{condition on w}
\begin{split}
 &w(X,Z=0,t)=w(X,Z=1,t)=0\\ 
 &\int_0^1 u(X=-L,Z,t) dZ=\int_0^1 u(X=L,Z,t) dZ=0.
\end{split}
\end{equation}
}

To further reduce the system, we impose symmetry conditions on the solution 
\begin{enumerate}
    \item {$u$ is odd in $X\implies$ $w$ is even in $X$},
    \item $\theta$ is even in $X$.
\end{enumerate}}{Consider the 3D inviscid primitive equations 
\begin{equation}
    \begin{cases}
        u_t+uu_x+vu_y+wu_Z+p_x-\Omega v=0\\
        v_t+uv_x+vv_y+wv_Z+p_y+\Omega u=0\\
        p_Z+\theta=0\\
        \theta_t+u \theta_x+v\theta_y+w\theta_Z-\sigma\theta_{ZZ}=0\\
        u_x+v_y+w_Z=0,
    \end{cases}
\end{equation}
where $\sigma\in\set{0,1}$, corresponding to the cases where there the temperature $\theta$ is non-diffusive or not.
Assuming the initial conditions and solutions are smooth and do not depend on $y$:
\begin{equation}
    u(x,y,Z,t=0)=u_0(x,Z),\,v(x,y,Z,t=0)=v_0(x,Z),\,\theta(x,y,Z,t=0)=\theta_0(x,Z),
\end{equation}
we obtain the 2D inviscid PEs 
\begin{subequations}\label{PE-system-intro-2d}
\begin{align}
    &u_t + u\, u_X + w u_Z +p_X -\Omega v= 0 , \label{EQ2-1}
    \\
    &v_t + uv_X + wv_Z + \Omega u = 0 \label{EQ2-2}
    \\
    &p_Z +\theta=0 ,   \label{EQ2-3}  
    \\
    &u_X + w_Z =0,   \label{EQ2-4}
    \\
    &\theta_t + u\theta_X + w\theta_Z-\sigma \theta_{ZZ} = 0, \label{EQ2-5}
\end{align}
\end{subequations}
where the domain satisfies that $X\in[-L,L]$ is periodic and $Z\in[0,1]$. To further reduce the system, we consider the cases where 
\begin{enumerate}
    \item { $u$ is odd in $X\implies$ $w$ is even in $X$},
    \item $\theta$ is even in $x$,
    \item $w(X,Z=0,t)=w(X,Z=1,t)=0$,
    \item $\Omega=0$.
\end{enumerate}
}
{Notice that such symmetry conditions are invariant under the dynamics of system \eqref{PE-system-intro-2d}.}
{Combining \eqref{EQ2-4} and \eqref{condition on w},} we immediately arrive at the compatibility condition
\begin{eqnarray}\label{compatibility-preliminary}
&&\hskip-.38in
\frac d{dX}\int_0^1 udZ=\int_0^1 u_X(X,Z)dZ=0. \label{com}
\end{eqnarray}
The pressure term $p$ in \eqref{PE-system-intro-2d} can be explicitly solved using $u,v$ and $\theta$. Indeed, we first notice that  since $u$ is odd, so {$\int_0^1 u(X=0,Z)dZ = 0$}. Next, Integrating (\ref{EQ2-1}) with respect to $Z$ over $(0,1)$ and by integration by parts, one has:
\ifthenelse{\boolean{isSimplified}}{\begin{eqnarray}
	&&\hskip-.58in
 \int_0^1 \Big((u^2)_X(X,Z) + p_X(X,Z)\Big) dZ = 0 \implies \int_0^1 p_X(X,Z)dZ =  - \int_0^1 2uu_X(X,Z) dZ.\label{P-1}
\end{eqnarray}}
{\begin{eqnarray}
	&&\hskip-.58in
 \int_0^1 \Big((u^2)_x(X,Z) + p_X(X,Z)\Big) dZ = 0\label{c-1}
\end{eqnarray}
Thus,
\begin{eqnarray}
&&\hskip-.58in
\int_0^1 p_X(X,Z)dZ =  - \int_0^1 2uu_X(X,Z) dZ. \label{P-1}
\end{eqnarray}}
Next, from \eqref{EQ2-3}, we have 
\begin{eqnarray}
&&\hskip-.58in
p(X,Z)=p_s(X)-\int_0^Z\theta(X,s)ds,  \label{P-2}
\end{eqnarray}
where $p_s(X) = p(X,0)$ is the pressure at $Z=0$. By differentiating (\ref{P-2}) with respect to $X$, and integrating respect to $Z$ over $(0,1)$, by virtue of (\ref{P-1}), we have
\begin{eqnarray}
&&\hskip-.8in
(p_s)_{X}(X)=\int_0^1 \Big[\int_0^{Z} \theta_X(X,\tilde Z)d\tilde Z-2uu_X(X,Z)\Big]dZ. \label{P-3}
\end{eqnarray}
Therefore, by differentiating (\ref{P-2}) with respect to $X$, and using (\ref{P-3}), we obtain
\begin{eqnarray}\label{expression-of-pressure}
&&\hskip-.8in
p_{X}(X,Z)= -\int_0^Z\theta_X(X,\tilde Z)d\tilde Z 
+\int_0^1 \Big[\int_0^{Z} \theta_X(X,\tilde Z)d\tilde Z  -2uu_X(X,Z)\Big]dZ. \label{px}
\end{eqnarray}

{Taking the $X$ derivative of \eqref{EQ2-1} and taking twice $X$ derivative of $\eqref{EQ2-5}$, we have
\begin{subequations}\label{equation of u-v-theta}
\begin{align}
    & u_{Xt} + uu_{XX} + u_X^2 + wu_{XZ} + w_X u_Z + p_{XX} = 0,
    \\
    & \theta_{XXt} + u_{XX} \theta_X + 2u_X \theta_{XX} + u\theta_{XXX} + w_{XX} \theta_Z + 2w_X \theta_{XZ} + w\theta_{XXZ}-\sigma \theta_{XXZZ}  = 0.
\end{align}
\end{subequations}
}
Setting $X=0$ in the last equation of \eqref{PE-system-intro-2d}, we obtain
\begin{equation}\label{theta at Zero}
    \theta_t(X=0,Z,t)+w(X=0,Z,t) \theta_Z(X=0,Z,t) -\sigma\theta_{ZZ}(X=0,Z,t)=0.
\end{equation}\par
To convert the system that can be solely expressed by the variables $u_X\rvert_{X=0},v_X\rvert_{X=0}$ and $\theta_{XX}\rvert_{X=0}$, we now need to impose some {specific initial conditions onto $\theta\rvert_{X=0}$ and boundary conditions} for the non-diffusive and diffusive cases, respectively. For non-diffusive case ($\sigma=0$), if we assume that 
\begin{equation}\label{vanishing theta initial}
  \text  {$\theta_0 = c$ is a constant on $X=0$,}
\end{equation} then since on $X=0$ we have \ifthenelse{\boolean{isSimplified}}{$
 \partial_t \theta(X=0) =- w(X=0) \partial_Z \theta(X=0),
$}{$$
 \partial_t \theta(x=0) =- w(x=0) \partial_Z \theta(x=0),
$$} $\theta=c$ remains a constant on $X=0$ at any later time, i.e., the condition $\partial_Z \theta(X=0)=0$ is invariant in time. If $\sigma=1$, we need to further impose the Dirichlet boundary condition
\begin{equation}
    \theta(X,Z=0,t)=\theta(X,Z=1,t)=0 \qquad\forall X\in[-L,L],\, t\geq0,
\end{equation}
to account for the extra derivative. \ifthenelse{\boolean{isSimplified}}{Since $\theta\rvert_{X=0}$ satisfies the PDE
\begin{equation*}
    \theta_t(X=0,Z,t)+w(X=0,Z,t) \theta_Z(X=0,Z,t) -\theta_{ZZ}(X=0,Z,t)=0,
\end{equation*}}{We look for solutions that satisfies
\begin{equation}
    \theta(x,Z,t)=\theta(x,Z,t)\equiv0 \qquad\forall x\in[-L,L],\, t\geq0,\, Z\geq 0.
\end{equation}
This, in particular, implies that
\begin{equation}\label{boundary-condition-for-diffusive-case}
    \theta\rvert_{x=0} (Z=0,t)=\theta\rvert_{x=0} (Z=1,t)=0.
\end{equation}
We recall that at $x=0$, $\theta\rvert_{x=0}$ satisfies
\begin{equation*}
    \theta_t(x=0,Z,t)+w(x=0,Z,t) \theta_Z(x=0,Z,t) -\theta_{ZZ}(x=0,Z,t)=0.
\end{equation*}} 
by uniqueness of solution with  \ifthenelse{\boolean{isSimplified}}{Dirichlet boundary condition $\theta|_{X=0}(Z=0,t)=\theta|_{X=0}(Z=1,t)=0$}{boundary condition \eqref{boundary-condition-for-diffusive-case}},
 \ifthenelse{\boolean{isSimplified}}{\begin{equation}\label{vanishing theta}
     \theta\rvert_{X=0}(Z,t=0)=0\implies \theta\rvert_{X=0}(Z,t),\,\theta_Z\rvert_{X=0}(Z,t)\equiv 0\text{ for all time.}
 \end{equation}
 }{\begin{equation}\label{vanishing theta}
    \theta\rvert_{X=0}(Z,t)\equiv 0\quad \forall Z\in[0,1],\, t\geq0,
\end{equation}
as well as 
\begin{equation}\label{vanishing theta Z}
    \theta_Z\rvert_{X=0}(Z,t)\equiv 0\quad \forall Z\in[0,1],\, t\geq0
\end{equation}
are time-invariant.}\par

 \ifthenelse{\boolean{isSimplified}}{ For both cases, by introducing the traces at $x=0$:
{
\begin{equation}
    \begin{split}
        a(Z,t)&=-u_X(X=0,Z,t)\\
        c(Z,t)&=\theta_{XX}(X=0,Z,t),
    \end{split}
\end{equation}
}and combining the boundary/initial conditions for non-diffusive \eqref{vanishing theta initial} and diffusive case \eqref{vanishing theta}, expression of pressure \eqref{expression-of-pressure}, oddness/evenness of {$u,v$,} \eqref{equation of u-v-theta} becomes {($\partial_Z^{-1}f:=\int_0^Z f(\tilde Z)d\tilde Z$ below)}}{For both cases, by introducing the traces at $X=0$:
\begin{equation}
    \begin{split}
        a(Z,t)&=-u_X(X=0,Z,t)\\
        b(Z,t)&=-v_X(X=0,Z,t)\\
        c(Z,t)&=\theta_{XX}(X=0,Z,t),
    \end{split}
\end{equation} and combining the boundary/initial conditions for non-diffusive \eqref{vanishing theta initial} and diffusive case \eqref{vanishing theta}\eqref{vanishing theta Z}, expression of pressure \eqref{expression-of-pressure}, and oddness/evenness of $u,v,w$, we obtain
\begin{align*}
    & a_t - a^2 + \partial_Z^{-1}a a_Z  + \partial_Z^{-1} c  + \int_0^1 (2a^2 - \partial_Z^{-1} c ) dZ = 0,
    \\
    & b_t - ab + \partial_Z^{-1}a b_Z  = 0,
    \\
    & c_t - 2ac + \partial_Z^{-1} a c_Z-\sigma c_{ZZ} = 0,
\end{align*}
and we notice that the system of equations of $a$ and $c$ are independent of $b$ and thus from now on, we shall only focus on}
{
\begin{subequations}\label{trace-system}
\begin{align}
    & a_t - a^2 +( \partial_Z^{-1}a) a_Z  + \partial_Z^{-1} c  + \int_0^1 (2a^2 - \partial_Z^{-1} c ) dZ = 0,
    \\
    & c_t - 2ac + (\partial_Z^{-1} a) c_Z-\sigma c_{ZZ} = 0,
\end{align}
\end{subequations}
}
and the {compatibility condition \eqref{compatibility-preliminary} now} takes the form of
\begin{equation}\label{compatibility condition}
    \int_0^1 a(t,Z) dZ=0.
\end{equation}
For {the} diffusive case ($\sigma=1$), \ifthenelse{\boolean{isSimplified}}{in addition, we also have the Dirichlet boundary conditions
\begin{equation}\label{raw boundary condition of c-nc}
    c(t,Z=0)=0=c(t,Z=1)\quad t\geq 0.
\end{equation}
} {since $\theta(x,Z=0,t)=\theta(x,Z=1,t)=0$, the boundary condition of $c$ becomes 
\begin{equation}\label{raw boundary condition of c-nc}
    c(t,Z=0)=\theta_{xx}\rvert_{x=0}(Z=1,t)=0=c(t,Z=1)=\theta_{xx}\rvert_{x=0}(Z=1,t)\quad t\geq 0.
\end{equation}
}

\section{Temporal and Spatial Rescaling---Self-similar Ansatz}
\ifthenelse{\boolean{isSimplified}}{}{
\subsection{Blowup Profile}
In \cite{Collot2023}, the authors have conducted formal and rigorous analysis on the  the hydrostatic Euler equations:
\begin{equation*}
\begin{split}
     &a_t - a^2 + \partial_Z^{-1}a a_Z +\int_0^1 2a^2 dZ = 0\\
     &\int_0^1 a(t,Z) dZ=0.
\end{split}
\end{equation*}
while assuming and formally arguing $\int_0^1 2a^2 dZ$ has lower order:
\begin{equation*}
    a_t - a^2 + \partial_Z^{-1}a a_Z\approx 0,
\end{equation*}
and $a$ almost blows up at some time $T$:
\begin{equation*}
    a(t,Z)\approx \frac{1}{T-t}\phi_\beta(\frac{Z}{(T-t)^\beta}),\quad z\approx\frac{Z}{(T-t)^\beta},\quad \beta\geq 0,
\end{equation*}
they were able to show obtain the blowup profile $\phi$ in self-similar coordinate $z$ explicitly when $\beta=0$:
\begin{equation}
    \phi_0(z)=\exp(-z).
\end{equation}
Following their footsteps, we adopt this profile $\phi_0$ but choose different temporal rescalings depending on whether the temperature $c$ is diffusive, while keeping the spatial rescaling the same.}
\subsection{System of Equations and {Vanishing-Speed Condition at Origin for Non-diffusive Case ($\sigma=0$)} }
For non-diffusive case {($\sigma=0$)}, the equation of temperature is a transport equation of {the variable $c$}. This suggests that the blowup is purely driven by {the variable} $a$, and we thus choose the scaling of $c$ to be identical to {that of} $a$. \par
Consider the {rescaling for the spatial variable $Z$ by $\nu$ and for the time by $\lambda$,} two positive $C^1$ functions of time:
\begin{equation} \label{id:self-similarvariables}
    z = \frac{Z}{\nu(t)}, \quad \frac{ds}{dt} = \frac{1}{\lambda(t)}, \quad s(0)=s_0.
\end{equation}
The following computations are done for $\phi(Z)=e^{-Z}$. We write the solution {$a(t,Z)$, $c(t,Z)$} of system \eqref{trace-system} as
\begin{equation}\label{decomp}
\begin{split}
    &a(t,Z) = \frac{1}{\lambda(s(t))}\Big(\phi\Big(\frac Z{\nu(s(t))}\Big) + \tilde{a}\Big(s(t),\frac Z{\nu(s(t))}\Big)\Big)= \frac{1}{\lambda(s)}\Big(\phi (z) + \tilde{a} (s,z)\Big),
    \\
    &c(t,Z) = \frac{1}{\lambda(s(t))}\tilde{c}\Big(s(t),\frac Z{\nu(s(t))}\Big).
\end{split}
\end{equation}
{An explicit computation gives ($\partial_z^{-1}f(s,z):=\int_0^z f(s,\tilde z) d\tilde z$)}:
\begin{equation*}
\begin{split}
    &a_t = \frac{1}{\lambda^2} \Big( -\frac{\lambda_s}{\lambda}\phi - \phi' \frac{\nu_s}{\nu} z - \frac{\lambda_s}{\lambda} \tilde{a} + \tilde{a}_s - \tilde{a}_z \frac{\nu_s}{\nu} z\Big),
    \\
    &a^2 = \frac{1}{\lambda^2} \Big( \phi^2 + \tilde{a}^2 + 2\phi \tilde{a} \Big),
    \\
    &{\Big(\int_0^Z a^2(t,\tilde{Z})d\tilde{Z}\Big) a_Z = \frac{1}{\lambda^2} \Big( (\partial_z^{-1} \phi) \phi' + (\partial_z^{-1} \phi) \tilde{a}_z + (\partial_z^{-1} \tilde{a}) \phi' + (\partial_z^{-1} \tilde{a}) \tilde{a}_z \Big),}
    \\
    &\int_0^Z c(t,\tilde{Z})d\tilde{Z} = \frac\nu\lambda \partial_z^{-1} \tilde{c},
    \\
    & \int_0^1 \left(2a^2(Z)  - \int_0^Z c(t,\tilde{Z})d\tilde{Z}\right) dZ =  \int_0^{\frac{1}{\nu}}\left(\frac{2\nu}{\lambda^2} ( \phi+ \tilde{a})^2(z)   - \frac{\nu^2}\lambda \partial_z^{-1} \tilde{c}(z)\right)dz.
\end{split}
\end{equation*}
{All the other components appearing in \eqref{trace-system} can be computed in a similar manner and are thus omitted. Combining the previous computation with the identity $(\partial_{z}^{-1}\phi) \phi'-\phi^2+\phi=0$, we arrive at the evolution equations of $\tilde a$ and $\tilde c$.}

\begin{subequations}\label{equation:epsilon}
\begin{align}
    \begin{split}
        &\tilde{a}_s - \frac{\lambda_s}{\lambda}\tilde{a} - \frac{\nu_s}{\nu} z \tilde{a}_z - 2\phi \tilde{a} + (\partial_z^{-1} \phi) \tilde{a}_z + (\partial_z^{-1} \tilde{a}) \phi'  - \tilde{a}^2 + (\partial_z^{-1} \tilde{a}) \tilde{a}_z \label{equation:epsilon-a}
        \\
        &= (\frac{\lambda_s}{\lambda} +1)\phi +  \frac{\nu_s}{\nu}z\phi' - 2\nu \int_0^{\frac{1}{\nu}} (\phi+ \tilde{a})^2(z) dz
        \\
         &\qquad  + \lambda\left(- \nu\partial_z^{-1} \tilde{c}  + \int_0^{\frac{1}{\nu}} \nu^2 \partial_z^{-1} \tilde{c}(z) dz\right),
    \end{split}
    \\
    \begin{split}\label{equation:epsilon-c}
     &\tilde{c}_s - \frac{\lambda_s}{\lambda}\tilde{c} - \frac{\nu_s}{\nu} z \tilde{c}_z   - 2\tilde{a}\tilde{c} - 2\phi\tilde{c} + (\partial_z^{-1} \tilde{a}) \tilde{c}_z + (\partial_z^{-1} \phi) \tilde{c}_z
        = 0,
    \end{split}
    \\
    &\int_0^{\frac1\nu} (\phi + \tilde{a})(z) dz = 0. \label{equation:epsilon-bc}
\end{align}
\end{subequations}
By taking the $z$ derivative of $\tilde a$ we compute
\begin{subequations}\label{derivative of tilde a}
\begin{align}
    \begin{split}
        &\tilde{a}_{zs} - \frac{\lambda_s}{\lambda}\tilde{a}_z - \frac{\nu_s}{\nu} \tilde a_z - \frac{\nu_s}{\nu} z \tilde{a}_{zz} -\phi' \tilde{a} - \phi \tilde{a}_z + (\partial_z^{-1} \phi) \tilde{a}_{zz} + (\partial_z^{-1} \tilde{a}) \phi''  - \tilde{a} \tilde a_z + (\partial_z^{-1} \tilde{a}) \tilde{a}_{zz} 
        \\
        &= (\frac{\lambda_s}{\lambda} +1)\phi' +  \frac{\nu_s}{\nu}\phi' +  \frac{\nu_s}{\nu}z\phi''  - \lambda\nu \tilde c.
    \end{split}
\end{align}
\end{subequations}

The modulation parameters $\lambda$ and $\nu$ are determined by imposing the following vanishing {coefficients in the Taylor expansion} for the expansion of $\tilde a$, an orthogonality-like condition:

\be \label{smooth:orthogonality}
\begin{cases}
     \tilde a(s,z=0)=0 \\
 \pa_z \tilde a(s,z=0) =0
\end{cases}
\iff \tilde a(s,z)=o(z) \text{ as }z\rightarrow 0.
\ee
{
Combining the first condition of \eqref{smooth:orthogonality} and \eqref{equation:epsilon-a} gives 
\be
\begin{split}
    &\frac{\lambda_s}\lambda +1 = 2\nu \int_0^{\frac{1}{\nu}} (\phi+ \tilde{a})^2(z) dz - \lambda \int_0^{\frac{1}{\nu}} \nu^2 \partial_z^{-1} \tilde{c}(z) dz,
\end{split}
\ee
while using both conditions of \eqref{smooth:orthogonality} and \eqref{derivative of tilde a} gives
\begin{equation}\label{lambda-and-nu-relation}
    -(\frac{\lambda_s}{\lambda}+1)-\frac{\nu_s}{\nu}-\lambda \nu \tilde{c}(z=0)=0.
\end{equation}
}

{Next, }setting $z=0$ in {\eqref{equation:epsilon-c}} we obtain that 
\begin{equation}\label{equations-at-zero}
    \begin{split}
        &\tilde{c}_s\rvert_{z=0}-\frac{\lambda_s}{\lambda}\tilde{c}\rvert_{z=0}-2\tilde{c}\rvert_{z=0}=0,
    \end{split}
\end{equation}
and it implies that
\begin{equation}\label{sigma-zero}
    \tilde c(s,z=0)=0 \text{ is time-invariant if }\sigma=0.
\end{equation}
Combining \eqref{equations-at-zero}, \eqref{sigma-zero} and \eqref{lambda-and-nu-relation}, the conditions on $\lambda$ and $\nu$ are reduced into
\begin{subequations}\label{smoothmodulationequations:initial}
    \begin{align}
        &\frac{\lambda_s}\lambda +1 = 2\nu \int_0^{\frac{1}{\nu}} (\phi+ \tilde{a})^2(z) dz - \lambda \int_0^{\frac{1}{\nu}} \nu^2 \partial_z^{-1} \tilde{c}(z) dz\label{condiiton-on-lambda}\\
        &\frac{\lambda_s}{\lambda}+\frac{\nu_s}{\nu}=-1\label{condiiton-on-lambda-nu}.
    \end{align}
\end{subequations}
{
Thanks to \eqref{smoothmodulationequations:initial} and the fact that $\phi' = -\phi$, $\phi''=\phi$, the equation of $\tilde a_z$ is simplified to
\begin{subequations}\label{equations:dz}
\begin{align}
    \begin{split}\label{equation:az}
        &\tilde{a}_{zs} + (1-\phi) \tilde a_z +(\partial_z^{-1} \phi - \frac{\nu_s}{\nu} z) \tilde{a}_{zz} +\phi \tilde{a}  + (\partial_z^{-1} \tilde{a}) \phi  - \tilde{a} \tilde a_z + (\partial_z^{-1} \tilde{a}) \tilde{a}_{zz} 
        \\
        &=  \frac{\nu_s}{\nu}z\phi - \lambda\nu \tilde c.
    \end{split}
\end{align}
\end{subequations}
}

\subsection{System of Equations and {Vanishing-Speed condition at Origin for Diffusive Case ($\sigma=1$)}}
For diffusive case, the temperature $c$ is of the type of Convection-diffusion equation and this suggests that the rescaling of heat equation would be appropriate for $c$, we thus introduce the rescalings
\begin{equation}
    z=\frac{Z}{\nu(t)},\,\frac{ds}{dt}=\frac{1}{\lambda(t)},\,s(0)=s_0,
\end{equation}
and {fix the same} approximating profile $\phi(z)=e^{-z}$. {With that being said, }we decompose
\begin{equation}\label{decomposition-nc}
    \begin{split}
        a(s,z)&=a(t,Z)=\frac{1}{\lambda(s(t))}\left(\phi(\frac{Z}{\nu(s(t))})+\Tilde{a}(s(t),\frac{Z}{\nu(s(t))})\right)=\frac{1}{\lambda(s)}(\phi(z)+\tilde{a}(s,z))\\
        c(s,z)&=c(t,Z)=\frac{1}{\lambda(s)^2}\tilde{c}(s,z).
    \end{split}
\end{equation}
Using {$(\antipartial_z \phi) \phi'-\phi^2+\phi=0$ (Similarly, we define $\partial_z^{-1}f(s,z):=\int_0^z f(s,\tilde z) d\tilde z$.)}, we compute
{
\begin{equation}\label{raw-nc}
    \begin{split}
        &\tilde{a}_s-\frac{\lambda_s}{\lambda}\tilde{a}-\frac{\nu_s}{\nu}z \tilde{a}_z-2\phi \tilde{a}+(\antipartial_z\phi)\tilde{a}_z+(\antipartial_z\tilde{a})\phi'-\tilde{a}^2+(\antipartial_z\tilde{a})\tilde{a}_z\\
        &=(\frac{\lambda_s}{\lambda}+1)\phi+\frac{\nu_s}{\nu}z\phi'-2\nu \int_0^{\nu^{-1}}(\phi+\tilde{a})^2(z)dz\\
        &\qquad+\lambda\left(-\nu \frac{1}{\lambda}\antipartial_z \tilde{c}+\int_0^{\nu^{-1}} \nu^2 \frac{1}{\lambda}\antipartial_z\tilde{c}(z)dz\right)\\
        &\tilde{c}_s-2\frac{\lambda_s}{\lambda}\tilde{c}-\frac{\nu_s}{\nu}z\tilde{c}_z-2\tilde{a}\tilde{c}-2\phi \tilde{c}+(\antipartial_z\tilde{a})\tilde{c}_z+(\antipartial_z\phi)\tilde{c}_z-\frac{\lambda}{\nu^2}\tilde{c}_{zz}=0.\\
        &\int_0^{\frac1\nu} (\phi + \tilde{a})(z) dz = 0.
    \end{split}
\end{equation}
}
{Imposing the same orthogonal-like boundary condition on $\tilde a$}:

\be \label{smooth:orthogonality-nc}
\begin{cases}
     \tilde a(s,z=0)=0 \\
 \pa_z \tilde a(s,z=0) =0
\end{cases}
\iff \tilde a(s,z)=o(z) \text{ as }z\rightarrow 0,
\ee
{similarly to the non-diffusive case, the conditions on $\lambda$ and $\nu$ are reduced into }
\begin{subequations}\label{smoothmodulationequations:initial-nc}
    \begin{align}
        &\frac{\lambda_s}\lambda +1 = 2\nu \int_0^{\frac{1}{\nu}} (\phi+ \tilde{a})^2(z) dz - \lambda  \int_0^{\frac{1}{\nu}} \frac{\nu^2}{\lambda} \partial_z^{-1} \tilde{c}(z) dz\label{condiiton-on-lambda}\\
        &\frac{\lambda_s}{\lambda}+\frac{\nu_s}{\nu}=-1\label{condiiton-on-lambda-nu-nc},
    \end{align}
\end{subequations}
{after computing the derivative of $\tilde a$ from \eqref{raw-nc}. Consequently, with the help of \eqref{smoothmodulationequations:initial-nc}, the equation of $\tilde a_z$ is simplified into 
{
\begin{subequations}\label{equations:dz-nc}
\begin{align}
    \begin{split}\label{equation:az-nc}
        &\tilde{a}_{zs} + (1-\phi) \tilde a_z +(\partial_z^{-1} \phi - \frac{\nu_s}{\nu} z) \tilde{a}_{zz} +\phi \tilde{a}  + (\partial_z^{-1} \tilde{a}) \phi  - \tilde{a} \tilde a_z + (\partial_z^{-1} \tilde{a}) \tilde{a}_{zz} 
        \\
        &=  \frac{\nu_s}{\nu}z\phi  - \nu \tilde c.
    \end{split}
\end{align}
\end{subequations}}}

What differs from the non-diffusive case is that
\begin{equation}\label{equations-at-zero-nc}
    \begin{split}
        &\tilde{c}_s\rvert_{z=0}-2\frac{\lambda_s}{\lambda}\tilde{c}\rvert_{z=0}-2\tilde{c}\rvert_{z=0}-\sigma\frac{\lambda}{\nu^2}\tilde{c}_{zz}\rvert_{z=0}=0,
    \end{split}
\end{equation}
{which is effectively obtained by taking the trace of \eqref{raw-nc}.}
{Since $\tilde c$ inherits the Dirichlet boundary condition of $c$ \eqref{raw boundary condition of c-nc} by the choice of rescaling \eqref{decomposition-nc}, } {\eqref{equations-at-zero-nc} implies} that the boundary condition we imposed earlier becomes
\begin{equation}\label{sigma-one-nc}
\begin{split}
    & \tilde c(s,z=0)=0=\tilde c(s,z={\frac{1}{\nu(s)}}),\,\tilde c_{zz}(s,z=0)=0\text{ for all $s\geq s_0$ }.
\end{split}
\end{equation}

\subsection{{More Invariant Vanishing Speed at Origin}}
It turns out that {at origin ($z=0$),}  we can impose more time-invariant {conditions, for which the variables $\tilde a$ and $\tilde c$} vanish strictly faster than the order of constant. This would be crucial to the well-definedness of some weighted energies, as we shall see later in our analysis. \par
\begin{corollary}[Invariance of Vanishing Speed at Zero]\label{Invariance-of-vanishing speed}
    Let $0<\epsilon_a\leq \epsilon_c<1$ and {$s\geq s_0\gg 1$}. Suppose that { $\tilde a\in C_s^1C_z^2([s_0,\infty];\R^+)$ and $\tilde c\in C_s^1C_z^3([s_0,\infty];\R^+)$}, and orthogonality-like condition on $\tilde a$: \ifthenelse{\boolean{isSimplified}}{$ \tilde a(s,0)=\partial_z\tilde a(s,0)=0 $}{\begin{equation}
        \tilde a(s,0)=\partial_z\tilde a(s,0)=0 
    \end{equation}}   
    are imposed for all time. \ifthenelse{\boolean{isSimplified}}{Assume also that {$\sigma\tilde c(s,0)=0=\sigma\tilde c(s,\nu^{-1}(s))$}}{For diffusive case $\sigma=1$, assume also the Dirichlet boundary condition on $\tilde c$: \begin{equation}
        \tilde c(s,0)=0=\tilde c(s,\nu^{-1}(s)) 
    \end{equation}
    } are imposed for each $s\geq s_0$. The following boundary conditions are time-invariant: 
    {
    \begin{enumerate}
        \item For diffusive case ($\sigma=1$), if initially $\tilde a_z(s_0,z)=o(z^{\epsilon_a}),\, \tilde c(s_0,z)=o(z^{\epsilon_c}) \text{ as }z\rightarrow 0^+$, then for each $s\geq s_0$,  \ifthenelse{\boolean{isSimplified}}{$\tilde a_z(s,z)=o(z^{\epsilon_a}),\, \tilde c(s,z)=o(z^{\epsilon_c}) \text{ as }z\rightarrow 0^+$.}{ \begin{equation}
            \begin{split}
                \tilde a_z(s,z)=o(z^{\epsilon_a}),\, \tilde c(s,z)=o(z^{\epsilon_c}) \text{ as }z\rightarrow 0^+.
            \end{split}
        \end{equation}}
        \item For non-diffusive case ($\sigma=0$), if initially $\tilde a_z(s_0,z)=o(z^{\epsilon_a}),\, \tilde c(s_0,z)=o(z^{\epsilon_c}),\, c_z(s_0,z)=o(z^{\epsilon_c}) \text{ as }z\rightarrow 0^+$, then for each $s\geq s_0$,   \ifthenelse{\boolean{isSimplified}}{$\tilde a_z(s,z)=o(z^{\epsilon_a}),\, \tilde c(s,z)=o(z^{\epsilon_c}) ,\, \tilde c_z(s,z)=o(z^{\epsilon_c})\text{ as }z\rightarrow 0^+.$}{\begin{equation}
            \begin{split}
                \tilde a_z(s,z)=o(z^{\epsilon_a}),\, \tilde c(s,z)=o(z^{\epsilon_c}) ,\, \tilde c_z(s,z)=o(z^{\epsilon_c})\text{ as }z\rightarrow 0^+.
            \end{split}
            \end{equation}}
    \end{enumerate}
    }
\end{corollary}
\begin{proof}
 \ifthenelse{\boolean{isSimplified}}{}{Suppose that $\tilde a\in C_s^1([s_0,\infty]) C_z^2(\R^+)$, and $\tilde c\in C_s^1([s_0,\infty]) C_z^3(\R^+)$, we rewrite equations \eqref{raw-nc} and \eqref{equation:epsilon-c} as
\begin{subequations}
    \begin{align*}
        &\tilde{c}_s=  ((\sigma+1)\frac{\lambda_s}{\lambda}+2\tilde a+2\phi)\tilde{c} + \frac{\nu_s}{\nu} z \tilde{c}_z    - \partial_z^{-1} \tilde{a} \tilde{c}_z - \partial_z^{-1} \phi \tilde{c}_z+\sigma\frac{\lambda}{\nu^2} \tilde{c}_{zz}.
    \end{align*}
\end{subequations}}
\noindent \textbf{Step 1}. \emph{Equations of $\tilde c$ near $z=0$. }Let $\kappa<1$, we divide equations \eqref{equation:epsilon-c} and \eqref{raw-nc} by $z^\kappa$ and {thus both cases can be written compactly: {For each $s\geq s_0$,}} 
\begin{equation}\label{vanishing-speed-analysis}
\begin{split}
    &\partial_s\left(\frac{\tilde{c}}{z^\kappa}\right)=  ((\sigma+1)\frac{\lambda_s}{\lambda}+2\tilde a+2\phi)\frac{\tilde{c}}{z^\kappa} + \frac{\nu_s}{\nu} z^{1-\kappa} \tilde{c}_z - \frac{\partial_z^{-1} \phi}{z^\kappa} \tilde{c}_{z}   - \frac{\partial_z^{-1} \tilde a}{z^\kappa} \tilde{c}_{z}+\sigma\frac{\lambda}{\nu^2} \frac{\tilde{c}_{zz}}{z^\kappa}.
    \end{split}
\end{equation}
Using L'hopital's rule and \eqref{smooth:orthogonality} and \eqref{smooth:orthogonality-nc}, we check
\begin{equation*}
    \frac{\partial_z^{-1} \tilde a}{z}\simeq {\tilde a}=o(z),\,\partial_z^{-1} \phi=1-e^{-z}\simeq z \implies  \frac{\partial_z^{-1} \tilde a}{z^\kappa}=o(1)= \frac{\partial_z^{-1} \phi}{z^\kappa} \text{ as }z\rightarrow 0^+\text{ if }{0\leq \kappa<1}.
\end{equation*}
Thus, equation \eqref{vanishing-speed-analysis} is simplified to
\begin{equation}\label{vanishing-speed-analysis-simplified}
    \begin{split}
    &\partial_s\left(\frac{\tilde{c}}{z^\kappa}\right)=  ((\sigma+1)\frac{\lambda_s}{\lambda}+2\tilde a+2\phi)\frac{\tilde{c}}{z^\kappa} +\sigma\frac{\lambda}{\nu^2} \frac{\tilde{c}_{zz}}{z^\kappa}+o(1), \text{ as }z\rightarrow 0^+\text{ if }{0\leq \kappa<1},
    \end{split}
\end{equation}
{for each $s\geq s_0$.}

\textbf{Step 2}. \emph{The vanishing speeds of $\tilde c$ for non-diffusive case. } {For $\sigma=0$, from \eqref{vanishing-speed-analysis-simplified}, we immediately conclude that $\tilde c=o(z^\kappa)$ is time-invariant for each {$0\leq \kappa<1$}.}
\ifthenelse{\boolean{isSimplified}}{Moreover, by differentiating the PDE of $\tilde c$ {\eqref{vanishing-speed-analysis-simplified}} and dividing both sides by $z^\kappa$, we see that {for each $s\geq s_0$,}
\begin{equation}
\begin{split}
      &\tilde{c}_{zs}=(\underbrace{\frac{\lambda_s}{\lambda}+\frac{\nu_s}{\nu}}_{-1}+\tilde a+\phi)\tilde{c}_z+(\frac{\nu_s}{\nu}z-\pa^{-1}_z\tilde{a}-\pa^{-1}_z{\phi})\tilde{c}_{zz}+(2\tilde{a}_z-2\phi)\tilde{c}\\
       &\partial_s(\frac{\tilde{c}_{z}}{z^\kappa})=\frac{-1+\tilde a+\phi}{z^\kappa}\tilde{c}_z+\frac{(\frac{\nu_s}{\nu}z-\pa^{-1}_z\tilde{a}-\pa^{-1}_z{\phi})}{z^\kappa}\tilde{c}_{zz}+(2\tilde{a}_z-2\phi)\frac{\tilde{c}}{z^\kappa}.
\end{split}
\end{equation}
}{Moreover, for non-diffusive case, differentiating the PDE of $\tilde c$, we see
\begin{equation}
    \tilde{c}_{zs}=(\underbrace{\frac{\lambda_s}{\lambda}+\frac{\nu_s}{\nu}}_{-1}+\tilde a+\phi)\tilde{c}_z+(\frac{\nu_s}{\nu}z-\pa^{-1}_z\tilde{a}-\pa^{-1}_z{\phi})\tilde{c}_{zz}+(2\tilde{a}_z-2\phi)\tilde{c},
\end{equation}
by utilizing the identities involving $\phi=e^{-z}$ and $\lambda_s/\lambda+\nu_s/\nu=-1$. Dividing $z^k$ on both sides gives
\begin{equation}
    \partial_s(\frac{\tilde{c}_{z}}{z^k})=\frac{-1+\tilde a+\phi}{z^k}\tilde{c}_z+\frac{(\frac{\nu_s}{\nu}z-\pa^{-1}_z\tilde{a}-\pa^{-1}_z{\phi})}{z^k}\tilde{c}_{zz}+(2\tilde{a}_z-2\phi)\frac{\tilde{c}}{z^k}.
\end{equation}
}
\ifthenelse{\boolean{isSimplified}}{Similarly, assuming that $ \tilde c=o(z^{\kappa})$ is imposed, we compute that 
\begin{equation*}
    (-1+\phi)+\tilde a\simeq -z,\, (\frac{\nu_s}{\nu}z-\pa^{-1}_z\tilde{a}-\pa^{-1}_z{\phi})\approx z,\, \text{as }z\rightarrow 0^+.
\end{equation*}
}{we compute (recall \eqref{identity-vanishing-speed} and $\tilde a=o(z)$)
\begin{equation}
    \begin{split}
        &(-1+\phi)+\tilde a\simeq -z\\
        &(\frac{\nu_s}{\nu}z-\pa^{-1}_z\tilde{a}-\pa^{-1}_z{\phi})\approx z\\
        & \tilde c=o(z^{k})\qquad\text{ (if imposed)},
    \end{split}
\end{equation}
where $z\rightarrow 0^+$.}
Combining the computations above, we observe that for non-diffusive case
\begin{equation}
    \partial_s(\frac{\tilde{c}_{z}}{z^\kappa})=o(1)\implies \tilde c_z=o(z^\kappa) \text{ is time invariant for }{0\leq \kappa<1,} \qquad\text{ as }z\rightarrow 0^+,
\end{equation}
{for each $s\geq s_0$.}

\textbf{Step 3}. \emph{The vanishing speeds of $\tilde c$ for diffusive case. }For diffusive case ($\sigma=1$), since $\tilde c\in C^3_z([0,\nu^{-1}])$, using Peano remainder with order 3, we expand at $z=0$ along the $z-$axis for each $s\geq s_0$. 
\ifthenelse{\boolean{isSimplified}}{Upon applying the boundary condition on $\tilde c$ \eqref{sigma-one-nc} and using L'hopital's rule, we obtain
\begin{equation}\label{Taylor-second-derivative-on-tilde-c:diffusive}
    \tilde c_{zz}(s,z)=c_{zzz}(s,0)z+o(z),\text{ as }z\rightarrow 0^+.
\end{equation}
Injecting
\eqref{Taylor-second-derivative-on-tilde-c:diffusive} into \eqref{vanishing-speed-analysis-simplified}, gives
\begin{equation*}
\begin{split}
     \partial_s\left(\frac{\tilde{c}}{z^\kappa}\right)
     &=(\frac{\lambda_s}{\lambda}+2\tilde a+2\phi)\frac{\tilde{c}}{z^\kappa}+o(1), \text{ as }z\rightarrow 0^+\text{ if }\kappa<1.
\end{split}
\end{equation*}}{Upon applying the boundary condition on $\tilde c$ \eqref{sigma-one-nc}, we see \begin{equation}
    \begin{split}
        \tilde c(s,z)&=c(s,0)+c_z(s,0)z+\frac{c_{zz}(s,0)}{2!}z^2+\frac{c_{zzz}(s,0)}{3!}z^3+o(z^3)\\
        &=c_z(s,0)z+\frac{c_{zzz}(s,0)}{3!}z^3+o(z^3),\text{ as }z\rightarrow 0^+.
    \end{split}
\end{equation}
Differentiating twice gives
\begin{equation}\label{Taylor-second-derivative-on-tilde-c:diffusive}
    \tilde c_{zz}(s,z)=c_{zzz}(s,0)z+o(z),\text{ as }z\rightarrow 0.
\end{equation}
where we repeatedly applied L'hopital's rule on $o(z^3)$ to conclude $\partial_z o(z^3)=o(z^2)$ and $\partial_z o(z^2)=o(z)$. Injecting
\eqref{Taylor-second-derivative-on-tilde-c:diffusive} into \eqref{vanishing-speed-analysis-simplified}, gives
\begin{equation*}
\begin{split}
     \partial_s\left(\frac{\tilde{c}}{z^\kappa}\right)&=(\frac{\lambda_s}{\lambda}+2\tilde a+2\phi)\frac{\tilde{c}}{z^\kappa}+\frac{\lambda}{\nu^2}\tilde c_{zzz}(s,0)z^{1-\kappa}+o(z^{1-\kappa})+o(1)\\
     &=(\frac{\lambda_s}{\lambda}+2\tilde a+2\phi)\frac{\tilde{c}}{z^\kappa}+o(1), \text{ as }z\rightarrow 0^+\text{ if  }{0\leq \kappa<1},
\end{split}
\end{equation*}}{for each $s\geq s_0$.} Thus, for diffusive case, $\tilde c=o(z^\kappa)$ is also time-invariant.

\textbf{Step 4}. \emph{The vanishing speeds of $\tilde a$. }\ifthenelse{\boolean{isSimplified}}{ Assuming that time-invariant condition $\tilde c=o(z^{\kappa_2})$ ($\kappa_2<1$) is imposed, dividing the PDE by $z^\kappa$, we see {for each $s\geq s_0$,}\begin{equation}
    \partial_s\left(\frac{\tilde{a}_{z}}{z^\kappa}\right)   
        =-\frac{(1-\phi-\tilde a)}{z^\kappa} \tilde a_z-\frac{(\partial_z^{-1} \phi - \frac{\nu_s}{\nu} z+\partial_z^{-1} \tilde{a})}{z^\kappa} \tilde{a}_{zz}-\phi \frac{\tilde{a}}{z^\kappa}  - \frac{\partial_z^{-1} \tilde{a}}{z^\kappa} \phi+ \frac{\nu_s}{\nu}z^{1-\kappa}\phi - \lambda^{1-\sigma}\nu \frac{\tilde c}{z^\kappa}.
\end{equation} In a similar fashion, we compute that for $\kappa\leq \kappa_2<1$
\begin{equation*}
     -\frac{(1-\phi-\tilde a)}{z^\kappa},\,\frac{(\partial_z^{-1} \phi - \frac{\nu_s}{\nu} z+\partial_z^{-1} \tilde{a})}{z^\kappa},\,\frac{\tilde{a}}{z^\kappa},\,\frac{\partial_z^{-1} \tilde{a}}{z^\kappa},\, z^{1-\kappa}\phi,\,\frac{\tilde c}{z^\kappa}=o(1)\text{ as }z\rightarrow 0^+,
\end{equation*}
and thus the equation of $a_z$ is reduced to
\begin{equation}
    \partial_s\left(\frac{\tilde{a}_{z}}{z^\kappa}\right)=o(1) \text{ as }z\rightarrow 0^+\text{ if }\kappa\leq \kappa_2<1,
\end{equation}{for each $s\geq s_0$.}}{The PDE of $\tilde a_z$ can be rewritten as
\begin{equation*}
\begin{split}
     &\partial_s\tilde{a}_{z} + (1-\phi-\tilde a) \tilde a_z +(\partial_z^{-1} \phi - \frac{\nu_s}{\nu} z+\partial_z^{-1} \tilde{a}) \tilde{a}_{zz} +\phi \tilde{a}  + \partial_z^{-1} \tilde{a} \phi 
        \\
        &=  \frac{\nu_s}{\nu}z\phi - \lambda^{1-\sigma}\nu \tilde c,
\end{split}  
\end{equation*}
if we know that as $z\rightarrow 0^+$.
\begin{equation*}
    \tilde c=o(z^{\kappa_2}) \text{ for }\kappa_2<1 \text{ is time-invariant,}
\end{equation*}
setting $\kappa\leq \kappa_2<1$ and dividing the previous equation by $z^\kappa$ gives
\begin{equation}
    \partial_s\left(\frac{\tilde{a}_{z}}{z^\kappa}\right)   
        =-\frac{(1-\phi-\tilde a)}{z^\kappa} \tilde a_z-\frac{(\partial_z^{-1} \phi - \frac{\nu_s}{\nu} z+\partial_z^{-1} \tilde{a})}{z^\kappa} \tilde{a}_{zz}-\phi \frac{\tilde{a}}{z^\kappa}  - \frac{\partial_z^{-1} \tilde{a}}{z^\kappa} \phi+ \frac{\nu_s}{\nu}z^{1-\kappa}\phi - \lambda^{1-\sigma}\nu \frac{\tilde c}{z^\kappa},
\end{equation}
we compute, using $\tilde a=o(z)$ and $\antipartial_z\tilde a=o(z^2)$ that
\begin{equation*}
    \begin{split}
       & 1-\phi-\tilde a\simeq z\implies -\frac{(1-\phi-\tilde a)}{z^\kappa}=o(1)\\
       & (\partial_z^{-1} \phi - \frac{\nu_s}{\nu} z+\partial_z^{-1} \tilde{a})\approx z \implies   \frac{(\partial_z^{-1} \phi - \frac{\nu_s}{\nu} z+\partial_z^{-1} \tilde{a})}{z^\kappa}=o(1)\\
       & \frac{\tilde{a}}{z^\kappa}=o(z^{1-\kappa})\\
    &\frac{\partial_z^{-1} \tilde{a}}{z^\kappa}=o(z^{2-\kappa})\\
    & z^{1-\kappa}\phi \simeq z^{1-\kappa}\\
    &\frac{\tilde c}{z^\kappa}=o(z^{\kappa_2-\kappa}),
    \end{split}
\end{equation*}
thus the equation of $a_z$ is reduced to
\begin{equation}
    \partial_s\left(\frac{\tilde{a}_{z}}{z^\kappa}\right)=o(1) \text{ as }z\rightarrow 0^+{ if 0\leq \kappa\leq \kappa_2<1}.
\end{equation}} 
\end{proof}
\ifthenelse{\boolean{isSimplified}}{}{Now we split our analysis into non-diffusive and diffusive cases and treat the vanishing speeds $\epsilon_a, \epsilon_c$ as parameters.}

\section{Bootstrap Argument for the case of Non-Diffusive Temperature}
\ifthenelse{\boolean{isSimplified}}{In this section, we perform a bootstrap argument on the non-diffusive case ($\sigma=0$). Recall that earlier we have imposed the orthogonality-like condition
\begin{equation*}
    \tilde a(s,z)=o(z)\text{ as }z\rightarrow 0^+.
\end{equation*}
In light of Corollary \ref{Invariance-of-vanishing speed}, we shall further impose the {time-invariant vanishing-speed condition at $z=0$: That is, for each $s\geq s_0$ ($s_0$ is the start time defined below),}
\begin{equation}\label{decay-speed-condition}
    \tilde a_z(s,z)=o(z^{\epsilon_a}),\, \tilde c=o(z^{\epsilon_c}), \, \tilde c_z=o(z^{\epsilon_c})\text{ as }z\rightarrow 0^+,\,\text{ for } 0<\epsilon_a\leq \epsilon_c<1.
\end{equation}
\begin{remark}
    Notice that, upon using L'hopital's rule, in particular, \label{decay-speed-condition} implies \begin{equation}\label{vanishing speed sum}
        \tilde c=o(z^{1+\epsilon_c})\text{ as }z\rightarrow 0^+ \text{ for }0<\epsilon_c<1 \text{ holds, and it is also time-invariant. }
    \end{equation} 
\end{remark}}{In this section, we study the case $\sigma=0$ and initially impose the conditions that all the weights
\begin{equation}\label{initial-parameter-weights}
    \text{$z^{-\alpha},$ and $z^{-k}$},\quad 0<\alpha,\gamma<3,
\end{equation}
as well as further impose the time-invariant vanishing-speed boundary conditions (Recall Corollary \ref{Invariance-of-vanishing speed})
\begin{equation}\label{decay-speed-condition}
    \tilde a_z(s,z)=o(z^{\epsilon_a}),\, \tilde c=o(z^{\epsilon_c}), \, \tilde c_z=o(z^{\epsilon_c})\text{ as }z\rightarrow 0^+,\,\text{ for } 0<\epsilon_a\leq \epsilon_c<1.
\end{equation}
Recall that earlier we have imposed the orthogonality-like condition
\begin{equation*}
    \tilde a(s,z)=o(z)\text{ as }z\rightarrow 0^+.
\end{equation*}
We shall refine the range of these parameters by the end of this section.
\begin{remark}
    Notice that this, upon using L'hopital's rule, in particular, implies \begin{equation}\label{vanishing speed sum}
        \tilde c=o(z^{1+\epsilon_c})\text{ as }z\rightarrow 0^+ \text{ for }0<\epsilon_c<1 \text{ holds, and it is also time-invariant. }
    \end{equation} 
\end{remark}
}
{For notational simplicity, we will be writing $\partial_z^{-1} f(s,z)=\int_0^z f(s,\tilde z)d\tilde z$ and $\partial_z^{-1}f g=(\partial_z^{-1}f) g$.}

\subsection{Definitions}

\ifthenelse{\boolean{isSimplified}}{In the following two definitions, we impose that all the weights satisfy
\begin{equation}\label{initial-parameter-weights}
    \text{$z^{-\alpha},$ and $z^{-\gamma}$},\quad 0<\alpha,\gamma<3.
\end{equation}
{All the parameters are shared in the following two definitions.}}{}
\begin{definition}[Initial closeness]\label{def:ini crit-beta=0:diffusive}
Let $\lambda_0^*>0$ be fixed. We say that {$(a_0,c_0)$} is initially close to the blowup profile if there exists {$\lambda_0^*>\lambda_0>0$} and $\nu_0>0$ such that the decomposition \eqref{decomp} satisfies:
\begin{itemize}
\item[(i)] \emph{Initial values of the modulation parameters} (note that this fixes the value of $s_0$):
\begin{eqnarray}
\label{bd:parametersini-beta=0}
 \lambda_0 = s_0e^{-s_0}, \ \  \frac{1}{N_0 s_0}\leq \nu_0 \leq  \frac{N_0}{s_0} {\text{ for some }s_0>0}.
\end{eqnarray} 
\item[(ii)] \emph{Compatibility condition for the initial perturbation}. $\tilde a_0\in C^2 ([0,\frac1{\nu_0}))$ satisfies the boundary conditions \eqref{smooth:orthogonality} and the integral condition \eqref{equation:epsilon-bc}.
\item[(iii)] \emph{Initial size of the remainder in the self-similar variables}. 
\begin{eqnarray}\label{bd:eini-beta=0:diffusive}
&\mathcal I_a^2(s_0) = \int_0^{z^*} z^{-\alpha}(\pa_z\tilde a_0)^2\;dz  < \delta^2 s_0^{-h_a},
\ \ \mathcal E_a^2(s_0) =\sup_{z^*\leq z\leq\frac1{\nu_0}}|\tilde a_0|^2 < \frac{1}{16}s_0^{-h_a}
\\
&\mathcal I_c^2(s_0) = \int_0^{z^*} z^{-\gamma}(\pa_z\tilde c_0)^2\;dz <  \delta^2 e^{-h_c s_0}, \ \ \mathcal E_c^2(s_0) =\sup_{z^*\leq z\leq\frac1{\nu_0}}|\tilde c_0|^2 < \frac{1}{16} e^{-h_c s_0},
\\
\end{eqnarray} 
{for some small number $\delta > 0$, $z^*>0 $ and $h_a,h_c>0$.}
\end{itemize}
\end{definition}

\begin{definition}[Trapped on $\left(s_0,s_1\right)$ ]\label{def:crit-beta=0:diffusive}
 We say that a solution $\tilde{a}(s,z)$ is trapped on $[s_0,s_1]$ with $s_0<s_1\leq \infty$, if it satisfies the properties of Definition \ref{def:ini crit-beta=0:diffusive} at time $s_0$ and if for all $s\in [s_0,s_1]$, $a(s,z)$ can be decomposed as in \eqref{decomp} with:
\begin{itemize}
\item[(i)] \emph{Values of the modulation parameters:}
\begin{eqnarray}\label{bd:parameterstrap-critical:diffusive}
\frac 1{M} se^{- s}< \lambda < M se^{- s}, \ \ \frac {1}{Ns} 
<\nu < \frac{N}{s}.
\end{eqnarray} 
{ for some $M,N>0$.}
\item[(ii)] \emph{Decay in time of the remainder in the self-similar variables:}
\begin{equation}\label{smooth:bd:etrap:diffusive}
    \begin{split}
    &\mathcal I_a^2(s) = \int_0^{z^*}z^{-\alpha}\tilde a_z^2\;dz  < s^{-h_a}, \qquad \mathcal E_a^2(s) = \sup_{z^*\leq z\leq\frac1{\nu(s)}}|\tilde a|^2 < s^{-h_a}
    \\
    &\mathcal I_c^2(s) = \int_0^{z^*} z^{-\gamma}(\pa_z\tilde c)^2\;dz <  e^{-h_c s}, \ \ \mathcal E_c^2(s) =\sup_{z^*\leq z\leq\frac1{\nu(s)}}|\tilde c|^2 < e^{-h_c s},
\end{split}
\end{equation}
\end{itemize}
\end{definition}
{These two definitions essentially encapsulate the bootstrap assumptions.}
{
\begin{remark}
    Notice that $\mathcal I_a^2$, $\mathcal I_c^2$ {may} have a singularity at zero. To remove such singularity, one type of sufficient conditions to impose is that as $z\rightarrow 0^+$,
    \begin{equation}\label{singularity-at-definition}
        \begin{split}
            &z^{-\alpha}\tilde a_z^2=o(z^{-1+\epsilon_1})\iff \tilde a_z=o(z^{\frac{\alpha-1+\epsilon_1}{2}})\impliedby \tilde a_z=o(z^{\frac{\alpha-1}{2}+\epsilon_1})\\
             &z^{-\gamma}\tilde c^{2}_z=o(z^{-1+\epsilon_3}) \iff \tilde c_z=o(z^{\frac{\gamma-1+\epsilon_3}{2}})\impliedby \tilde c_z=o(z^{\frac{\gamma-1}{2}+\epsilon_3}),
        \end{split}
    \end{equation}
    for some $\epsilon_1,\epsilon_2,\epsilon_3>0$, by our initial choice of parameters \eqref{initial-parameter-weights}.
\end{remark}
}\ifthenelse{\boolean{isSimplified}}{We make the following assumptions to keep computation simple initially:
    \begin{alignat}{2}\label{initial-parameter-other}
        & 1<h_a\leq 2,\qquad && h_c>0 \nonumber \\
        &M=2,\qquad &&N=4>N_0=3\geq 2\\
        &0<\frac{\alpha-1}{2}<\epsilon_a\leq \epsilon_c,\qquad &&0<\frac{\gamma-1}{2}< \epsilon_c <1 \nonumber .
    \end{alignat}
}{Initially, we make the following assumptions to keep computation simple and keep the structure similar to \cite{Collot2023}
\begin{subequations}\label{initial-parameter-other}
\begin{align}
    &1<h_a\leq 2\\
    &h_c>0\\
    &M=2,N>N_0\geq 2\\
    & \frac{\alpha-1}{2}<\epsilon_a\leq \epsilon_c\\
    &\frac{\gamma-1}{2}< \epsilon_c <1
\end{align} 
\end{subequations}
whereas all the additional parameters shall be chosen as our analysis progresses.}

\subsection{A-priori Estimates}
{In this subsection, we derive some a-priori estimates implied by the bootstrap assumptions we defined above, which will be useful when we improve their corresponding bounds later. The next Lemma gives $L^\infty$ estimates in terms of the energy bounds we defined in the previous subsection.}
\begin{lemma}\label{lemma:prelim-est:diffusive}

Assume that \eqref{initial-parameter-weights} and \eqref{initial-parameter-other} are satisfied. For any $\delta>0,z^*\geq 1,\epsilon>0$, { there exists an} $s_0^*$ large enough, {such that} if $a$ is trapped on $[s_0,s_1]$, then for all $s_0^*\leq s_0\leq s \leq s_1$:
\be \label{smooth:bd:varepsilonLinfty:diffusive}
\begin{split}
    &\| \tilde a (s) \|_{L^\infty([0,\nu^{-1}])} \leq C(z^*) s^{-\frac{h_a}{2}},\\
    &\,\| \tilde c (s) \|_{L^\infty([0,\nu^{-1}])} \leq C(z^*) e^{-\frac{h_c}{2}s},\\
    &\, \norm{\tilde c(s,\cdot) (\cdot)^{-\frac{\gamma+1}{2}}}_{L^\infty(0,z^*)}\leq C \mathcal I_c(s) \leq C e^{-\frac{h_c}{2}s}\\
\end{split}
\ee
and
\be \label{smooth:bd:integral:diffusive}
\begin{split}
  &\nu  \int_0^{\frac{1}{\nu}} (\phi + \tilde a)^2 (z) dz  \leq  s^{-1} N,
 \quad \left|\lambda\int_0^{\frac{1}{\nu}} \nu^2 \partial_z^{-1} \tilde{c}(z) dz\right|\leq e^{(-1-h_c/2+\epsilon)s}.
\end{split}
\ee
\end{lemma}

\begin{proof}
From the vanishing boundary condition, Cauchy-Schwarz, and \eqref{smooth:bd:etrap:diffusive}, for $0<z\leq z^*$:
\be \label{smooth:bd:varepsiloninterior:diffusive}
\resizebox{.95 \textwidth}{!}{$
\begin{split}
    &|\tilde a(z)|=\left|\int_0^z \pa_z \tilde a d\tilde z\right|\leq \mathcal I_a \sqrt{\int_0^z z^\alpha d\tilde z}= \mathcal I_a(s) (\frac{z^{\alpha+1}}{\alpha+1})^\frac{1}{2},\,|\tilde c(z)|=\left|\int_0^z \pa_z \tilde c d\tilde z\right|\leq \mathcal I_c \sqrt{\int_0^z z^\gamma d\tilde z}=\mathcal  I_c(s) (\frac{z^{\gamma+1}}{\gamma+1})^\frac{1}{2}\\
    &|\pa^{-1}_z\tilde a (z)|=\left|\int_0^z \tilde a d\tilde z\right| \leq \mathcal  I_a(s) \frac{2}{(\alpha+3)\sqrt{\alpha+1}}z^{\frac{\alpha+3}{2}},\,|\pa^{-1}_z\tilde c (z)|=\left|\int_0^z \tilde c d\tilde z\right| \leq \mathcal I_c(s) \frac{2}{(\gamma+3)\sqrt{\gamma+1}}z^{\frac{\gamma+3}{2}}\\
\end{split}$}
\ee
and thus we have 
\begin{equation}
\resizebox{.92 \textwidth}{!}{$
     \norm{\tilde a(s)}_{L^\infty(0,z^*)}\leq C(z^*) s^{-\frac{h_a}{2}},\, \norm{\tilde c(s)}_{L^\infty(0,z^*)}\leq C(z^*) e^{-\frac{h_c}{2}s},\, \norm{\tilde c(s,z) z^{-\frac{\gamma+1}{2}}}_{L^\infty(0,z^*)}\leq C \mathcal I_c(s)\leq C e^{-\frac{h_c}{2}s}.
     $}
\end{equation}
This, combined with the bounds of $\mathcal E_a$ and $\mathcal E_c$ shows equation of \eqref{smooth:bd:varepsilonLinfty:diffusive}.
{Now, to show \eqref{smooth:bd:integral:diffusive}, we use \eqref{bd:parameterstrap-critical:diffusive}--\eqref{smooth:bd:varepsilonLinfty:diffusive} to obtain}
\begin{equation*}
    \resizebox{.95 \textwidth}{!}{$\begin{split}
\nu\int_0^{\frac{1}{\nu}} (\phi + \tilde a)^2 (z) dz&= \nu\left( \int_0^\infty \phi^2 (z)dz -  \int_{\nu^{-1}}^\infty \phi^2 (z)dz+ 2 \int_0^{\frac{1}{\nu}} \phi  \tilde a (z) dz +\int_0^{\nu^{-1}}  \tilde a^2 dz \right)
\\
&= \nu\left( \frac12 +O(\| \tilde a \|_{L^\infty([0,\nu^{-1}])})+O(\nu^{-1}\| \tilde a \|_{L^\infty([0,\nu^{-1}])}^2) \right) 
\\
&\leq {Ns^{-1}\left(\frac12  +O(C(z^*)s^{-\frac{h_a}{2}}\bigvee C(z^*)s^{-{h_a}+1})\right)=\frac{N}{2} s^{-1}+O(C(z^*)s^{-\frac{h_a}{2}-1}\bigvee C(z^*)s^{-{h_a}})} .
\end{split}$}
\end{equation*}
{Next, }using the first equation \eqref{smooth:bd:varepsilonLinfty:diffusive} and the definition of trapped solution Definition \eqref{def:crit-beta=0:diffusive}, we have
    \begin{equation}\label{estimate-on-primitive-of-c-diffusive-1}
            \abs{\partial_z^{-1} \tilde c}\leq \int_0^z \abs{\tilde c} d\tilde z \leq \norm{c}_{L^\infty[0,\nu^{-1}]} \nu^{-1}(s)\leq C(z^*) s e^{-\frac{h_c s}{2}}  \\
    \end{equation}
    Thus, for $s$ large enough, we have
    \begin{equation}\label{estimate-on-primitive-of-c-diffusive-2}
        \left|\lambda\int_0^{\frac{1}{\nu}} \nu^2 \partial_z^{-1} \tilde{c}(z) dz\right|\leq \lambda\nu^2 \frac{1}{\nu} \norm{\partial_z^{-1} \tilde c}_{L^\infty(0,\nu^{-1})}\leq C(z^*) s e^{(-\frac{h_c}{2}-1)s} \leq e^{(-1-h_c/2+\epsilon)s}
    \end{equation}
\end{proof}

\begin{remark}
    Notice that {Lemma \ref{lemma:prelim-est:diffusive} implies that }{for each $s_1\geq s\geq s_0$,} $\tilde c=O(z^\frac{\gamma+1}{2}),$ as $z\rightarrow 0^+$, which can also be implied by our {choice of parameters} and vanishing speed boundary condition $\tilde c=o(z^{1+\epsilon_c})$ and $\epsilon_c>\frac{\gamma-1}{2}$.
\end{remark}

{The following Lemma gives estimates on the modular variables.}
\begin{lemma}[Modulation Equations] \label{lemma:smoothmodulation:diffusion}
Assume that \eqref{initial-parameter-weights} and \eqref{initial-parameter-other} are satisfied. For any $z^*\geq 1$, $\epsilon>0$, and $ \delta>0$, there exists a large self-similar time $s_0^*$ such that for any $s_0\geq s_0^*$, for any solution which is trapped on $[s_0,s_1]$, we have for $s\in [s_0,s_1]$:
\begin{eqnarray}\label{smoothmodulationequations:diffusive}
\left| \frac{\lambda_s}{\lambda}+1\right|\leq C s^{-1}, \qquad \left| \frac{\nu_s}{\nu} \right|\leq C s^{-1},\qquad \frac{\lambda_s}{\lambda}=-1+O(C(z^*)s^{-1})
\end{eqnarray}
for $C>0$ independent of the bootstrap constants, and
\begin{eqnarray}\label{smooth:bd:boostrap improved parameters:diffusive}
\frac{1}{1+\epsilon} se^{-s}\leq \lambda \leq  (1+\epsilon) se^{-s}, \qquad \frac{1}{(N_0+\epsilon)s} \leq \nu \leq\frac{N_0+\epsilon}{s}.
\end{eqnarray}
Moreover, if $s_1=\infty$ then there exists a constant $\tilde \lambda_\infty>0$ such that
\begin{eqnarray}\label{smooth:bd:boostrap improved parameters:diffusive}
\lambda= s e^{-s}\big(\tilde \lambda_\infty + O(s^{-h_a+1}) \big) ,\quad
\nu=\frac{1}{s}+O(s^{-h_a}).
\end{eqnarray}
\end{lemma}

\begin{proof}
\noindent \textbf{Step 1}.  \eqref{smoothmodulationequations:diffusive} is an immediate consequence of Lemma \ref{lemma:prelim-est:diffusive} by recalling \eqref{smoothmodulationequations:initial}.

\noindent \textbf{Step 2}. \emph{Equation for $\nu$}. From the proof of Lemma \ref{lemma:prelim-est:diffusive}, we know that
\be \label{smooth:modulation:inter:diffusive}
\begin{split}
  &2\int_0^{\frac{1}{\nu}} (\phi + \tilde a)^2 (z) dz =1+O(C(z^*)s^{-h_a+1}) ,
  \\
  &\lambda\int_0^{\frac{1}{\nu}} \nu \partial_z^{-1} \tilde{c}(z) dz = O(C(z^*) s e^{(-\frac{h_c}{2}-1)s}).
\end{split}
\ee
Injecting \eqref{smooth:modulation:inter:diffusive} in \eqref{smoothmodulationequations:initial} gives:
\be \label{smooth:modulation:inter2:diffusive}
-\frac{\nu_s}{\nu^2} = \frac1\nu (\frac{\lambda_s}\lambda +1)= 1 +O(C(z^*)s^{-{h_a}+1})
\ee
when $s_0$ is large enough.
Integrating \eqref{smooth:modulation:inter2:diffusive} with time, we find:
\be \label{smooth:modulation:inter3}
\frac{1}{\nu}=\frac{1}{\nu_0}+(s-s_0)+O(C(z^*)(s^{-h_a+1+1}-s_0^{-h_a+1+1}))=\frac{1}{\nu_0}+(s-s_0)+O(C(z^*)s^{-h_a+1}(s-s_0)).
\ee
Therefore, since $\nu_0^{-1}\leq N_0 s_0$ from \eqref{bd:parametersini-beta=0} we infer for $s_0^*$ large enough depending on $z^*$:
\be \label{smooth:modulation:inter1:diffusive}
\frac{1}{\nu}\leq N_0 s_0+(s-s_0)+O(C(z^*)s^{-h_a+1}(s-s_0))= s+(N_0-1)s_0+O(C(z^*)s^{-h_a+1}(s-s_0))\leq (N_0+\epsilon)s.
\ee
One finds similarly using $s_0/N_0\leq \nu_0^{-1}$ from \eqref{bd:parametersini-beta=0} that $\frac{1}{\nu}\geq s/(N_0+\epsilon)$. This and \eqref{smooth:modulation:inter1:diffusive} imply the second inequality in \eqref{smooth:bd:boostrap improved parameters:diffusive}. Finally, if $s_1=\infty$ then \eqref{smooth:modulation:inter3} implies $\nu^{-1}=s+O(s^{-h_a+2})$ and the second inequality in \eqref{smooth:bd:boostrap improved parameters:diffusive} follows.\\

\noindent \textbf{Step 3}. \emph{Equation for $\lambda$}. Injecting \eqref{smooth:modulation:inter:diffusive} in \eqref{smoothmodulationequations:initial} one finds:
\begin{equation}\label{asy on lambda derivative}
    \frac{\lambda_s}{\lambda}+1 = \frac{1}{s}+O(C(z^*)s^{-h_a})\implies \frac{\lambda_s}{\lambda}=-1+O(C(z^*)s^{-1})
\end{equation}
Since $\lambda = O(se^{- s})$ from \eqref{bd:parameterstrap-critical:diffusive}, one has
\begin{equation*}
    \frac{d}{ds}(\frac{e^s\lambda}{s}) =O(s^{-h_a}).
\end{equation*}
We integrate with time the above equation using $\lambda_0 = s_0 e^{-s_0}$ and find
$$
\lambda(s) = se^{-s}(1+\int_{s_0}^sO(s^{-h_a})ds)=se^{-s}(1+O(s_0^{-h_a+1}))
$$
This implies the first inequality in \eqref{smooth:bd:boostrap improved parameters:diffusive} for $s_0$ large enough. If $s_1=\infty$ then we set $\tilde\lambda_\infty=1+\int_{s_0}^\infty O(s^{-h_a})ds$ and rewrite the above equality as:
$$
\lambda(s) = se^{-s}(1+\int_{s_0}^\infty O(s^{-h_a})ds-\int_{s}^\infty O(s^{-h_a})ds )= se^{-s}(\tilde\lambda_\infty+O(s^{-h_a+1}) ),
$$
which gives the first inequality in \eqref{smooth:bd:boostrap improved parameters:diffusive}.
\end{proof}

\subsection{Interior Estimates}\label{differentiational-estimates}
{In this subsection, we derive two differential inequalities on the energies in the interior ($[0,z^*]$) of the domain ($[0,\nu^{-1}]$). These inequalities shall be later solved after we correctly refine the parameters in the subsection that follows.}

Since $\gamma>1$ and $\alpha<3$, we notice that
\begin{equation}\label{interior-parameter-a}
        \gamma>\alpha-2 \text{ is automatically satisfied.}
\end{equation}
\ifthenelse{\boolean{isSimplified}}{}{By our choice of parameter earlier \eqref{initial-parameter-other}, we see
\begin{equation}\label{vanishing-assumptions-on-a-and-b}
    \tilde a_z(s,z)=o(z^\frac{\alpha-1}{2}) \text{ as }z\rightarrow 0^+,
\end{equation}
and this fact will be useful for the following lemma.} 

\begin{lemma}[Interior estimate for $\tilde a$]\label{smooth:lemma:interior:new idea}
Assume that \eqref{initial-parameter-weights}, \eqref{initial-parameter-other} and  \eqref{interior-parameter-a}are satisfied. For any $z^*\geq 1$, $\epsilon>0$, and $ \delta>0$, there exists a large self-similar time $s_0^*$ such that for any $s_0\geq s_0^*$, for any solution which is trapped on $[s_0,s_1]$, we have for $s\in [s_0,s_1]$:
\begin{equation}\label{interior-estimate-of-a-diffusive}
\begin{split}
     \frac{d}{ds}\mathcal I_a^2&\leq \mathcal I_a^2 \left( {-\alpha+1}+O(C(z^*) s^{-\frac{h_a}{2}})+\frac{1}{\sqrt{\alpha+1}}+\frac{4}{(\alpha+3)}\sqrt{\frac{3}{8(\alpha+1)}}\right)\\
     &\quad+C(z^*)\mathcal I_a(s^{-1}+ e^{(-1-\frac{h_c}{2}+\epsilon) s}+ e^{(-1-\frac{h_c}{2})s}).\\
\end{split}
\end{equation}

\end{lemma}
\ifthenelse{\boolean{isSimplified}}{}{\begin{remark}
    Notice ${-\alpha+1}+\frac{1}{\sqrt{\alpha+1}}+\frac{4}{(\alpha+3)}\sqrt{\frac{3}{8(\alpha+1)}}<0\iff\alpha>1.88415$.
\end{remark}}

\begin{proof}

\ifthenelse{\boolean{isSimplified}}{Let $w=z^{-\alpha}$, calculating the derivative of $\mathcal I_a$, we obtain that \begin{equation} \label{smooth:interior:id:expression2}\resizebox{.92 \textwidth}{!}{$
\begin{split}
    &\frac{1}{2}\frac{d}{ds} \int_0^{z^*}  w \tilde a_z ^2 dz=   \int_0^{z^*} ( -1  + \phi +\tilde a)   w \tilde a_z^2 dz 
     - \int_0^{z^*} (\partial_z^{-1} \phi - \frac{\nu_s}{\nu} z + \partial_z^{-1} \tilde a) \tilde a_{zz}  w \tilde a_z dz - \lambda \nu \int_0^{z^*} w\tilde a_z \tilde cdz
     \\
     &-\int_0^{z^*} \phi \tilde a  w \tilde a_z dz - 
     \int_0^{z^*} \partial_z^{-1} \tilde a   w \tilde a_z \phi dz 
     -\left[2\nu \int_0^{\frac{1}{\nu}} (\phi+ \tilde{a})^2(z) dz - \lambda \left(\int_0^{\frac{1}{\nu}} \nu^2 \partial_z^{-1} \tilde{c}(z) dz\right) \right]\int_0^{z^*} z\phi  w \tilde a_z dz,
\end{split}$}
\end{equation}
{where we have used the identities \eqref{smoothmodulationequations:initial-nc} and the PDE of $\tilde a_z$ \eqref{equation:az}.}
}{Calculating the derivative of $I_a$, one obtain
\begin{equation} \label{smooth:interior:id:expression2}
\begin{split} 
    &\frac{1}{2}\frac{d}{ds} \int_0^{z^*}  w \tilde a_z ^2 dz=   \int_0^{z^*} ( -1  + \phi +\tilde a)   w \tilde a_z^2 dz 
     - \int_0^{z^*} (\partial_z^{-1} \phi - \frac{\nu_s}{\nu} z + \partial_z^{-1} \tilde a) \tilde a_{zz}  w \tilde a_z dz
     \\
     &-\int_0^{z^*} \phi \tilde a  w \tilde a_z dz - 
     \int_0^{z^*} \partial_z^{-1} \tilde a   w \tilde a_z \phi dz \\
     &-\left[2\nu \int_0^{\frac{1}{\nu}} (\phi+ \tilde{a})^2(z) dz - \lambda \left(\int_0^{\frac{1}{\nu}} \nu^2 \partial_z^{-1} \tilde{c}(z) dz\right) \right]\int_0^{z^*} z\phi  w \tilde a_z dz
     \\
     & - \lambda \nu \int_0^{z^*} w\tilde a_z \tilde cdz,
\end{split}
\end{equation}
by utilizing the identities involving $\phi=e^{-z}$,{ the identities \eqref{smoothmodulationequations:initial-nc} and the PDE of $\tilde a_z$ \eqref{equation:az}}, where $w=z^{-\alpha}$.}

\noindent \underline{Potential and transport terms}.
\ifthenelse{\boolean{isSimplified}}{Notice that since $\partial_z^{-1} \phi\simeq z, \antipartial_z \tilde a=o(z^2)$, by  {using \eqref{vanishing speed sum} and \eqref{initial-parameter-other}, we have that}\begin{equation}
     \tilde a_z=o(z^{\frac{\alpha-1}{2}}) \text{ as }z\xrightarrow{}0^+ \implies  (\partial_z^{-1} \phi - \frac{\nu_s}{\nu} z + \partial_z^{-1} \tilde a)   z^{-\alpha} \tilde a_z^2=o(1) \text{ as }z\xrightarrow{}0^+.
\end{equation}}{}
{Consequently, }integrating by parts yields
\begin{equation}  \label{smooth:interior:id:expression}
    \begin{split}
        &  \int_0^{z^*} ( -1  + \phi +\tilde a)   w \tilde a_z^2 dz -\int_0^{z^*} (\partial_z^{-1} \phi - \frac{\nu_s}{\nu} z + \partial_z^{-1} \tilde a) \tilde a_{zz}  w \tilde a_z dz
        \\
        = & \Big(-\int_0^{z^*} (\phi+\tilde a)(\tilde{z})d\tilde{z}+ \frac{\nu_s}{\nu}z^* \Big) \frac{1}{2}w(z^*) \tilde a_z^2(z^*) \\
        &+\int_0^{z^*} \left( \frac{3}{2} \phi+\frac 32 \tilde a-1+\frac{\alpha-1}{2}\frac{\nu_s}{\nu}-\frac{\alpha}{2}z^{-1}\left(\pa_z^{-1}\phi+\pa_z^{-1}\tilde a \right)\right) w\tilde a_z^2dz .
    \end{split}
\end{equation}
For the boundary term, we know that $\|\tilde a\|_{L^\infty} \leq C(z^*) s^{-\frac{h_a}{2}}$ from \eqref{smooth:bd:varepsilonLinfty:diffusive}, and thus using \eqref{smoothmodulationequations:diffusive} gives:
\begin{equation}\label{condition-critical:s0-3}
\resizebox{.93 \textwidth}{!}{$\displaystyle
      -\int_0^{z^*} (\phi+\tilde a)(\tilde{z})d\tilde{z}+ \frac{\nu_s}{\nu}z^* \leq -(1-e^{-z^*}) + \|\tilde a\|_{L^\infty} z^*+O(C(z^*)s^{-1}) \leq -(1-e^{-z^*}) +C(z^*)s^{-\frac{h_a}{2}} \leq 0$,}
\end{equation} 
when $s_0$ is large enough. 
\ifthenelse{\boolean{isSimplified}}{}{This requires
\begin{equation}\label{first-vanishing-condition}
    (\partial_z^{-1} \phi - \frac{\nu_s}{\nu} z + \partial_z^{-1} \tilde a)   z^{-\alpha} \tilde a_z^2=o(1) \text{ as }z\xrightarrow{}0^+.
\end{equation}
Since $\partial_z^{-1} \phi\simeq z, \antipartial_z \tilde a=o(z^2)$, and thus 
\begin{equation}\label{identity-vanishing-speed}
    (\partial_z^{-1} \phi - \frac{\nu_s}{\nu} z + \partial_z^{-1} \tilde a)\approx z \text{ as }z\xrightarrow{}0^+,
\end{equation}
this condition is equivalent to
\begin{equation}
        \tilde a_z=o(z^{\frac{\alpha-1}{2}}) \text{ as }z\xrightarrow{}0^+.
\end{equation}}
Since the profile satisfies $\phi=e^{-z}$, \eqref{smoothmodulationequations:diffusive} and \eqref{smooth:bd:varepsilonLinfty:diffusive} yield:
\begin{equation}  \label{smooth:interior:inter1}
\begin{split}
    &\frac{3}{2} \phi+\frac 32 \tilde a-1+\frac{\alpha-1}{2}\frac{\nu_s}{\nu}-\frac{\alpha}{2}z^{-1}\left(\pa_z^{-1}\phi+\pa_z^{-1}\tilde a \right)
    \\
    &=\left(\frac{3}{2} e^{-z}+O(C(z^*) s^{-\frac{h_a}{2}})-1+O(s^{-1})-\frac{\alpha}{2}z^{-1}\left(1-e^{-z} +O(s^{-\frac {h_a}{2}}z) \right)\right)
    \\
&   =-1+\frac{3}{2}e^{-z}-\frac{\alpha}{2}\frac{1-e^{-z}}{z}+O(C(z^*)s^{-\frac{h_a}{2}})\\
& \leq -1+(\frac 32-\frac \alpha2)e^{-z}+O(C(z^*)s^{-\frac{h_a}{2}})\\
&\leq \begin{cases}
    -1+O(C(z^*)s^{-\frac{h_a}{2}}) \text{ if }\alpha \geq 3\\
    \frac{-\alpha+1}{2}+O(C(z^*) s^{-\frac{h_a}{2}}) \text{ if otherwise}
\end{cases}
\end{split}
\end{equation}
For {$0<\alpha<3$,} the previous argument shows:
\begin{equation}  \label{smooth:interior:bd:transport-a}
\resizebox{.93 \textwidth}{!}{$\displaystyle
\int_0^{z^*} ( -1  + \phi +\tilde a)   w \tilde a_z^2 dz 
     - \int_0^{z^*} (\partial_z^{-1} \phi - \frac{\nu_s}{\nu} z + \partial_z^{-1} \tilde a) \tilde a_{zz}  w \tilde a_z dz\leq  \mathcal I_a^2(s)
    \left(\frac{-\alpha+1}{2}+O(C(z^*) s^{-\frac{h_a}{2}})\right).$}
\end{equation}

\noindent \underline{The nonlocal terms}. \ifthenelse{\boolean{isSimplified}}{By using Cauchy-Schwarz and equation \eqref{smooth:bd:varepsiloninterior:diffusive} we see}{ By direct computations, using that Cauchy-Schwarz one gets
\begin{equation}
    \begin{split}
        &|\tilde a (z)|\leq \int_0^{z}  \abs{\tilde a_z} \sqrt{w} \sqrt{w^{-1}}dx \leq \mathcal I_a(s) \frac{z^{\frac{\alpha+1}{2}}}{(\alpha+1)^\frac12}\\
        &|\pa^{-1}_z\tilde a (z)| \leq \mathcal I_a(s) \frac{2}{(\alpha+3)\sqrt{\alpha+1}}z^{\frac{\alpha+3}{2}}
    \end{split}
\end{equation} for $z\in[0,z^*]$. Thus, we conclude, by using Hölder again}
\begin{equation*}
\begin{split}
    \int_0^{z^*} \Big| \phi \tilde a  w \tilde a_z \Big| dz=\int_0^{z^*} \Big| \phi \tilde a  \sqrt{w}\cdot \sqrt{w}\tilde a_z \Big| dz &\leq \frac{1}{\sqrt{\alpha+1}} \mathcal I_a^2(s) (\int_0^\infty e^{-2z}z)^\frac{1}{2}= \frac{1}{2\sqrt{\alpha+1}}\mathcal I_a^2(s),
\end{split}
\end{equation*}
\begin{equation*}
\begin{split}
    \int_0^{z^*} \Big| \partial_z^{-1} \tilde a   w \tilde a_z \phi \Big| dz &=  \int_0^{z^*} \Big| \phi\partial_z^{-1} \tilde a   \sqrt{w}\cdot\sqrt{w} \tilde a_z  \Big| dz \\ 
    &\leq \frac{2}{(\alpha+3)\sqrt{\alpha+1}}\mathcal I_a^2(s)\Big(\int_0^{\infty}  z^3e^{-2z}dz \Big)^{\frac{1}{2}} = \frac{2}{(\alpha+3)}\mathcal I_a^2(s) \sqrt{\frac{3}{8(\alpha+1)}} ,
\end{split}
\end{equation*}

\noindent \underline{The source term}. Using \eqref{smooth:bd:etrap:diffusive}, \eqref{smooth:bd:integral:diffusive} and Cauchy-Schwarz:
$$
\left|2 \nu \Big(\int_0^{\frac{1}{\nu}} (\phi + \tilde a)^2 (z) dz\Big) \int_0^{z^*} z\phi  w \tilde a_z dz \right| \leq N s^{-1} \mathcal I_a \sqrt{\int_0^{z^*} \underbrace{z^2 \phi^2 w}_{z^{2-\alpha}e^{-2z}}dz} \leq C(z^*) s^{-1}\mathcal I_a(s),
$$

\begin{equation*}
    \begin{split}
        &\left|\lambda\nu^2 \int_0^{\frac1\nu} \pa_z^{-1} \tilde c(z) dz \int_0^{z^*} z\phi  w \tilde a_z dz \right|\leq e^{(-h_c/2-1+\epsilon) s} \mathcal I_a \sqrt{\int_0^{z^*} z^2\phi^2wdz} \leq C(z^*) e^{(-h_c/2-1+\epsilon) s}\mathcal I_a(s).
    \end{split}
\end{equation*}
Notice that { the }integrability is guaranteed when $\alpha<3$. Next, using \eqref{smooth:bd:varepsilonLinfty:diffusive}, \eqref{interior-parameter-a} and Hölder inequality, we estimate
\begin{equation*}
     \begin{split}
       \left| \lambda \nu \int_0^{z^*} w\tilde a_z \tilde cdz \right| &\leq Ce^{-s} \int_0^{z^*} \abs{ \tilde{a}_z z^{-\alpha/2}}\cdot \abs{z^{-\frac{\gamma+1}{2}} \tilde c} \cdot \abs{ z^{\frac{\gamma+1-\alpha}{2}}}\\
       &\leq Ce^{-s} \cdot C \mathcal I_c \cdot \int_0^{z^*}\abs{ \tilde{a}_z z^{-\alpha/2}}\cdot  \abs{ z^{\frac{\gamma+1-\alpha}{2}}} dz \\
       &\leq C e^{-s} \mathcal I_c \mathcal I_a(\int_0^{z^*}  z^{{\gamma+1-\alpha}} dz)^\frac{1}{2}\leq C(z^*) e^{-s} \mathcal I_c \mathcal I_a
    \end{split}
\end{equation*}
Finally, $\mathcal I_a$ satisfies
\ifthenelse{\boolean{isSimplified}}{\begin{equation}
\begin{split}
     \frac{1}{2}\frac{d}{ds}\mathcal I_a^2&\leq \mathcal I_a^2 ( \frac{-\alpha+1}{2}+O(C(z^*)s^{-\frac{h_a}{2}})+\frac{1}{2\sqrt{\alpha+1}}+\frac{2}{(\alpha+3)}\sqrt{\frac{3}{8(\alpha+1)}})\\
     &\quad+\mathcal I_a(C(z^*) s^{-1}+C(z^*) e^{(-h_c/2-1+\epsilon) s})+\mathcal I_a \mathcal I_c C(z^*) e^{-s},\\
\end{split}
\end{equation}}{\begin{equation}
\begin{split}
     \frac{1}{2}\frac{d}{ds}\mathcal I_a^2&\leq \mathcal  I_a^2 ( \frac{-\alpha+1}{2}+O(C(z^*)s^{-\frac{h_a}{2}})+\frac{1}{2\sqrt{\alpha+1}}+\frac{2}{(\alpha+3)}\sqrt{\frac{3}{8(\alpha+1)}})\\
     &\quad+\mathcal I_a(C(z^*) s^{-1}+C(z^*) e^{(-h_c/2-1+\epsilon) s})\\
     &\quad+\mathcal I_a\mathcal I_c C(z^*) e^{-s},
\end{split}
\end{equation}}
and we finish the proof by injecting the bootstrap assumptions \eqref{smooth:bd:etrap:diffusive}.
\end{proof}

For the following lemma, we further assume {that}
\begin{equation}\label{paramer-choice-c}
    \epsilon_c\geq  \frac{\gamma}{2}-1>0,
\end{equation}
{This, combined with} \eqref{vanishing speed sum}, implies that
\begin{equation}\label{tilde c vanishing corollary}
    \tilde c=o(z^\frac{\gamma}{2}),\text{ as }z\rightarrow 0^+.
\end{equation}
\begin{lemma}[Estimate on $\mathcal {I}_c$]\label{estimate on c_z}
Assume that \eqref{initial-parameter-weights}, \eqref{initial-parameter-other}, \eqref{interior-parameter-a} and \eqref{paramer-choice-c} are satisfied. For any $z^*\geq 1$, there exists a large self-similar time $s_0^*$ such that for any $s_0\geq s_0^*$, for any solution which is trapped on $[s_0,s_1]$, we have for $s\in [s_0,s_1]$:

\begin{equation}\label{I_c-bound:non-diffusive}
\frac{d}{ds} \mathcal I_c^2\leq(-\gamma+1+O(C(z^*)s^{-\frac{h_a}{2}}))\mathcal I_c^2
\end{equation}
\end{lemma}

\begin{proof}
\ifthenelse{\boolean{isSimplified}}{Let $w=z^{-\gamma}$, the derivative of $\mathcal I_c$ satisfies \begin{equation} \label{smooth:interior:id:expression2}
\frac{1}{2}\frac{d}{ds} \int_0^{z^*}  w \tilde c_z ^2 dz=  \int_0^{z^*} ( -1  + \phi +\tilde a)   w \tilde c_z^2 dz 
     + \int_0^{z^*} (-\partial_z^{-1} \phi + \frac{\nu_s}{\nu} z - \partial_z^{-1} \tilde a) \tilde c_{zz}  w \tilde c_z dz
     +\int_0^{z^*}(2\tilde a_z-2\phi)\tilde c_{z}  w \tilde c dz.
\end{equation}}{Using the PDE of $\tilde c_z$:
\begin{equation}
    \tilde{c}_{zs}=(\underbrace{\frac{\lambda_s}{\lambda}+\frac{\nu_s}{\nu}}_{-1}+\tilde a+\phi)\tilde{c}_z+(\frac{\nu_s}{\nu}z-\pa^{-1}_z\tilde{a}-\pa^{-1}_z\tilde{\phi})\tilde{c}_{zz}+(2\tilde{a}_z-2\phi)\tilde{c}.
\end{equation}
we compute the derivative of $I_c$,
\begin{equation} \label{smooth:interior:id:expression2}
\begin{split}
    &\frac{1}{2}\frac{d}{ds} \int_0^{z^*}  w \tilde c_z ^2 dz=  \int_0^{z^*} ( -1  + \phi +\tilde a)   w \tilde c_z^2 dz 
     + \int_0^{z^*} (-\partial_z^{-1} \phi + \frac{\nu_s}{\nu} z - \partial_z^{-1} \tilde a) \tilde c_{zz}  w \tilde c_z dz
     \\
     &+\int_0^{z^*}(2\tilde a_z-2\phi)\tilde c_{z}  w \tilde c dz \\
\end{split}
\end{equation}
where $w=z^{-\gamma}$.}
{
Integrating by parts yields
\begin{equation}  \label{smooth:interior:id:expression}
\resizebox{.93 \textwidth}{!}{$\displaystyle
    \begin{split}
        & \int_0^{z^*} ( -1  + \phi +\tilde a)   w \tilde c_z^2 dz 
     + \int_0^{z^*} (-\partial_z^{-1} \phi + \frac{\nu_s}{\nu} z - \partial_z^{-1} \tilde a) \tilde c_{zz}  w \tilde c_z dz
        \\
        &=\left((-\antipartial_z \phi+\frac{\nu_s}{\nu}z-\antipartial_z \tilde a)\tilde c_z^2 w\right)_{z=0}^{z=z^*}+\int_0^{z^*} \left( \frac{3}{2} \phi+\frac 32 \tilde a-1+\frac{\gamma-1}{2}\frac{\nu_s}{\nu}-\frac{\gamma}{2}z^{-1}\left(\pa_z^{-1}\phi+\pa_z^{-1}\tilde a \right)\right) w\tilde c_z^2dz.
    \end{split}
    $}
\end{equation}
Reapplying \eqref{condition-critical:s0-3} and using \eqref{initial-parameter-other} respectively, the boundary terms can be bounded above:
\begin{equation}
    \resizebox{.93 \textwidth}{!}{$\displaystyle \left(-\int_0^{z^*} (\phi+\tilde a)(\tilde{z})d\tilde{z}+ \frac{\nu_s}{\nu}z^*\right)w(z^*)\tilde c_z^2(s,z^*)\leq 0\quad\text{and}\quad \left(-\antipartial_z \phi+\frac{\nu_s}{\nu}z-\antipartial_z \tilde a\right)\tilde c_z^2 w=o(1)\text{ as }z\rightarrow 0^+,$}
\end{equation}
for $s$ large enough.
}\ifthenelse{\boolean{isSimplified}}{}{This requires
\begin{equation*}
    (\partial_z^{-1} \phi - \frac{\nu_s}{\nu} z + \partial_z^{-1} \tilde a)   z^{-\gamma} \tilde c^2_z=o(1) \text{ as }z\xrightarrow{}0^+,
\end{equation*}
and this condition is equivalent to
\begin{equation}
        \tilde c_z=o(z^{\frac{\gamma-1}{2}}) \text{ as }z\xrightarrow{}0^+,
\end{equation}
which is satisfied by \eqref{initial-parameter-other}.}
\ifthenelse{\boolean{isSimplified}}{Similarly to Lemma \ref{smooth:lemma:interior:new idea}, we see {\begin{equation}  \label{smooth:interior:inter1}
\begin{split}
    &\left( \frac{3}{2} \phi+\frac 32 \tilde a-1+\frac{\gamma-1}{2}\frac{\nu_s}{\nu}-\frac{\gamma}{2}z^{-1}\left(\pa_z^{-1}\phi+\pa_z^{-1}\tilde a \right)\right)
    \\
    &=\left(\frac{3}{2} e^{-z}+O(C(z^*)s^{-\frac{h_a}{2}})-1+O(s^{-1})-\frac{\gamma}{2}z^{-1}\left(1-e^{-z} +O(s^{-\frac {h_a}{2}}z) \right)\right)
    \\&\leq O(C(z^*)s^{-h_a/2})+
  \begin{cases}
    \frac{1-\gamma}{2}\qquad\text{ if }0<\gamma<3 \\
    -1\qquad\text{ if }\gamma>3.\\
\end{cases}
\end{split}
\end{equation}}}{One uses $\phi=e^{-z}$, \eqref{smoothmodulationequations:diffusive} and \eqref{smooth:bd:varepsilonLinfty:diffusive} to see:
\begin{equation}  \label{smooth:interior:inter1}
\begin{split}
    &\left( \frac{3}{2} \phi+\frac 32 \tilde a-1+\frac{\gamma-1}{2}\frac{\nu_s}{\nu}-\frac{\gamma}{2}z^{-1}\left(\pa_z^{-1}\phi+\pa_z^{-1}\tilde a \right)\right)
    \\
    &=\left(\frac{3}{2} e^{-z}+O(C(z^*)s^{-\frac{h_a}{2}})-1+O(s^{-1})-\frac{\gamma}{2}z^{-1}\left(1-e^{-z} +O(s^{-\frac {h_a}{2}}z) \right)\right)
    \\&\leq O(C(z^*)s^{-h_a/2})+
  \begin{cases}
    \frac{1-\gamma}{2}\qquad\text{ if }0<\gamma<3 \\
    -1\qquad\text{ if }\gamma>3\\
\end{cases}
\end{split}
\end{equation}
where we have used the fact that $1-e^{-z} \geq ze^{-z} $.}
Next, using Hölder's inequality and \eqref{smooth:bd:varepsilonLinfty:diffusive}, we estimate
{
\begin{equation}
\begin{split}
     \abs{\int_0^{z^*}2\tilde a_z\tilde c_{z}  \underbrace{w}_{z^{-\gamma}} \tilde c dz} &\leq 2\int_0^{z^*} \abs{\tilde a_z z^{-\frac{\alpha}{2}}}\cdot \abs{\tilde c_z z^{-\frac{\gamma}{2}}}\cdot \abs{\tilde c z^{-\frac{\gamma+1}{2}}} \cdot  \abs{z^\frac{\alpha+1}{2}}  dz \\
     &= C \mathcal I_c\cdot C(z^*)\cdot \int_0^{z^*}  \abs{\tilde a_z z^{-\frac{\alpha}{2}}}\cdot \abs{\tilde c_z z^{-\frac{\gamma}{2}}}  dz\leq C(z^*) \mathcal I_c^2 \mathcal I_a.\\
\end{split}
\end{equation}
}
\ifthenelse{\boolean{isSimplified}}{Combining with integration by parts, \eqref{tilde c vanishing corollary} and $\phi=e^{-z} \geq 0$, we compute {\begin{equation}
     \int_0^{z^*}-2\phi\tilde c_{z}  w \tilde c dz=-e^{-z^*}{z^*}^{-\gamma}\tilde c^2(s,z^*)+\int_0^{z^*} -\phi z^{-\gamma}\tilde c^2-\gamma\phi  z^{-\gamma-1}\tilde c^2 dz
    \leq 0.
\end{equation}}} {Using integration by parts, we estimate
\begin{equation}
\begin{split}
    \int_0^{z^*}-2\phi\tilde c_{z}  w \tilde c dz&=-e^{-z^*}{z^*}^{-\gamma}\tilde c^2(s,z^*)-(\lim_{z\rightarrow 0^+}-e^{-z}{z}^{-\gamma}\tilde c^2(s,z))+\int_0^{z^*} -\phi z^{-\gamma}c^2-\gamma\phi  z^{-\gamma-1}c^2 dz\\
    &\leq (\lim_{z\rightarrow 0^+}e^{-z}{z}^{-\gamma}\tilde c^2(s,z))=0.\\
\end{split}   
\end{equation}
This requires that
\begin{equation}\label{H_c-vanishing-on-c:non-diffusive}
    \tilde c=o(z^{\frac{\gamma}{2}}),
\end{equation}
which is satisfied by \eqref{paramer-choice-c}.}
{
Finally, we conclude that
\begin{equation}
    \frac{d}{ds} \mathcal I_c^2\leq(-\gamma+1+C(z^*) \mathcal I_a+O(s^{-\frac{h_a}{2}})) \mathcal I_c^2,
\end{equation}
and use the trappedness definition on $\mathcal I_a$ (Definition \ref{def:crit-beta=0:diffusive}) gives us the desired result.
}
\end{proof}

\subsection{Solution to the Interior Differential Inequalities-Choice of Parameters}\label{choice of parameters non-diffusive}

{The differential inequality in \eqref{interior-estimate-of-a-diffusive} can be further reduces to \begin{equation}
         \frac{d}{ds}\mathcal I_a^2+\mathcal I_a^2({\alpha-1}-\frac{1}{\sqrt{\alpha+1}}-\frac{4}{(\alpha+3)}\sqrt{\frac{3}{8(\alpha+1)}}-\kappa)\leq C(z^*)s^{-2},
    \end{equation}
    upon taking sufficiently large $s$ and applying Young's inequality. Gronwall's inequality suggests that the source term dominates and $\mathcal I_a^2$ decays, if the coefficient before the linear term $\mathcal I_a^2$ is strictly positive. More precisely, in this scenario the energy $\mathcal I_a^2$  shall decay at a speed close to $s^{-2}$. Therefore, in view of the bootstrap assumption $\mathcal{I}_a^2\lesssim s^{-h_a}$ \eqref{smooth:bd:etrap:diffusive}, in order for the bootstrap argument to work, we need $(s^{-1})^2=o(s^{-h_a})$. Symbolically, this suggests the parameters shall satisfy
    \begin{equation}\label{additional-requirement-on-alpha}
\begin{split}
   &{-\alpha+1}+\frac{1}{\sqrt{\alpha+1}}+\frac{4}{(\alpha+3)}\sqrt{\frac{3}{8(\alpha+1)}}<0 \iff \alpha>\alpha_0,\\
    & h_a<2,
\end{split}
\end{equation}
    where, we can check that $\alpha_0\approx 1.88415$ is the unique positive solution of
\begin{equation}\label{definition-of_alpha_knot}
     {-\alpha_0+1}+\frac{1}{\sqrt{\alpha_0+1}}+\frac{4}{(\alpha_0+3)}\sqrt{\frac{3}{8(\alpha_0+1)}}=0.
\end{equation}}
Similarly, in light of \eqref{I_c-bound:non-diffusive} and Gronwall's inequality, {$\mathcal I_c^2$ decays at most as fast as $e^{(-\gamma+1)s}$ (if we impose $\gamma>1$). Combining with bootstrap assumption $\mathcal I_c^2\lesssim e^{-h_cs}$, in order for the bootstrap argument \eqref{smooth:bd:etrap:diffusive} to work, we need to further impose}
\begin{equation}\label{additional-requirement-on-gamma}
    e^{(-\gamma+1)s}=o(e^{-h_cs})\iff h_c<\gamma-1.
\end{equation}

Summarizing all the conditions \eqref{initial-parameter-weights}, \eqref{initial-parameter-other}, \eqref{interior-parameter-a}, \eqref{paramer-choice-c},  \eqref{additional-requirement-on-alpha} and \eqref{additional-requirement-on-gamma}  on our parameters, we obtain the conditions on parameters
\ifthenelse{\boolean{isSimplified}}{\begin{alignat}{2}\label{final-parameter-conditions-diffusive}
    &\alpha_0<\alpha<3 \qquad && 2<\gamma<3\nonumber\\
        & 1<h_a<2\qquad && 0<h_c< \gamma-1\\
        & \frac{\alpha-1}{2}<\epsilon_a\leq \epsilon_c \qquad&&\frac{\gamma-1}{2}<\epsilon_c<1,\nonumber
\end{alignat}
which is a non-empty set since, {for example,}
\begin{equation*}
    (\alpha,\gamma,h_a,h_c,\epsilon_a,\epsilon_c)=(2,9/4,4/3,1/2,3/5,3/4)
\end{equation*}
satisfies the conditions above.}
{\begin{subequations}\label{final-parameter-conditions-diffusive}
    \begin{align}
        &\alpha_0<\alpha<3\\
        &1<\gamma<3\\
        & 1<h_a<2\\
        & 0<h_c< \gamma-1\\
        & \frac{\alpha-1}{2}<\epsilon_a\leq \epsilon_c\\
        & \frac{\gamma-1}{2}<\epsilon_c<1,
    \end{align}
\end{subequations}
where $\alpha_0\approx 1.88415$ satisfies \eqref{definition-of_alpha_knot}. One can check that 
\begin{equation*}
    (\alpha,\gamma,h_a,h_c,\epsilon_a,\epsilon_c)=(2,\frac{9}{4},4/3,1/2,3/5,3/4)
\end{equation*}
satisfies the conditions above. We now rigorously prove the previous observations, as well as partially finish the bootstrap argument.}

\begin{lemma}[{Solutions to interior differential inequalities}]\label{smooth:lemma:interior:diffusive}

Let $(\alpha,\gamma,h_a,h_c,\epsilon_a,\epsilon_c)$ be in the range of \eqref{final-parameter-conditions-diffusive}. For any $z^*\geq 1$ and $ \delta>0$, there exists a large self-similar time $s_0^*$ such that for any $s_0\geq s_0^*$, for any solution which is trapped on $[s_0,s_1]$, we have for $s\in [s_0,s_1]$:
\begin{eqnarray}\label{smooth:bd:int e1:diffusive}
\mathcal I_a^2(s)  \leq 2\delta^2 s^{-h_a}, \quad \mathcal I_c^2(s) \leq \delta^2 e^{-h_c s}.
\end{eqnarray}

\end{lemma}
\ifthenelse{\boolean{isSimplified}}{\begin{proof}
    Let $0<\kappa\ll 1$ be fixed, for which whose value will be chosen later. In light of \eqref{interior-estimate-of-a-diffusive} and Young's inequality, we observe that when $s$ is sufficiently large, $O(C(z^*)s^{-\frac{h_a}{2}})<\frac{\kappa}{2}$ and
    \begin{equation}\label{trivial-condition}
     C(z^*) \mathcal I_a(s^{-1}+ e^{(-1-\frac{h_c}{2}+\epsilon) s}+ C(z^*)e^{(-1-\frac{h_c}{2})s})\leq C(z^*)s^{-1}\mathcal I_a\leq C(z^*)s^{-2}+\kappa/2 \mathcal I_a^2.
    \end{equation}
   Now \eqref{interior-estimate-of-a-diffusive} is reduced to
    \begin{equation}
         \frac{d}{ds}\mathcal I_a^2+\mathcal I_a^2({\alpha-1}-\frac{1}{\sqrt{\alpha+1}}-\frac{4}{(\alpha+3)}\sqrt{\frac{3}{8(\alpha+1)}}-\kappa)\leq C(z^*)s^{-2},
    \end{equation}
    from which Gronwall-type lemma with source can be applied, in conjunction with Definition \ref{def:ini crit-beta=0:diffusive}, provided that $\kappa$ is small enough such that ${\alpha-1}-\frac{1}{\sqrt{\alpha+1}}-\frac{4}{(\alpha+3)}\sqrt{\frac{3}{8(\alpha+1)}}-\kappa>0$. (Recall also the range of $\alpha$: \eqref{final-parameter-conditions-diffusive}.)
    For the second inequality, again, we let $0<\kappa\ll 1$ be fixed and choose $s_0$ large enough so that $O(C(z^*)s^{-\frac{h_a}{2}})<\frac{\kappa}{2}$. For $s$ large enough, thanks to Definition \ref{def:ini crit-beta=0:diffusive}), with choice of \eqref{final-parameter-conditions-diffusive}, \eqref{I_c-bound:non-diffusive}  reads
    \begin{equation}
       \quad\frac{d}{ds}\mathcal I_{c}^2 \leq \mathcal I_c^2(-\gamma+1+\kappa),
    \end{equation}
    from which the desired result follows, provided that $-\gamma+1+\kappa<-h_c$.
\end{proof}}{
\begin{proof}
    Let $0<\kappa\ll 1$ be fixed, for which whose value will be chosen later. Choose $s_0$ large enough so that $O(s^{-\frac{h_a}{2}})<\frac{\kappa}{2}$. By Young's inequality, we see the linear source term in \eqref{interior-estimate-of-a-diffusive} can be bounded:
    \begin{equation}
    \begin{split}
          &C(z^*) I_a(s^{-1}+ e^{(-1-\frac{h_c}{2}+\epsilon) s}+ C(z^*)e^{(-1-\frac{h_c}{2})s})\\
          &\qquad\lesssim C(z^*)s^{-1}\\
          &\qquad\leq C(z^*) (\frac{s^{-2}}{2\frac{\kappa}{C(z^*)}}+\frac{\frac{\kappa}{C(z^*)}}{2}I_a^2)\\
          &\qquad\leq C(z^*)s^{-2}+\kappa/2 I_a^2,
    \end{split}
    \end{equation}
    for $s_0$ large enough. Now \eqref{interior-estimate-of-a-diffusive} is reduced to
    \begin{equation}
         \frac{d}{ds}\mathcal I_a^2+\mathcal I_a^2({\alpha-1}-\frac{1}{\sqrt{\alpha+1}}-\frac{4}{(\alpha+3)}\sqrt{\frac{3}{8(\alpha+1)}}-\kappa)\leq C(z^*)s^{-2},
    \end{equation}
    from which Gronwall-type \eqref{Gronwall-with-source} of lemma can be applied (Recall \eqref{final-parameter-conditions-diffusive}), provided that $\kappa$ is small enough such that ${\alpha-1}-\frac{1}{\sqrt{\alpha+1}}-\frac{4}{(\alpha+3)}\sqrt{\frac{3}{8(\alpha+1)}}-\kappa>0$.

    For the second inequality, again, we let $0<\kappa\ll 1$ be fixed and choose $s_0$ large enough so that $O(C(z^*)s^{-\frac{h_a}{2}})<\frac{\kappa}{2}$. Now, for $s$ large enough, with choice of \eqref{final-parameter-conditions-diffusive}, \eqref{I_c-bound:non-diffusive}  reads
    \begin{equation}
       \quad\frac{d}{ds}\mathcal I_{c}^2 \leq \mathcal I_c^2(-\gamma+1+\kappa) 
    \end{equation}
    and  Gronwall's lemma \eqref{Gronwall-with-source} gives
    \begin{equation}
    \begin{split}
        & \mathcal I_c^2(s)\leq \mathcal I_c^2(s_0)e^{(-\gamma+1+\kappa)(s-s_0)}\leq \delta^2 e^{-h_c s_0} e^{-h_c(s-s_0)}=\delta^2 e^{-h_c s}\\
    \end{split}
    \end{equation}
    provided that $-\gamma+1+\kappa<-h_c$.
\end{proof}}

\subsection{Exterior Estimates}
{The following estimate improve the bounds of energies $\mathcal E_a^2$ and $\mathcal E_c^2$ in the exterior $[z^*,\nu^{-1}]$ of the domain. We shall use maximum principle to prove the following lemma.}
\begin{lemma}[Exterior Estimate-Exterior Bootstrap Argument]\label{smooth:lemma:main-critical:diffusive}
Let $(\alpha,\gamma,h_a,h_c,\epsilon_a,\epsilon_c)$ be in the range of \eqref{final-parameter-conditions-diffusive}. There exists $\bar z^*\geq 1$, and for any $z^*\geq \bar z^*$, a $\delta^*>0$ {exists} such that for $0<\delta\leq \delta^*$ the following holds true. There exists $s_0^*$ large enough such that if a solution is trapped on $[s_0,s_1]$ with $s_0\geq s_0^*$, for any time $s_0\leq s\leq s_1$ we have

\begin{eqnarray}\label{bd:int e-critical2:diffusive}
\mathcal E_a^2(s)\leq \frac14 s^{-{h_a}}, \quad \mathcal E_c^2(s)\leq \frac14 e^{-h_cs}.
\end{eqnarray}

\end{lemma}

\begin{proof}

The proof relies on the maximum principle. We rewrite \eqref{equation:epsilon-a} and {\eqref{equation:epsilon-c}} as
\begin{equation}\label{smooth:equation:epsilon-2}
    \begin{split}
        &\tilde a_s +\mathcal L_a \tilde a    = F_a, \qquad  \tilde c_s +\mathcal L_c \tilde c    = 0.
    \end{split}
    \end{equation}
where the transport operator $\mathcal L_a$ and $\mathcal L_c$ (note that $\mathcal L_a$ has nonlinear parts) and the source term $F_a$ are:
\begin{align*}
&\mathcal L_a v=  - \frac{\lambda_s}{\lambda}v - \frac{\nu_s}{\nu} z v_z - 2\phi v + \partial_z^{-1} \phi v_z    - \tilde a v + \partial_z^{-1} \tilde a v_z ,\\
&\mathcal L_c v =   - \frac{\lambda_s}{\lambda}v - \frac{\nu_s}{\nu} z v_z - 2\phi v + \partial_z^{-1} \phi v_z    - 2\tilde a v + \partial_z^{-1} \tilde a v_z,\\
&F_a = -\partial_z^{-1} \tilde a \phi' +(\frac{\lambda_s}{\lambda} +1)\phi +  \frac{\nu_s}{\nu}z\phi' - 2\nu \int_0^{\frac{1}{\nu}} (\phi+ \tilde{a})^2(z) dz
        \\
         &\qquad  + \lambda\left( - \nu\partial_z^{-1} \tilde{c}  + \int_0^{\frac{1}{\nu}} \nu^2 \partial_z^{-1} \tilde{c}(z) dz\right).
\end{align*}

\noindent \textbf{Step 1}. \emph{{Supersolutions} for $\pa_s+\mathcal L_a$ and $\pa_s+\mathcal L_c$ on $[z^*,\nu^{-1}]$}. We introduce
$$
f(s,z)= \frac 12 s^{-h_a/2}, \qquad g(s,z) = \frac 12 e^{-h_c s/2}
$$
and claim that there exists $z^*$ large enough such that for $s_0$ large enough {and $z\geq z^*$}, for all $ s_0\leq s \leq s_1$ :
\be \label{smooth:exterior:id:supersolution:diffusive}
(\pa_s +\mathcal L_a) f\geq \frac{1}{4} s^{-h_a/2}, \qquad (\pa_s +\mathcal L_c) g\geq \frac{1-\frac{h_c}{2}}{4} e^{-h_c s/2}.
\ee
To prove \eqref{smooth:exterior:id:supersolution:diffusive}, we compute using \eqref{smoothmodulationequations:diffusive} and \eqref{asy on lambda derivative}:
\be
\begin{split}
    (\pa_s +\mathcal L_a) f &= (-\frac{h_a}{4}s^{-h_a/2-1})+(-\frac{\lambda_s}{\lambda}-2e^{-z}-\tilde{a})\frac{1}{2}s^{-h_a/2}\\
    &\geq (-\frac{h_a}{2}s^{-1}+1-Cs^{-1}-2e^{-z^*}-C(z^*)s^{-h_a/2})\frac{1}{2}s^{-h_a/2}\\
    &\geq \frac{1}{4} s^{-h_a/2}
\end{split}
\ee
and
\be
\begin{split}
    (\pa_s +\mathcal L_c) g &= (-\frac{h_c}{4}e^{-h_c s/2})+(-\frac{\lambda_s}{\lambda}-2e^{-z}-2\tilde{a})\frac{1}{2}e^{-h_cs/2}\\
    &\geq (-\frac{h_c}{2}+1-Cs^{-1}-2e^{-z^*}-C(z^*)s^{-h_a/2})\frac{1}{2}e^{-h_cs/2}\\
    &\geq \frac{1-\frac{h_c}{2}}{4} e^{-h_cs/2} (> 0 \text{ since } h_c<\gamma-1<3-1=2.)
\end{split}
\ee
which implies \eqref{smooth:exterior:id:supersolution:diffusive} upon first taking $z^*$ large enough and then $s_0^*$ large enough. \\

\noindent \textbf{Step 2}. \emph{Estimate for the source term}. For any $\kappa>0$, we claim that for $z^*$ large enough and then for $\delta$ small enough, for all $s_0\leq s \leq s_1$ and  $z\in [z^*,\nu^{-1}]$:
\be \label{smooth:exterior:bd:source:diffusive}
|F_a(s,z)|\leq \kappa{s^{-\frac{h_a}{2}}}.
\ee
Let $0<\kappa$ be fixed. We inject the improved bootstrap bound \eqref{smooth:bd:int e1:diffusive} in the computation \eqref{smooth:bd:varepsiloninterior:diffusive} and get:
\be \label{smooth:exterior:bd:interiorvarepsilon:diffusive}
\begin{split}
    |\tilde a (z)|&\leq \frac{z^\frac{\alpha+1}{2}}{(\alpha+1)^\frac{1}{2}} \mathcal I_a(s)\leq C \delta s^{-h_a/2} {z^*}^\frac{\alpha+1}{2} \quad \mbox{for }z\in [0,z^*],\\
    |\tilde c (z)|&\leq \mathcal I_c(s)  \frac{z^\frac{\gamma+1}{2}}{(\gamma+1)^\frac{1}{2}}\leq C \delta e^{-h_c/2} {z^*}^\frac{\alpha+1}{2} \quad \mbox{for }z\in [0,z^*],\\
\end{split}
\ee
where we have taken $s$ (may depend on $z^*$ and $\delta$) large enough.
Combining with 
\begin{equation}\label{a-b-exterior:diffusive}
\begin{split}
    |\tilde a(z)|&\leq s^{-\frac{h_a}{2}} \qquad\text{for $z\in [z^*,\nu^{-1}]$}
\end{split}  
\end{equation}
using $\phi(z)=e^{-z}$:
\begin{equation} \label{smooth:exterior:bd:inter1:diffusive}
\begin{split}
    |\partial_z^{-1} \tilde a \phi'| &\leq (\|\tilde a\|_{L^\infty[0,z^*]}+\|\tilde a\|_{L^\infty[z^*,\nu^{-1}]}) z |\phi'(z)|\\
    &\leq \Big(({z^*}^\frac{\alpha+1}{2}) \delta s^{-\frac {h_a}{2}}+s^{-\frac {h_a}{2}} \Big)z^*e^{-z^*} \leq  \kappa s^{-\frac{h_a}{2}} \qquad\text{for $z\in [z^*,\nu^{-1}]$},
\end{split}
\end{equation}
where $\kappa>0$ and we have taken $z^*$ large enough such that all of $z^*e^{-z^*}<\kappa/2$, and that $z^{1}e^{-z}$ is decreasing for $z>z^*$ are satisfied, and $\delta$ small enough such that $\delta {z^*}^\frac{\alpha+1}{2}<1$. Next, using \eqref{smooth:bd:integral:diffusive}, \eqref{estimate-on-primitive-of-c-diffusive-1}, \eqref{smoothmodulationequations:diffusive}, and \eqref{a-b-exterior:diffusive}, for all $z\geq z^*$ and $z^*$ large enough:
\be \label{smooth:exterior:bd:inter2:diffusive}
\begin{split}
   &\left|(\frac{\lambda_s}\lambda +1)\phi \right|  \leq Cs^{-1}\leq \kappa s^{-\frac{h_a}{2}},
   \\
   &\left|\frac{\nu_s}\nu z\phi'  \right|\leq Cs^{-1}\cdot1 \leq Cs^{-1}\leq \kappa s^{-\frac{h_a}{2}},
   \\
   &\left|- 2\nu\Big(\int_0^{\frac{1}{\nu}} (\phi+\tilde a)^2(z)dz\Big)\right|\leq Cs^{-1}\leq \kappa s^{-\frac{h_a}{2}},\\
   &\left|\lambda\nu \pa_z^{-1} \tilde c\right|\leq Ce^{-s}s^{1}e^{-\frac{h_c}{2}s}  \leq Cs^{-1}\leq \kappa s^{-\frac{h_a}{2}},
   \\
   &\left|\lambda\int_0^{\frac{1}{\nu}} \nu^2 \partial_z^{-1} \tilde{c}(z) dz\right|\leq e^{(-1-\frac{h_c}{2}+\epsilon) s} \leq Cs^{-1}\leq \kappa s^{-\frac{h_a}{2}},
   \\
\end{split}
\ee
where we have used $h_c>0$ and $s_0^*$ is taken to be large enough that may depend on $z^*$.
Combining \eqref{smooth:exterior:bd:inter1:diffusive} and \eqref{smooth:exterior:bd:inter2:diffusive} shows  \eqref{smooth:exterior:bd:source:diffusive}.\\

\noindent \textbf{Step 3}. \emph{End of the proof}. We introduce
\be \label{smoothexterior:def:fpm:diffusive}
{f^\pm= \pm  f-\tilde a , \quad g^\pm = \pm g-\tilde c},
\ee
and try to apply maximum principle on them.
Using \eqref{smooth:equation:epsilon-2}, \eqref{smooth:exterior:id:supersolution:diffusive} and \eqref{smooth:exterior:bd:source:diffusive} one obtains that for $s_0\leq s\leq s_1$ and $z\in [z^*,\nu^{-1}]$:
\be \label{smooth:exterior:inter1:diffusive}
\begin{split}
  &(\pa_s+\mathcal L_a)f^+= (\pa_s+\mathcal L_a)f - F_a\geq (\frac{1}{4}-\kappa) s^{-\frac{h_a}{2}}\geq 0 \quad \mbox{and similarly} \quad (\pa_s+\mathcal L_a)f^-\leq 0,
  \\
  &(\pa_s+\mathcal L_c)g^+ = (\pa_s+\mathcal L_c)g \geq \frac{1-\frac{h_c}{2}}{4} e^{-h_cs/2} \geq 0\quad\mbox{and similarly} \quad (\pa_s+\mathcal L_c)g^-\leq 0,
\end{split}
\ee
provided that $\kappa \leq 1/4$ and $h_c\leq 2$.
Thanks to \eqref{smooth:modulation:inter2:diffusive}, one has { the following estimate for transport speed:}
\begin{equation}
\begin{split}
    -\frac{\nu_s}{\nu}z^* + \int_0^{z^*} (\phi+\tilde a)(\tilde z) d\tilde z &\geq -\frac{\nu_s}{\nu^2}z^* \nu + 1-e^{-z^*}-\norm{\tilde{a}}_{L^\infty[0,\nu^{-1}]}z^*\\
    &\geq (1-C(z^*)s^{-h_a+1})z^*(Cs^{-1})+1-e^{-z^*}-C(z^*)z^*s^{-h_a/2}\\
    &\geq 1-e^{-z^*}-(Cs^{-1}-Cs^{-1} C(z^*) s^{-h_a+1}-C(z^*) s^{-h_a/2})z^*\\
    &\geq 1-e^{-z^*}-Cs^{-1}z^*\geq 0,
\end{split}  
\end{equation}
where we have first taken $z^*$ sufficiently large, then $s$ sufficiently large. From this we know that the particles are always moving from region $0\leq z\leq z^*$ to $z^*\leq z \leq \frac1\nu$.
At the boundary $z=z^*$ one has using \eqref{smooth:exterior:bd:interiorvarepsilon:diffusive} that:
\be \label{smooth:exterior:inter2}
\begin{split}
   &f^+(s,z^*)= \frac{s^{-\frac{h_a}{2}}}{2}-\tilde a (s,z^*)\geq (\frac{1}{2}-C \delta  {z^*}^\frac{\alpha+1}{2}) s^{-h_a/2}\geq 0 \quad \mbox{and similarly} \quad f^-(s,z^*)\leq 0, 
   \\
   &g^+(s,z^*)= \frac{e^{-\frac{h_c}{2}s}}{2}-\tilde c (s,z^*)\geq (\frac{1}{2}-C\delta  {z^*}^{\frac{\gamma+1}{2}}  )e^{-\frac{h_c}{2}s} \geq 0 \quad \mbox{and similarly} \quad g^-(s,z^*)\leq 0
\end{split}
\ee
provided $\delta$ is small enough depending on $z^*$. At {the} initial time $s=s_0$, we have using \eqref{bd:eini-beta=0:diffusive} that for all $z\in [z^*,\nu_0^{-1}]$:
\begin{equation}\label{smooth:exterior:inter3:diffusive}
\begin{split}
&f^+(s_0,z)\geq \frac{s_0^{-\frac{h_a}{2}}}{2}-\| \tilde a_0\|_{L^\infty [z^*,\nu_0^{-1}]}\geq \frac{s_0^{-\frac{h_a}{2}}}{2}-\frac{s_0^{-\frac{h_a}{2}}}{4}\geq 0,  \quad \mbox{and similarly} \quad f^-(s_0,z)\leq 0,
\\
&g^+(s_0,z)\geq \frac{e^{-\frac{h_c}{2}s_0}}{2}-\| \tilde c_0\|_{L^\infty [z^*,\nu_0^{-1}]}\geq \frac{e^{-\frac{h_c}{2}s_0}}{2}-\frac{e^{-\frac{h_b}{2}s_0}}{4}\geq 0,  \quad \mbox{and similarly} \quad g^-(s_0,z)\leq 0.
\end{split}
\end{equation}
At the point $z(s)=\frac1{\nu(s)}$, thanks to the boundary condition \eqref{equation:epsilon-bc}, the characteristics of the full transport field stay on
the boundary (or the characteristics curves passing any points in the interior of the domain cannot be issued from the boundary $z=\nu^{-1}(s)$, so it can only be issued from $(s_0,\nu^{-1}(s_0))$) since
\begin{equation}\label{partical-stay}
    \frac{d}{ds}\frac1{\nu(s)}= -\frac{\nu_s}{\nu^2} = -\frac{\nu_s}{\nu}\nu^{-1} + \underbrace{\int_0^{\frac1\nu} (\phi+\varepsilon) dz}_{0}=\text{transport coefficient}.
\end{equation}
This together with \eqref{smooth:exterior:inter3:diffusive} imply that $f^+(s,\frac1\nu), g^+(s,\frac1\nu) \geq 0$ and $f^-(s,\frac1\nu), g^-(s,\frac1\nu)\leq 0$.
Therefore, in view of \eqref{smooth:exterior:bd:inter1:diffusive}, \eqref{smooth:exterior:bd:inter2:diffusive} and \eqref{smooth:exterior:inter3:diffusive} one can apply the maximum principle and obtain that $f^+(s,z), g^+(s,z)\geq 0$ and $f^-(s,z), g^-(s,z)\leq 0$ for all $s_0\leq s \leq s_1$ and $z^*\leq z \leq \nu^{-1}$. By the definition \eqref{smoothexterior:def:fpm:diffusive}  of $f^\pm$ and $g^\pm$ this implies the desired estimate \eqref{bd:int e-critical2:diffusive} and completes the proof of the Lemma.

\end{proof}

\subsection{Conclusion}
In view of Lemma \ref{lemma:smoothmodulation:diffusion}, Lemma \ref{smooth:lemma:interior:diffusive} and Lemma \ref{smooth:lemma:main-critical:diffusive}, Definition \ref{def:crit-beta=0:diffusive} and Definition \ref{def:ini crit-beta=0:diffusive}, the bootstrap argument is completed.

\begin{proposition} \label{smooth:pr:bootstrap:diffusive}

Let $(\alpha,\gamma,h_a,h_c,\epsilon_a,\epsilon_c)$ be in the range of \eqref{final-parameter-conditions-diffusive}. There exist universal constants $z^*\geq 1$, $\delta(z^*)>0$ and $s_0^*(z^*,\delta)\geq 0$ such that the following holds true. For all $s_0\geq s_0^*$, any solution of \eqref{equation:epsilon} which is initially close to the blowup profile in the sense of Definition \ref{def:ini crit-beta=0:diffusive} is trapped on $[s_0,+\infty)$ in the sense of Definition \ref{def:crit-beta=0:diffusive}.

\end{proposition}
\begin{proof}
       Let $s^*=\sup \set{s_1\geq s_0: \text{The solution is trapped in }[s_0,s_1]}$. If $s^*=\infty$, then we are done. Otherwise, at $s=s^*<\infty$, by Lemma \ref{lemma:smoothmodulation:diffusion}, Lemma \ref{smooth:lemma:interior:diffusive} and Lemma \ref{smooth:lemma:main-critical:diffusive}, all the inequalities in Definition \ref{def:crit-beta=0:diffusive} are strictly satisfied, and we arrive at a contradiction.
\end{proof}

\section{Bootstrap Argument for the case of Diffusive Temperature}
\ifthenelse{\boolean{isSimplified}}{In this section, we perform a bootstrap argument on the diffusive case ($\sigma=1$). Recall that earlier we have imposed the orthogonality-like condition
\begin{equation*}
    \tilde a(s,z)=o(z)\text{ as }z\rightarrow 0^+.
\end{equation*}
In light of Corollary \ref{Invariance-of-vanishing speed}, we further impose  {time-invariant vanishing-speed condition at $z=0$: That is, for each $s\geq s_0$ ($s_0$ is the start time defined below),}
\begin{equation}\label{decay speed condition nc}
    \tilde a_z(s,z)=o(z^{\epsilon_a}),\, \tilde c=o(z^{\epsilon_c}),\text{ as }z\rightarrow 0^+,\,\text{ for } 0<\epsilon_a\leq \epsilon_c<1,
\end{equation}
{where the range of $\epsilon_a$ and $\epsilon_c$ shall be refined as our analysis progresses.}}{In this section, we study the case $\sigma=1$ and introduce the weights whose parameters are in the range of
\begin{equation}\label{initial-parameter-weights-nc}
    \text{$z^{-\alpha}$ and $z^{-k\eta_0}$},\quad 0<\alpha<3,\,0<k<2,\text{ and } \eta_0\geq 2,\, \eta_0\in \N,
\end{equation}
as well as further impose the vanishing-speed boundary conditions
\begin{equation}\label{decay speed condition nc}
    \tilde a_z(s,z)=o(z^{\epsilon_a}),\, \tilde c=o(z^{\epsilon_c}),\text{ as }z\rightarrow 0^+,\,\text{ for } 0<\epsilon_a\leq \epsilon_c<1.
\end{equation}
Recall that earlier we have imposed the orthogonality-like condition
\begin{equation*}
    \tilde a(s,z)=o(z)\text{ as }z\rightarrow 0^+.
\end{equation*}
We shall refine the range of these parameters by the end of this section.}

{As in the last section, we will be writing $\partial_z^{-1} f(s,z)=\int_0^z f(s,\tilde z)d\tilde z$ and $\partial_z^{-1}f g=(\partial_z^{-1}f) g$.}

\subsection{Definitions}
\ifthenelse{\boolean{isSimplified}}{Initially, the weights appear in the following definitions satisfy
\begin{equation}\label{initial-parameter-weights-nc}
    \text{$z^{-\alpha}$ and $z^{-k\eta_0}$},\quad 0<\alpha<3,\,0<k<2,\text{ and } \eta_0\geq 2,\, \eta_0\in \N.
\end{equation}}{}
\begin{definition}[Initial closeness]\label{def:ini crit-beta=0:diffusive-nc}
Let $\lambda_0^*>0$ and $\delta>0$. We say that $a_0$ is initially close to the blowup profile if there exists $0<\lambda_0<{\lambda_0^*}$ and $\nu_0>0$ such that the decomposition \eqref{decomposition-nc} satisfies:
\begin{itemize}
\item[(i)] \emph{Initial values of the modulation parameters} (note that this fixes the value of $s_0$):
\begin{eqnarray}
\label{bd:parametersini-beta=0-nc}
 \lambda_0 = s_0e^{-s_0}, \ \  \frac{1}{N_0 s_0}\leq \nu_0 \leq  \frac{N_0}{s_0}.
\end{eqnarray} 
\item[(ii)] \emph{Compatibility condition for the initial perturbation}. $\tilde a_0\in C^2 ([0,\frac1{\nu_0}))$ satisfies the boundary conditions \eqref{smooth:orthogonality-nc} and the integral condition \eqref{equation:epsilon-bc}.
\item[(iii)] \emph{Initial size of the remainder in the self-similar variables}. For some small number $\delta > 0$:
\begin{eqnarray}\label{bd:eini-beta=0:diffusive-nc}
&\mathcal I_a^2(s_0) = \int_0^{z^*} z^{-\alpha}(\pa_z\tilde a_0)^2\;dz  < \delta^2 s_0^{-h_a},
\ \ \mathcal E_a^2(s_0) =\sup_{z^*\leq z\leq\frac1{\nu_0}}|\tilde a_0|^2 < \frac{1}{16}s_0^{-h_a}
\\
&\mathcal T_{k\eta_0,2\eta_0}^{2\eta_0}(s_0)=\int_0^{\nu^{-1}(s_0)} z^{-k\eta_0}\tilde c^{2\eta_0}<\frac{1}{4}e^{-\eta_0 l s_0}
\end{eqnarray} 
\end{itemize}
\end{definition}

\begin{definition}[Trapped on $\left(s_0,s_1\right)$ ]\label{def:crit-beta=0:diffusive-nc}
Let $z^*\geq 1$ and $\delta>0$. We say that a solution $\tilde{a}(s,z)$ is trapped on $[s_0,s_1]$ with $s_0<s_1\leq \infty$, if it satisfies the properties of Definition \ref{def:ini crit-beta=0:diffusive-nc} at time $s_0$ and if for all $s\in [s_0,s_1]$, $a(s,z)$ can be decomposed as in \eqref{decomposition-nc} with:
\begin{itemize}
\item[(i)] \emph{Values of the modulation parameters:}
\begin{eqnarray}\label{bd:parameterstrap-critical:diffusive-nc}
\frac 1{M} se^{- s}< \lambda < M se^{- s}, \ \ \frac {1}{Ns} 
<\nu < \frac{N}{s}.
\end{eqnarray} 
\item[(ii)] \emph{Decay in time of the remainder in the self-similar variables:}
\begin{equation}\label{smooth:bd:etrap:diffusive-nc}
    \begin{split}
    &\mathcal I_a^2(s) = \int_0^{z^*}z^{-\alpha}\tilde a_z^2\;dz  < s^{-h_a}, \qquad \mathcal E_a^2(s) = \sup_{z^*\leq z\leq\frac1{\nu(s)}}|\tilde a|^2 < s^{-h_a}
    \\
    &\mathcal T_{k\eta_0,2\eta_0}^{2\eta_0}(s)=\int_0^{\nu^{-1}(s)} z^{-k\eta_0}\tilde c^{2\eta_0}<e^{-\eta_0 l s}
\end{split}
\end{equation}
\end{itemize}
\end{definition}
{
\begin{remark}
    Notice that the choice of $\eta_0\in \N$ ensures that $\mathcal T_{k\eta_0,2\eta_0}(s)$ is differentiable.
\end{remark}
}
\begin{remark}
    Notice that {the weights appearing in $\mathcal I_a$ and $\mathcal T_{k\eta_0,2\eta_0}$ may} have a singularity at zero. To remove such singularity, one type of sufficient conditions to impose is that {as $z\rightarrow 0^+$,}
    \begin{equation}\label{singularity-at-definition-nc}
        \begin{split}
            &z^{-\alpha}\tilde a_z^2=o(z^{-1+\epsilon_1})\iff \tilde a_z=o(z^{\frac{\alpha-1+\epsilon_1}{2}})\impliedby \tilde a_z=o(z^{\frac{\alpha-1}{2}+\epsilon_1})\\
             &z^{-k\eta_0}\tilde c^{2\eta_0}=o(z^{-1+\epsilon_3}) \iff \tilde c^{2\eta_0}=o(z^{\frac{k\eta_0-1+\epsilon_3}{2\eta_0}})\impliedby \tilde c=o(z^{\frac{k\eta_0-1}{2\eta_0}+\epsilon_3}),
        \end{split}
    \end{equation}
    for some $\epsilon_1,\epsilon_2,\epsilon_3>0$, by our initial choice of parameters \eqref{initial-parameter-weights-nc}.
\end{remark}
Initially, we further make the following assumptions to keep computation simple 
\ifthenelse{\boolean{isSimplified}}{\begin{alignat}{2}\label{initial-parameter-other-nc}
    &1<h_a\leq 2\qquad &&0<l<2\nonumber\\
    &M=2\qquad &&N=4>N_0=3\geq 2\\
    &0<\frac{\alpha-1}{2}<\epsilon_a\leq \epsilon_c\qquad&& 0<\frac{k}{2}-\frac{1}{2\eta_0}<\epsilon_c<1.\nonumber.
\end{alignat}
{As usual, these parameters shall be further refined by the end of this section.}}{\begin{subequations}\label{initial-parameter-other-nc}
\begin{align}
    &1<h_a\leq 2\\
    &0<l<2\\
    &M=2,N>N_0\geq 2\\
    &\frac{\alpha-1}{2}<\epsilon_a\leq \epsilon_c,\, \frac{k}{2}-\frac{1}{2\eta_0}<\epsilon_c<1.
\end{align}
\end{subequations}
whereas all the additional parameters shall be chosen as our analysis progresses.
\begin{remark}
    Notice that the choice of $\eta$ ensures $T_{k\eta_0,2\eta_0}(s)$ is differentiable and positive.
\end{remark}
}

\subsection{A-priori Estimates}
{The following a-priori estimates are derived from the bootstrap assumptions.}
\begin{lemma}\label{lemma:prelim-est:diffusive-nc}
Assume that \eqref{initial-parameter-weights-nc} and \eqref{initial-parameter-other-nc} are satisfied. For any $\delta>0,z^*\geq 1,\epsilon>0$, for $s_0^*$ large enough, if $a$ is trapped on $[s_0,s_1]$ then for all $s_0^*\leq s_0\leq s \leq s_1$:
\be \label{smooth:bd:varepsilonLinfty:diffusive-nc}
\begin{split}
&|\tilde a(z)|\leq \mathcal I_a(s) (\frac{z^{\alpha+1}}{\alpha+1})^\frac{1}{2},\, |\pa^{-1}_z\tilde a (z)| \leq \mathcal I_a(s) \frac{2}{(\alpha+3)\sqrt{\alpha+1}}z^{\frac{\alpha+3}{2}} \text{ for }0\leq z\leq z^*\\
&\| \tilde a (s) \|_{L^\infty([0,\nu^{-1}])} \leq C({z^*}^{\frac{\alpha+1}{2}}) s^{-\frac{h_a}{2}},\\
\end{split}
\ee
and
\be \label{smooth:bd:integral:diffusive-nc}
\begin{split}
  &\nu  \int_0^{\frac{1}{\nu}} (\phi + \tilde a)^2 (z) dz  \leq  s^{-1} N,
  \quad \left|\int_0^{\frac{1}{\nu}} \nu^2 \partial_z^{-1} \tilde{c}(z) dz\right|\leq e^{(-l/2+\epsilon)s}.
\end{split}
\ee
\end{lemma}
\begin{proof}
\ifthenelse{\boolean{isSimplified}}{The proof of Equation \eqref{smooth:bd:varepsilonLinfty:diffusive-nc} and the first equation of \eqref{smooth:bd:integral:diffusive-nc} is similar to its corresponding non-diffusive case. As for the second equation of \eqref{smooth:bd:integral:diffusive-nc}, Hölder's inequality and \eqref{initial-parameter-weights-nc} yield}{From the vanishing boundary condition, Cauchy-Schwarz, and \eqref{smooth:bd:etrap:diffusive}, for $0<z\leq z^*$:
\be \label{smooth:bd:varepsiloninterior:diffusive-nc}
\begin{split}
    &|\tilde a(z)|=\left|\int_0^z \pa_z \tilde a d\tilde z\right|\leq \mathcal I_a \sqrt{\int_0^z z^\alpha d\tilde z}=I_a(s) (\frac{z^{\alpha+1}}{\alpha+1})^\frac{1}{2},\\
    &|\pa^{-1}_z\tilde a (z)|=\left|\int_0^z \tilde a d\tilde z\right| \leq I_a(s) \frac{2}{(\alpha+3)\sqrt{\alpha+1}}z^{\frac{\alpha+3}{2}}\\
\end{split}
\ee
and thus we have
\begin{equation}
     \norm{\tilde a(z)}_{L^\infty(0,z^*)}\leq C(z^*) s^{-\frac{h_a}{2}}
\end{equation}
This, combined with the bounds of $\mathcal E_a$ shows the first equation of \eqref{smooth:bd:varepsilonLinfty:diffusive}.
Thanks to \eqref{bd:parameterstrap-critical:diffusive-nc}--\eqref{smooth:bd:varepsilonLinfty:diffusive-nc}, we have
\begin{align*}
\nu\int_0^{\frac{1}{\nu}} (\phi + \tilde a)^2 (z) dz&= \nu\left( \int_0^\infty \phi^2 (z)dz -  \int_{\nu^{-1}}^\infty \phi^2 (z)dz+ 2 \int_0^{\frac{1}{\nu}} \phi  \tilde a (z) dz +\int_0^{\nu^{-1}}  \tilde a^2 dz \right)
\\
&= \nu\left( \frac12 +O(\| \tilde a \|_{L^\infty([0,\nu^{-1}])})+O(\nu^{-1}\| \tilde a \|_{L^\infty([0,\nu^{-1}])}^2) \right) 
\\
&= Ns^{-1}\left(\frac12  +O(C(z^*)s^{-\frac{h_a}{2}},C(z^*)s^{-{h_a}+1})\right)=\frac{N}{2} s^{-1}+O(C(z^*)s^{-\frac{h_a}{2}-1},C(z^*)s^{-{h_a}}) ,
\end{align*}Using Hölder's inequality and \eqref{initial-parameter-weights-nc}, we have }
    \begin{equation}\label{estimate-on-primitive-of-c-diffusive-1-nc}
        \begin{split}
            \abs{\partial_z^{-1} \tilde c}&\leq \int_0^z \abs{\tilde c z^{-\frac{k}{2}} z^{\frac{k}{2}}}\\
            &\leq (\int_0^{\nu^{-1}} \tilde c^{2{\eta_0}} z^{-k{\eta_0}})^\frac{1}{2{\eta_0}} (\int_0^{\nu^{-1}} z^{\frac{2k{\eta_0}}{4{\eta_0}-2}})^\frac{2{\eta_0}-1}{2{\eta_0}}\\
            &\leq \mathcal T_{k{\eta_0},2{\eta_0}} (\nu^{-1})^{(\frac{2k{\eta_0}}{4{\eta_0}-2}+1)\frac{2{\eta_0}-1}{2{\eta_0}}} (\frac{1}{\frac{2k{\eta_0}}{4{\eta_0}-2}+1})^{1-\frac{1}{2{\eta_0}}}\\
            &=1^{1-\frac{1}{2{\eta_0}}} {(Cs)}^{\frac{(k+2){\eta_0}-1}{2{\eta_0}}}e^{-\frac{l}{2}s}\\
            &\leq C^{\frac{k}{2}+(1-\frac{1}{2{\eta_0}})} s^{\frac{k+2}{2}}e^{-\frac{l}{2}s}\leq C s^2 e^{-\frac{l}{2}s}
        \end{split}
    \end{equation}
    Thus, for $s$ large enough, we have
    \begin{equation}\label{estimate-on-primitive-of-c-diffusive-2-nc}
        \left|\int_0^{\frac{1}{\nu}} \nu^2 \partial_z^{-1} \tilde{c}(z) dz\right|\leq \nu^2 \frac{1}{\nu} \norm{\partial_z^{-1} \tilde c}_{L^\infty(0,\nu^{-1})}\leq C s^{1}e^{-\frac{l}{2}s}\leq e^{(-l/2+\epsilon)s}
    \end{equation}
\end{proof}

\begin{lemma}[Modulation Equations] \label{lemma:smoothmodulation:diffusion-nc}
Assume that \eqref{initial-parameter-weights-nc} and \eqref{initial-parameter-other-nc} are satisfied. For any $z^*\geq 1$, $\epsilon>0$, and $ \delta>0$, there exists a large self-similar time $s_0^*$ such that for any $s_0\geq s_0^*$, for any solution which is trapped on $[s_0,s_1]$, we have for $s\in [s_0,s_1]$:
\begin{eqnarray}\label{smoothmodulationequations:diffusive-nc}
\left| \frac{\lambda_s}{\lambda}+1\right|\leq C s^{-1}, \qquad \left| \frac{\nu_s}{\nu} \right|\leq C s^{-1},\qquad \frac{\lambda_s}{\lambda}=-1+O(C(z^*)s^{-1})
\end{eqnarray}
for $C>0$ independent of the bootstrap constants, and
\begin{eqnarray}\label{smooth:bd:boostrap improved parameters:diffusive-nc}
\frac{1}{1+\epsilon} se^{-s}\leq \lambda \leq  (1+\epsilon) se^{-s}, \qquad \frac{1}{(N_0+\epsilon)s} \leq \nu \leq\frac{N_0+\epsilon}{s}.
\end{eqnarray}
Moreover, if $s_1=\infty$ then there exists a constant $\tilde \lambda_\infty>0$ such that
\begin{eqnarray}\label{smooth:bd:boostrap improved parameters:diffusive-nc}
\lambda= s e^{-s}\big(\tilde \lambda_\infty + O(s^{-h_a+1}) \big) ,\quad
\nu=\frac{1}{s}+O(s^{-h_a}).
\end{eqnarray}
\end{lemma}

\ifthenelse{\boolean{isSimplified}}{\begin{proof}
    This proof is similar to its non-diffusive version: Lemma \ref{lemma:smoothmodulation:diffusion}.
\end{proof}}{\begin{proof}
\noindent \textbf{Step 1}.  \eqref{smoothmodulationequations:diffusive-nc} is an immediate consequence of Lemma \ref{lemma:prelim-est:diffusive-nc} by recalling \eqref{smoothmodulationequations:initial-nc}.

\noindent \textbf{Step 2}. \emph{Equation for $\nu$}. From the proof of Lemma \ref{lemma:prelim-est:diffusive}, we know that
\be \label{smooth:modulation:inter:diffusive-nc}
\begin{split}
  &2\int_0^{\frac{1}{\nu}} (\phi + \tilde a)^2 (z) dz =1+O(C(z^*)s^{-h_a+1}) ,
  \\
  &\int_0^{\frac{1}{\nu}} \nu \partial_z^{-1} \tilde{c}(z) dz = O(s^{1}e^{-\frac{l}{2}s}).
\end{split}
\ee
Injecting \eqref{smooth:modulation:inter:diffusive-nc} in \eqref{smoothmodulationequations:initial-nc} gives:
\be \label{smooth:modulation:inter2:diffusive-nc}
-\frac{\nu_s}{\nu^2} = \frac1\nu (\frac{\lambda_s}\lambda +1)= 1 +O(C(z^*)s^{-{h_a}+1})
\ee
when $s_0$ is large enough.
Integrating \eqref{smooth:modulation:inter2:diffusive-nc} with time, we find:
\be \label{smooth:modulation:inter3-nc}
\frac{1}{\nu}=\frac{1}{\nu_0}+(s-s_0)+O(C(z^*)(s^{-h_a+1+1}-s_0^{-h_a+1+1}))=\frac{1}{\nu_0}+(s-s_0)+O(C(z^*)s^{-h_a+1}(s-s_0)).
\ee
Therefore, since $\nu_0^{-1}\leq N_0 s_0$ from \eqref{bd:parametersini-beta=0-nc} we infer for $s_0^*$ large enough depending on $z^*$:
\be \label{smooth:modulation:inter1:diffusive-nc}
\frac{1}{\nu}\leq N_0 s_0+(s-s_0)+O(C(z^*)s^{-h_a+1}(s-s_0))= s+(N_0-1)s_0+O(C(z^*)s^{-h_a+1}(s-s_0))\leq (N_0+\epsilon)s.
\ee
One finds similarly using $s_0/N_0\leq \nu_0^{-1}$ from \eqref{bd:parametersini-beta=0-nc} that $\frac{1}{\nu}\geq s/(N_0+\epsilon)$. This and \eqref{smooth:modulation:inter1:diffusive-nc} imply the second inequality in \eqref{smooth:bd:boostrap improved parameters:diffusive-nc}. Finally, if $s_1=\infty$ then \eqref{smooth:modulation:inter3-nc} implies $\nu^{-1}=s+O(s^{-h_a+2})$ and the second inequality in \eqref{smooth:bd:boostrap improved parameters:diffusive-nc} follows.\\

\noindent \textbf{Step 3}. \emph{Equation for $\lambda$}. Injecting \eqref{smooth:modulation:inter:diffusive-nc} in \eqref{smoothmodulationequations:initial-nc} one finds:
\begin{equation}\label{asy on lambda derivative-nc}
    \frac{\lambda_s}{\lambda}+1 = \frac{1}{s^1}+O(C(z^*)s^{-h_a})\implies \frac{\lambda_s}{\lambda}=-1+O(C(z^*)s^{-1})
\end{equation}
Since $\lambda = O(se^{-1 s})$ from \eqref{bd:parameterstrap-critical:diffusive-nc}, one has
\begin{equation*}
    \frac{d}{ds}(\frac{e^s\lambda}{s}) =O(s^{-h_a}).
\end{equation*}
We integrate with time the above equation using $\lambda_0 = s_0 e^{-s_0}$ and find
$$
\lambda(s) = se^{-s}(1+\int_{s_0}^sO(s^{-h_a})ds)=se^{-s}(1+O(s_0^{-h_a+1}))
$$
This implies the first inequality in \eqref{smooth:bd:boostrap improved parameters:diffusive-nc} for $s_0$ large enough. If $s_1=\infty$ then we set $\tilde\lambda_\infty=1+\int_{s_0}^\infty O(s^{-h_a})ds$ and rewrite the above equality as:
$$
\lambda(s) = se^{-s}(1+\int_{s_0}^\infty O(s^{-h_a})ds-\int_{s}^\infty O(s^{-h_a})ds )= se^{-s}(\tilde\lambda_\infty+O(s^{-h_a+1}) ),
$$
which gives the first inequality in \eqref{smooth:bd:boostrap improved parameters:diffusive-nc}.
\end{proof}}

\subsection{Interior Estimate}
{The following interior differential inequality for $\mathcal I_a^2$ is similar to its counterpart in the non-diffusive case. However, we do not have an interior estimate for $\mathcal T_{k\eta_0,2\eta_0}^{2\eta_0}$ as it is defined in the whole domain.}

For the following lemma, we further assume 
\begin{equation}\label{interior-parameter-a-nc}
        (-\alpha+k)\frac{\eta_0}{\eta_0-1}>-1.
\end{equation}
\ifthenelse{\boolean{isSimplified}}{}{Notice that this condition implies
\begin{equation}
      k>\alpha-1+\frac{1}{\eta_0}\implies k>\alpha-1+\frac{1}{\eta}\quad\forall \eta\geq \eta_0,\, \eta\in \N.
\end{equation}}
\begin{lemma}[Interior estimate for $\tilde a$]\label{smooth:lemma:interior:new idea:nc}
Assume that \eqref{initial-parameter-weights-nc}, \eqref{initial-parameter-other-nc},  \eqref{interior-parameter-a-nc} are satisfied. For any $z^*\geq 1$, $\epsilon>0$, and $ \delta>0$, there exists a large self-similar time $s_0^*$ such that for any $s_0\geq s_0^*$, for any solution which is trapped on $[s_0,s_1]$, we have for $s\in [s_0,s_1]$:
\begin{equation}\label{interior-estimate-of-a-diffusive-nc}
\begin{split}
     \frac{d}{ds}\mathcal I_a^2&\leq \mathcal I_a^2 \left( {-\alpha+1}+O(C(z^*) s^{-\frac{h_a}{2}})+\frac{1}{\sqrt{\alpha+1}}+\frac{4}{(\alpha+3)}\sqrt{\frac{3}{8(\alpha+1)}}\right)\\
     &\quad{+\mathcal I_a C(z^*)(s^{-1}+ e^{(-\frac{l}{2}+\epsilon) s}+e^{-\frac{l}{2}s}s^{-1})}.\\
\end{split}
\end{equation}

\end{lemma}
\ifthenelse{\boolean{isSimplified}}{}{
\begin{remark}
    Notice ${-\alpha+1}+\frac{1}{\sqrt{\alpha+1}}+\frac{4}{(\alpha+3)}\sqrt{\frac{3}{8(\alpha+1)}}<0\iff\alpha>1.88415$.
\end{remark}
}

\begin{proof}

\ifthenelse{\boolean{isSimplified}}{Let $w=z^{-\alpha}$, differentiating $\mathcal I_a$ gives \begin{equation} \label{smooth:interior:id:expression2}
\resizebox{.95 \textwidth}{!}{$\begin{split}
    &\frac{1}{2}\frac{d}{ds} \int_0^{z^*}  w \tilde a_z ^2 dz=   \int_0^{z^*} ( -1  + \phi +\tilde a)   w \tilde a_z^2 dz 
     - \int_0^{z^*} (\partial_z^{-1} \phi - \frac{\nu_s}{\nu} z + \partial_z^{-1} \tilde a) \tilde a_{zz}  w \tilde a_z dz - \nu \int_0^{z^*} w\tilde a_z \tilde cdz
     \\
     &-\int_0^{z^*} \phi \tilde a  w \tilde a_z dz - 
     \int_0^{z^*} \partial_z^{-1} \tilde a   w \tilde a_z \phi dz -\left[2\nu \int_0^{\frac{1}{\nu}} (\phi+ \tilde{a})^2(z) dz - \lambda \left(\int_0^{\frac{1}{\nu}} \frac{\nu^2}{\lambda} \partial_z^{-1} \tilde{c}(z) dz\right) \right]\int_0^{z^*} z\phi  w \tilde a_z dz.\\
\end{split}$}
\end{equation}}
{Calculating the derivative of $I_a$, one obtains that \begin{equation} \label{smooth:interior:id:expression2}
\begin{split}
    &\frac{1}{2}\frac{d}{ds} \int_0^{z^*}  w \tilde a_z ^2 dz=   \int_0^{z^*} ( -1  + \phi +\tilde a)   w \tilde a_z^2 dz 
     - \int_0^{z^*} (\partial_z^{-1} \phi - \frac{\nu_s}{\nu} z + \partial_z^{-1} \tilde a) \tilde a_{zz}  w \tilde a_z dz
     \\
     &-\int_0^{z^*} \phi \tilde a  w \tilde a_z dz - 
     \int_0^{z^*} \partial_z^{-1} \tilde a   w \tilde a_z \phi dz \\
     &-\left[2\nu \int_0^{\frac{1}{\nu}} (\phi+ \tilde{a})^2(z) dz - \lambda \left(\int_0^{\frac{1}{\nu}} \frac{\nu^2}{\lambda} \partial_z^{-1} \tilde{c}(z) dz\right) \right]\int_0^{z^*} z\phi  w \tilde a_z dz
     \\
     & - \nu \int_0^{z^*} w\tilde a_z \tilde cdz,
\end{split}
\end{equation}
by utilizing the identities involving $\phi=e^{-z}$ and $\lambda_s/\lambda+\nu_s/\nu=-1$, where $w=z^{-\alpha}$.}

\noindent \underline{Potential and transport terms}. \ifthenelse{\boolean{isSimplified}}{Since $ \tilde a_z=o(z^{\frac{\alpha-1}{2}})$, by mimicking the proof of Lemma \ref{smooth:lemma:interior:new idea}, we can show that \begin{equation}  \label{smooth:interior:bd:transport-a-nc} \resizebox{.93 \textwidth}{!}{$
\displaystyle\int_0^{z^*} ( -1  + \phi +\tilde a)   w \tilde a_z^2 dz 
     - \int_0^{z^*} (\partial_z^{-1} \phi - \frac{\nu_s}{\nu} z + \partial_z^{-1} \tilde a) \tilde a_{zz}  w \tilde a_z dz\leq \mathcal I_a^2(s)
    \left(\frac{-\alpha+1}{2}+O(C(z^*)s^{-\frac{h_a}{2}})\right).$}
\end{equation}}{Integrating by parts yields
\begin{equation}  \label{smooth:interior:id:expression-nc}
    \begin{split}
        &  \int_0^{z^*} ( -1  + \phi +\tilde a)   w \tilde a_z^2 dz -\int_0^{z^*} (\partial_z^{-1} \phi - \frac{\nu_s}{\nu} z + \partial_z^{-1} \tilde a) \tilde a_{zz}  w \tilde a_z dz
        \\
        = & \Big(-\int_0^{z^*} (\phi+\tilde a)(\tilde{z})d\tilde{z}+ \frac{\nu_s}{\nu}z^* \Big) \frac{1}{2}w(z^*) \tilde a_z^2(z^*) \\
        &+\int_0^{z^*} \left( \frac{3}{2} \phi+\frac 32 \tilde a-1+\frac{\alpha-1}{2}\frac{\nu_s}{\nu}-\frac{\alpha}{2}z^{-1}\left(\pa_z^{-1}\phi+\pa_z^{-1}\tilde a \right)\right) w\tilde a_z^2dz 
    \end{split}
\end{equation}
For the boundary term, we know that $\|\tilde a\|_{L^\infty} \leq C(z^*) s^{-\frac{h_a}{2}}$ from \eqref{smooth:bd:varepsilonLinfty:diffusive-nc} and thus using \eqref{smoothmodulationequations:diffusive-nc}:
\begin{equation}\label{condition-critical:s0-3-nc}
      -\int_0^{z^*} (\phi+\tilde a)(\tilde{z})d\tilde{z}+ \frac{\nu_s}{\nu}z^* \leq -(1-e^{-z^*}) + \|\tilde a\|_{L^\infty} z^*+O(C(z^*)s^{-1}) \leq -(1-e^{-z^*}) +C(z^*)s^{-\frac{h_a}{2}} \leq 0
\end{equation} 
when $s_0$ is large enough.

Notice that this requires
\begin{equation}\label{first-vanishing-condition-nc}
    (\partial_z^{-1} \phi - \frac{\nu_s}{\nu} z + \partial_z^{-1} \tilde a)   z^{-\alpha} \tilde a_z^2=o(1) \text{ as }z\xrightarrow{}0^+.
\end{equation}
Since $\partial_z^{-1} \phi\simeq z, \antipartial_z \tilde a=o(z^2)$, and thus 
\begin{equation}\label{identity-vanishing-speed-nc}
    (\partial_z^{-1} \phi - \frac{\nu_s}{\nu} z + \partial_z^{-1} \tilde a)\approx z \text{ as }z\xrightarrow{}0^+,
\end{equation}
this condition is equivalent to
\begin{equation}
        \tilde a_z=o(z^{\frac{\alpha-1}{2}}) \text{ as }z\xrightarrow{}0^+,
\end{equation}
which is satisfied by \eqref{initial-parameter-other-nc}.
One uses $\phi=e^{-z}$, \eqref{smoothmodulationequations:diffusive-nc} and \eqref{smooth:bd:varepsilonLinfty:diffusive-nc} to see:
\begin{equation}  \label{smooth:interior:inter1-nc}
\begin{split}
    &\frac{3}{2} \phi+\frac 32 \tilde a-1+\frac{\alpha-1}{2}\frac{\nu_s}{\nu}-\frac{\alpha}{2}z^{-1}\left(\pa_z^{-1}\phi+\pa_z^{-1}\tilde a \right)
    \\
    &=\left(\frac{3}{2} e^{-z}+O(s^{-\frac{h_a}{2}})-1+O(s^{-1})-\frac{\alpha}{2}z^{-1}\left(1-e^{-z} +O(s^{-\frac {h_a}{2}}z) \right)\right)
    \\
&   =-1+\frac{3}{2}e^{-z}-\frac{\alpha}{2}\frac{1-e^{-z}}{z}+O(C(z^*)s^{-\frac{h_a}{2}})\\
& \leq -1+(\frac 32-\frac \alpha2)e^{-z}+O(C(z^*)s^{-\frac{h_a}{2}})\\
&\leq \begin{cases}
    -1+O(C(z^*)s^{-\frac{h_a}{2}}) \text{ if }\alpha \geq 3\\
    \frac{-\alpha+1}{2}+O(C(z^*)s^{-\frac{h_a}{2}}) \text{ if otherwise}
\end{cases}
\end{split}
\end{equation}
\begin{equation}  \label{smooth:interior:bd:transport-a-nc}
\int_0^{z^*} ( -1  + \phi +\tilde a)   w \tilde a_z^2 dz 
     - \int_0^{z^*} (\partial_z^{-1} \phi - \frac{\nu_s}{\nu} z + \partial_z^{-1} \tilde a) \tilde a_{zz}  w \tilde a_z dz\leq \mathcal I_a^2(s)
    \left(\frac{-\alpha+1}{2}+O(C(z^*)s^{-\frac{h_a}{2}})\right).
\end{equation}}
   
\ifthenelse{\boolean{isSimplified}}{\noindent \underline{The nonlocal terms and source terms}  Following the same strategy of Lemma \ref{smooth:lemma:interior:new idea}, we estimate using Hölder inequality and Lemma \ref{lemma:prelim-est:diffusive-nc}:
\begin{equation}
    \begin{split}
        &\int_0^{z^*} \Big| \phi \tilde a  w \tilde a_z \Big| dz\leq\frac{1}{2\sqrt{\alpha+1}}\mathcal I_a^2(s)\\
        &\int_0^{z^*} \Big| \partial_z^{-1} \tilde a   w \tilde a_z \phi \Big| dz \leq \frac{2}{(\alpha+3)} \sqrt{\frac{3}{8(\alpha+1)}}\mathcal I_a^2(s)\\
        &\left|2 \nu \Big(\int_0^{\frac{1}{\nu}} (\phi + \tilde a)^2 (z) dz\Big) \int_0^{z^*} z\phi  w \tilde a_z dz \right| \leq C(z^*) s^{-1}\mathcal I_a(s)\\
        &\left|\nu^2 \int_0^{\frac1\nu} \pa_z^{-1} \tilde c(z) dz \int_0^{z^*} z\phi  w \tilde a_z dz \right|\leq C(z^*) e^{(-l/2+\epsilon) s}\mathcal I_a(s),
    \end{split}
\end{equation}
{where, for the last inequality, we have used \eqref{estimate-on-primitive-of-c-diffusive-2-nc} to bound the term involving $\tilde c$.}}{\noindent \underline{The nonlocal terms}. By direct computations, using that Cauchy-Schwarz one gets
\begin{equation}
    \begin{split}
        &|\tilde a (z)|\leq \int_0^{z}  \abs{\tilde a_z} \sqrt{w} \sqrt{w^{-1}}dx \leq I_a(s) \frac{z^{\frac{\alpha+1}{2}}}{(\alpha+1)^\frac12}\\
        &|\pa^{-1}_z\tilde a (z)| \leq I_a(s) \frac{2}{(\alpha+3)\sqrt{\alpha+1}}z^{\frac{\alpha+3}{2}}
    \end{split}
\end{equation} for $z\in[0,z^*]$. Thus, we conclude, by using Hölder again
\begin{equation*}
\begin{split}
    \int_0^{z^*} \Big| \phi \tilde a  w \tilde a_z \Big| dz=\int_0^{z^*} \Big| \phi \tilde a  \sqrt{w}\cdot \sqrt{w}\tilde a_z \Big| dz &\leq \frac{1}{\sqrt{\alpha+1}}\mathcal I_a^2(s) (\int_0^\infty e^{-2z}z)^\frac{1}{2}= \frac{1}{2\sqrt{\alpha+1}}\mathcal I_a^2(s),
\end{split}
\end{equation*}
\begin{equation*}
\begin{split}
    \int_0^{z^*} \Big| \partial_z^{-1} \tilde a   w \tilde a_z \phi \Big| dz &=  \int_0^{z^*} \Big| \phi\partial_z^{-1} \tilde a   \sqrt{w}\cdot\sqrt{w} \tilde a_z  \Big| dz \\ 
    &\leq \frac{2}{(\alpha+3)\sqrt{\alpha+1}}I_a^2(s)\Big(\int_0^{\infty}  z^3e^{-2z}dz \Big)^{\frac{1}{2}} = \frac{2}{(\alpha+3)}I_a^2(s) \sqrt{\frac{3}{8(\alpha+1)}} ,
\end{split}
\end{equation*}

\noindent \underline{The source term}. Using \eqref{smooth:bd:etrap:diffusive-nc}, \eqref{smooth:bd:integral:diffusive-nc} and Cauchy-Schwarz:
$$
\left|2 \nu \Big(\int_0^{\frac{1}{\nu}} (\phi + \tilde a)^2 (z) dz\Big) \int_0^{z^*} z\phi  w \tilde a_z dz \right| \leq N s^{-1} \mathcal I_a \sqrt{\int_0^{z^*} \underbrace{z^2 \phi^2 w}_{z^{2-\alpha}e^{-2z}}dz} \leq C(z^*) s^{-1}\mathcal I_a(s),
$$

\begin{equation*}
    \begin{split}
        &\left|\nu^2 \int_0^{\frac1\nu} \pa_z^{-1} \tilde c(z) dz \int_0^{z^*} z\phi  w \tilde a_z dz \right|\leq e^{(-l/2+\epsilon) s} I_a \sqrt{\int_0^{z^*} z^2\phi^2wdz} \leq C(z^*) e^{(-l/2+\epsilon) s}\mathcal I_a(s).
    \end{split}
\end{equation*}
Notice that integrability is guaranteed when $\alpha<3$. }
\ifthenelse{\boolean{isSimplified}}{Next, using Cauchy-Schwarz and Hölder's inequality, we estimate 
{
\begin{equation}
    \begin{split}
        \left| \nu \int_0^{z^*} w\tilde a_z \tilde cdz \right| &\leq C s^{-1}\int_0^{z^*} \abs{ \tilde{a}_z z^{-\alpha/2}}\cdot \abs{z^{-\alpha/2} \tilde c} \,dz\\
        &\leq Cs^{-1} \mathcal I_a \left(\int_0^{z^*} z^{-\alpha}\tilde{c}^2 dz\right)^\frac{1}{2}\\
        &\leq Cs^{-1} \mathcal I_a \left(\left(\int_0^{\nu^{-1}} (z^{-k} \tilde c^2)^{\eta_0}dz\right)^\frac{1}{\eta_0}\left(\int_0^{z^*} (z^{-\alpha+k})^{\frac{\eta_0}{\eta_0-1}}dz\right)^{\frac{\eta_0-1}{\eta_0}}\right)^\frac{1}{2} \\
        &\leq \mathcal T_{k\eta_0,2\eta_0} \mathcal I_a C(z^*,\eta_0,\alpha,k)s^{-1},\\ 
    \end{split}
\end{equation}
where, for the last integral in the second last line can be bounded above by $C(z^*,\eta_0,\alpha,k)$, as the choice of parameters \eqref{interior-parameter-a-nc} holds.}}{Next, using Cauchy-Schwarz we compute
\begin{equation*}
     \begin{split}
       \left| \nu \int_0^{z^*} w\tilde a_z \tilde cdz \right| &\leq C s^{-1}\int_0^{z^*} \abs{ \tilde{a}_z z^{-\alpha/2}}\cdot \abs{z^{-\alpha/2} \tilde c}\\
       &\leq Cs^{-1}I_a (\int_0^{z^*} z^{-\alpha}\tilde{c}^2)^\frac{1}{2}.\\
    \end{split}
\end{equation*}
Using \eqref{smooth:bd:etrap:diffusive-nc} and Hölder's inequality, we can further estimate
\begin{equation}
    \begin{split}
        \int_0^{z^*} z^{-\alpha}\tilde{c}^2&=\int_0^{z^*} z^{-k}\tilde{c}^2\cdot z^{-\alpha+k}\\
        &\leq(\int_0^{\nu^{-1}} (z^{-k} \tilde c^2)^{\eta_0})^\frac{1}{\eta_0}(\int_0^{z^*} (z^{-\alpha+k})^{\frac{\eta_0}{\eta_0-1}})^{\frac{\eta_0-1}{\eta_0}}\\
        &\leq T_{k\eta_0,2\eta_0}^2 C(z^*,\eta_0,\alpha,k).
    \end{split}
\end{equation}
where we have used the condition 
\begin{equation*}
    (-\alpha+k)\frac{\eta_0}{\eta_0-1}>-1,\qquad \eta_0>1.
\end{equation*}
Thus, this term can be bounded by
\begin{equation*}
     \left|  \nu \int_0^{z^*} w\tilde a_z \tilde cdz \right| \leq C(z^*) s^{-1} I_a T_{k\eta_0, 2\eta_0}
\end{equation*}}

Summarizing all the estimates yields
{
\begin{equation}
\begin{split}
     \frac{1}{2}\frac{d}{ds}\mathcal I_a^2&\leq \mathcal I_a^2 ( \frac{-\alpha+1}{2}+O(C(z^*) s^{-\frac{h_a}{2}})+\frac{1}{2\sqrt{\alpha+1}}+\frac{2}{(\alpha+3)}\sqrt{\frac{3}{8(\alpha+1)}})\\
     &\quad+\mathcal I_a(C(z^*) s^{-1}+C(z^*) e^{(-l/2+\epsilon) s})+\mathcal  I_a \mathcal T_{k\eta_0,2\eta_0} C(z^*) s^{-1},
\end{split}
\end{equation}
}
and we finish the proof by injecting the bootstrap assumptions \eqref{smooth:bd:etrap:diffusive-nc}.
\end{proof}

\subsection{Total Estimate}
{In this subsection, we derive a differential inequality in $T_{k\eta,2\eta}$ for all $\eta \geq \eta_0$, in the entire domain. We are not only interested in the case where $\eta=\eta_0$. As usual, the parameter $\eta_0$ is contained in the bootstrap argument, however, once the bootstrap argument is finished, we will be taking $\eta\rightarrow\infty$ to derive an $L^\infty$ estimate on $\tilde c$.
}

{To establish such differential inequality, we will be using a weighted energy estimate on $H_0^1$, which is a direct consequence of weighted Hardy's inequality \cite{WEDESTIG2003}.}
\begin{theorem}[Weighted Hardy's Inequality \cite{WEDESTIG2003}]
Let $f$ be non-negative and measurable, and let $\epsilon<p-1$ and $p>1$. The weighted Hardy's inequality states that
\ifthenelse{\boolean{isSimplified}}{{\begin{equation}
    \begin{split}
        &\int_0^\infty (\frac{1}{x}\int_0^x f(t) dt)^p x^\epsilon dx\leq (\frac{p}{p-1-\epsilon})^p \int_0^\infty f^p(x) x^\epsilon dx.\\
    \end{split}
\end{equation}}}{
\begin{equation}
    \begin{split}
        &\int_0^\infty (\frac{1}{x}\int_0^x f(t) dt)^p x^\epsilon dx\leq (\frac{p}{p-1-\epsilon})^p \int_0^\infty f^p(x) x^\epsilon dx,\text{ for } p>1\text{ and } \epsilon<p-1,\\
        &\int_0^\infty (\frac{1}{x}\int_x^\infty f(t) dt)^p x^\epsilon dx\leq (\frac{p}{\epsilon+1-p})^p \int_0^\infty f^p(x) x^\epsilon dx,\text{ for } p>1\text{ and } \epsilon>p-1.
    \end{split}
    \end{equation}
    }
\end{theorem}Consequently, we have the following corollary.

\begin{corollary}[Corollary of Hardy's inequality in $H_0^1$]
Let $k>0$ {and $f\in H_0^1([0,a])$}. The following inequality holds
\begin{equation}\label{Corollary of Hardy's}
    \int_0^a  f^2 x^{-k-2}dx \leq (\frac{2}{1+k})^2\int_0^a f'^2 x^{-k} dx.
\end{equation}
\end{corollary}
\ifthenelse{\boolean{isSimplified}}{\begin{proof}
     Use fundamental theorem of calculus, {vanishing boundary condition, integral absolute value inequality}, and the fact that $x\mapsto x^2$ is increasing for non-negative numbers to reduce the target quantity until which the previous theorem  can be applied with $\epsilon=-k,p=2,f\xleftarrow{} \abs{f'\mathds{1}_{[0,a]}}$. {Concretely, we compute \begin{equation*}
    \begin{split}
          \int_0^a f^2 x^{-k-2}dx&=\int_0^\infty \abs{f(x)\mathds{1}_{[0,a]}(x)\cdot x^{-1}}^2\cdot x^{-k} dx\\
          &=\int_0^\infty \abs{x^{-1}\int_0^x f'(t)\mathds{1}_{[0,a]}(t) dt}^2 x^{-k} dx\\
          &\leq \int_0^\infty \left(x^{-1}\int_0^x \abs{f'(t)\mathds{1}_{[0,a]}(t) } dt\right)^2 x^{-k} dx\\
          &\leq (\frac{2}{2-1-(-k)})^2 \int_0^\infty \abs{f'(x)\mathds{1}_{[0,a]}(x)}^2 x^{-k}dx=(\frac{2}{1+k})^2 \int_0^a \abs{f'(x)}^2 x^{-k}dx.
    \end{split}
    \end{equation*}}
\end{proof}}{\begin{proof}
    Combining the previous theorem with $\epsilon=-k,p=2,f\xleftarrow{} \abs{f'\mathds{1}_{[0,a]}}$ and fundamental theorem of calculus, we compute
    \begin{equation*}
    \begin{split}
          \int_0^a f^2 x^{-k-2}dx&=\int_0^\infty \abs{f(x)\mathds{1}_{[0,a]}(x)\cdot x^{-1}}^2\cdot x^{-k} dx\\
          &=\int_0^\infty \abs{x^{-1}\int_0^x f'(t)\mathds{1}_{[0,a]}(t) dt}^2 x^{-k} dx\\
          &\leq \int_0^\infty \left(x^{-1}\int_0^x \abs{f'(t)\mathds{1}_{[0,a]}(t) } dt\right)^2 x^{-k} dx\\
          &\leq (\frac{2}{2-1-(-k)})^2 \int_0^\infty \abs{f'(x)\mathds{1}_{[0,a]}(x)}^2 x^{-k}dx=(\frac{2}{1+k})^2 \int_0^a \abs{f'(x)}^2 x^{-k}dx,
    \end{split}
    \end{equation*}
    where we have used that $f(0)=0$ and $f(a)=0$ in the second line to simplify
    \begin{equation*}
        \int_0^x f'(t)\mathds{1}_{[0,a]}(t) dt=\begin{cases}
            f(x)-f(0) \qquad\text{ if }x\leq a\\
            f(a)+\int_a^{x} f'(t)\cdot 0\, dt\qquad\text{ if }x> a
        \end{cases}
        =f(x)\mathds{1}_{[0,a]}(x).
    \end{equation*}
\end{proof}}
For the following lemma, we assume that $k$ and $\eta_0$ satisfy {the additional conditions}
\begin{equation}\label{condition on k:diffusive-nc}
    k< 2-\frac{1}{\eta_0},\, \eta_0\geq2,\, \eta_0\in\mathds{N},\, \epsilon_c\geq\frac{k}{2}+\frac{1}{2\eta_0}.
\end{equation}
Notice that this ensures that
\begin{equation}\label{vanishing-assumptions-on-c-nc}
    \tilde{c}=o(z^\frac{k\eta_0+1}{2\eta_0})=o(z^{\frac{k}{2}+\frac{1}{2\eta_0}}),
\end{equation}
and it implies that for $\eta\geq \eta_0$:
\begin{equation}
    k<2-\frac{1}{\eta},\,\eta\geq \eta_0\geq 2,\,\tilde c=o(z^{\frac{k}{2}+\frac{1}{2\eta}})
\end{equation}

\begin{lemma}[Energy estimate on $\mathcal T_{k\eta,2\eta}$]\label{differential inequality of c-nc}
Assume that \eqref{initial-parameter-weights-nc}, \eqref{initial-parameter-other-nc}, and \eqref{condition on k:diffusive-nc} are satisfied. For each $\eta\geq \eta_0$, for any $z^*\geq 1$ and $ \delta>0$, there exists a large self-similar time $s_0^*$ such that for any $s_0\geq s_0^*$, for any solution which is trapped on $[s_0,s_1]$, we have for $s\in [s_0,s_1]$:
\begin{equation}\label{smooth:lemma:interior:c:new idea:generalized-nc}
    \frac{d}{ds}\mathcal T_{k\eta,2\eta}^{2\eta} \leq  \mathcal T_{k\eta,2\eta}^{2\eta} (O(\eta C(z^*) s^{-\frac{h_a}{2}})-k\eta+1).
\end{equation}
Or, equivalently,
\begin{equation}
      \frac{d}{ds}\mathcal T_{k\eta,2\eta}\leq  \mathcal T_{k\eta,2\eta} (O( C(z^*) s^{-\frac{h_a}{2}})-\frac{k}{2}+\frac{1}{2\eta}).
\end{equation}
\end{lemma}
\begin{proof}
    First, let us observe that as long as $k< 2-\frac{1}{\eta}$ and $\eta\geq2,\, \eta\in\mathds{N}$ , there exists some $0<\epsilon_0\ll 1$ such that for each $0<\epsilon\leq\epsilon_0$,
    \begin{equation*}
        \begin{split}
            \frac{k\eta+1}{2\eta}>\frac{k\eta}{2\eta-1}>\frac{k\eta-1}{2\eta-2}+\epsilon>\frac{k\eta-1}{2\eta-2}>\frac{k\eta-1}{2\eta}+\epsilon>\frac{k\eta-1}{2\eta},
        \end{split}
    \end{equation*}
    and consequently $\text{ as }z\rightarrow 0^+$,
    \begin{equation}\label{diffusive-power-fact-nc}
        o(z^\frac{k\eta+1}{2\eta})\implies o(z^\frac{k\eta}{2\eta-1})\implies o(z^{\frac{k\eta-1}{2\eta-2}+\epsilon})\implies o(z^\frac{k\eta-1}{2\eta-2})\implies o(z^{\frac{k\eta-1}{2\eta}+\epsilon})\implies o(z^\frac{k\eta-1}{2\eta}),
    \end{equation}
    and the left side is given by assumption \eqref{vanishing-assumptions-on-c-nc}.
    This fact will be useful later.
    
    \ifthenelse{\boolean{isSimplified}}{Combining with the condition $ \tilde c=o(z^{\frac{k\eta-1}{2\eta}}) \text{ as }z\xrightarrow{}0^+$ and the Dirichlet boundary condition at $z=\nu^{-1}$ \eqref{equation:epsilon-bc}, differentiating $\mathcal T_{k\eta,2\eta}^{2\eta}(s)$ and integrate by parts gives
    \begin{equation} \label{smooth:interior:id:expression2-c:generalized-nc}
\begin{split}
    \frac{1}{2\eta}\frac{d}{ds} \int_0^{\frac1\nu}  z^{-k\eta}\tilde c ^{2\eta} dz   =&\frac{\lambda}{\nu^2}\int_0^{\frac1\nu}  z^{-k\eta}\tilde c^{2\eta-1} \tilde c_{zz} dz+\int_0^{\nu^{-1}}\tilde{c}^{2\eta}z^{-k\eta} f(z,s) dz,
\end{split}
\end{equation}
where $f(z,s)$ satisfies $$f(z,s)=2\frac{\lambda_s}{\lambda}+(\frac{k}{2}-\frac{1}{2\eta}) \frac{\nu_s}{\nu} +(2+\frac{1}{2\eta})\tilde{a}+(2+\frac{1}{2\eta})\phi-\frac{k}{2}z^{-1}\pa_z^{-1}\phi-\frac{k}{2}z^{-1}\pa_z^{-1}\tilde{a}.$$
    }{Differentiating $T_{k\eta,2\eta}(s)$ gives
    \begin{equation} \label{smooth:interior:id:expression2-c:generalized-nc}
\begin{split}
    &\frac{1}{2\eta}\frac{d}{ds} \int_0^{\frac1\nu}  z^{-k\eta}\tilde c ^{2\eta} dz   =\frac{1}{2\eta}\frac{-\nu_s}{\nu^2}\nu^{k\eta}\tilde c(s,\nu^{-1}(s))^{2\eta}+\int_0^{\frac1\nu} (2\frac{\lambda_s}{\lambda} + 2\tilde a + 2\phi) z^{-k\eta}\tilde c^{2\eta} dz\\
    &\qquad+ \int_0^{\frac1\nu} (-\partial_z^{-1} \phi +\frac{\nu_s}{\nu} z -\partial_z^{-1} \tilde a) \tilde   z^{-k\eta} \tilde c^{2\eta-1} \tilde c_z dz+\sigma\frac{\lambda}{\nu^2}\int_0^{\frac1\nu}  z^{-k\eta}\tilde c^{2\eta-1} \tilde c_{zz} dz.
\end{split}
\end{equation}
 Integrating by parts gives
\begin{equation}
\begin{split}
    &\frac{1}{2\eta}\frac{-\nu_s}{\nu^2}\nu^{k\eta}\tilde c(s,\nu^{-1}(s))^{2\eta}+\int_0^{\frac1\nu} (-\partial_z^{-1} \phi +\frac{\nu_s}{\nu} z -\partial_z^{-1} \tilde a) \tilde   z^{-k\eta} \tilde c^{2\eta-1} \tilde c_z dz\\
    &=-\frac{1}{2\eta}{\nu_s}\nu^{k\eta-2}\tilde c(s,\nu^{-1}(s))^{2\eta}+\frac{1}{2\eta}\int_0^{\frac1\nu} (-\partial_z^{-1} \phi +\frac{\nu_s}{\nu} z -\partial_z^{-1} \tilde a) \tilde  z^{-k\eta} \partial_z (\tilde c^{2\eta}) dz\\
    &=-\frac{1}{2\eta}{\nu_s}\nu^{k\eta-2}\tilde c(s,\nu^{-1}(s))^{2\eta}+\left(\frac{1}{2\eta}(-\partial_z^{-1} \phi +\frac{\nu_s}{\nu} z -\partial_z^{-1} \tilde a) \tilde  z^{-k\eta} \tilde c^{2\eta}\right)_{z=0}^{z=\frac{1}{\nu}}\\
    &\qquad-\frac{1}{2\eta} \int_0^{\nu^{-1}}\tilde c^{2\eta}\partial_z\left(z^{-k\eta} (-\partial_z^{-1} \phi +\frac{\nu_s}{\nu} z -\partial_z^{-1} \tilde a)\right) dz\\
    &=-\left(\frac{1}{2\eta}(-\partial_z^{-1} \phi +\frac{\nu_s}{\nu} z -\partial_z^{-1} \tilde a) \tilde  z^{-k\eta} \tilde c^{2\eta}\right)\rvert_{z=0}\\
    &\qquad-\frac{1}{2\eta} \int_0^{\nu^{-1}} \tilde c^{2\eta} z^{-k\eta} (\frac{\nu_s}{\nu}-\tilde a-\phi)+\frac{k}{2}\int_0^{\nu^{-1}} \tilde c^{2\eta} z^{-k\eta}(\frac{\nu_s}{\nu}-\antipartial_z\tilde a/z-\antipartial_z\phi/z),
\end{split}
\end{equation}
where we have used the compatibility condition \eqref{equation:epsilon-bc}, and simplified the following term using the boundary condition \eqref{sigma-one-nc} at $z=\nu^{-1}$:
\begin{equation*}
\begin{split}
    &\,\,-\frac{1}{2\eta}{\nu_s}\nu^{k\eta-2}\tilde c(s,\nu^{-1}(s))^{2\eta}+\left(\frac{1}{2\eta}(-\partial_z^{-1} \phi +\frac{\nu_s}{\nu} z -\partial_z^{-1} \tilde a) \tilde  z^{-k\eta} \tilde c^{2\eta}\right)\rvert_{z=\nu^{-1}(s)}\\
    &=-\frac{1}{2\eta}{\nu_s}\nu^{k\eta-2}\tilde c(s,\nu^{-1}(s))^{2\eta}+\frac{1}{2\eta}(\frac{\nu_s}{\nu^2}-\int_0^{\nu^{-1}} (\phi+\tilde a) dz)\nu^{k\eta}\tilde c(s,z=\frac{1}{\nu(s)})^{2\eta}\\
    &=-\frac{1}{2\eta}{\nu_s}\nu^{k\eta-2}\tilde c(s,\nu^{-1}(s))^{2\eta}+\frac{1}{2\eta}{\nu_s}\nu^{k\eta-2}\tilde c(s,\nu^{-1}(s))^{2\eta}=0.
\end{split}
\end{equation*}
Alternatively, if $\sigma=1$, then we simply use the boundary condition at $z=\nu^{-1}$ to eliminate these terms.
If, in addition, we have that
\begin{equation}
    (\partial_z^{-1} \phi - \frac{\nu_s}{\nu} z + \partial_z^{-1} \tilde a)   z^{-k\eta} \tilde c^{2\eta}=o(1) \text{ as }z\xrightarrow{}0^+,
\end{equation}
which is equivalent to
\begin{equation}
        \tilde c=o(z^{\frac{k}{2}-\frac{1}{2\eta}})=o(z^{\frac{k\eta-1}{2\eta}}) \text{ as }z\xrightarrow{}0^+,
\end{equation}
and the boundary term at $z=0$ indeed vanishes (thanks to \eqref{vanishing-assumptions-on-c-nc} and \eqref{diffusive-power-fact-nc}). We  now can further simplify the potential and boundary terms in \eqref{smooth:interior:id:expression2-c:generalized-nc}:
\begin{equation}
\begin{split}
    &\frac{1}{2\eta}\frac{-\nu_s}{\nu^2}\nu^{k\eta}\tilde c(s,\nu^{-1}(s))^{2\eta}+\int_0^{\frac1\nu} (2\frac{\lambda_s}{\lambda} + 2\tilde a + 2\phi) z^{-k\eta}\tilde c^{2\eta} dz
    + \int_0^{\frac1\nu} (-\partial_z^{-1} \phi +\frac{\nu_s}{\nu} z -\partial_z^{-1} \tilde a) \tilde   z^{-k\eta} \tilde c^{2\eta-1} \tilde c_z dz\\
    &=\int_0^{\nu^{-1}}\tilde{c}^{2\eta}z^{-k\eta} \left(2\frac{\lambda_s}{\lambda}+(\frac{k}{2}-\frac{1}{2\eta}) \frac{\nu_s}{\nu} +(2+\frac{1}{2\eta})\tilde{a}+(2+\frac{1}{2\eta})\phi-\frac{k}{2}z^{-1}\pa_z^{-1}\phi-\frac{k}{2}z^{-1}\pa_z^{-1}\tilde{a}\right),\\
\end{split}
\end{equation}}
Using $\frac{\lambda_s}{\lambda}+\frac{\nu_s}{\nu}=-1$ and \eqref{smoothmodulationequations:diffusive-nc}, we rewrite
\begin{equation*}
\begin{split}
   2\frac{\lambda_s}{\lambda}+(\frac{k}{2}-\frac{1}{2\eta}) \frac{\nu_s}{\nu}= 2(\frac{\lambda_s}{\lambda}+ \frac{\nu_s}{\nu})+(\frac{k}{2}-\frac{1}{2\eta}-2) \frac{\nu_s}{\nu}\leq 2(\frac{\lambda_s}{\lambda}+\frac{\nu_s}{\nu})+\abs{\frac{k}{2}-\frac{1}{2\eta}-2} \abs{\frac{\nu_s}{\nu}}\leq -2+Cs^{-1}.\\
\end{split}
\end{equation*}
Similar to the computations we did for $I_a$, we use \eqref{smoothmodulationequations:diffusive-nc} and \eqref{smooth:bd:varepsiloninterior:diffusive} to see
\begin{equation}\label{pot-estimate}
\begin{split}
     &\left(2\frac{\lambda_s}{\lambda}+(\frac{k}{2}-\frac{1}{2\eta}) \frac{\nu_s}{\nu} +(2+\frac{1}{2\eta})\tilde{a}+(2+\frac{1}{2\eta})\phi-\frac{k}{2}z^{-1}\pa_z^{-1}\phi-\frac{k}{2}z^{-1}\pa_z^{-1}\tilde{a}\right)\\
     &\leq -2+ Cs^{-1}+O(3 C(z^*) s^{-\frac{h_a}{2}})+(2+\frac{1}{2\eta})e^{-z}-\frac{k}{2}(\frac{1-e^{-z}}{z})+O(s^{-\frac{h_a}{2}})\\
     &\leq -2+(2+\frac{1}{2\eta}-\frac{k}{2})e^{-z}+O(C(z^*)s^\frac{-h_a}{2})\\
     &\leq O(C(z^*)s^\frac{-h_a}{2})+\begin{cases}
         -2\qquad\text{ if }k\geq 4+\frac{1}{\eta}\\
         \frac{1}{2\eta}-\frac{k}{2}\qquad\text{ if }k< 4+\frac{1}{\eta}
     \end{cases}
\end{split}
\end{equation}
where we have used that $1-e^{-x}\geq xe^{-x}$. 

For the higher order term, we integrate by parts repeatedly to see
\ifthenelse{\boolean{isSimplified}}{\begin{equation}\label{diffusive:integration-by-parts-nc}
\begin{split}
    \int_0^{\frac1\nu}  z^{-k\eta}\tilde c^{2\eta-1} \tilde c_{zz} dz&=\left(z^{-k\eta}\tilde c^{2\eta-1} \tilde c_z\right)_{z=0}^{z=\nu^{-1}}-\int_0^{\nu^{-1}}\tilde c_z \partial_z(z^{-k\eta}\tilde c^{2\eta-1})\\
    &=\left(\frac{k}{2}z^{-k\eta-1}\tilde c^{2\eta} \right)_{z=0}^{z=\nu^{-1}}-\frac{k}{2}\int_0^{\nu^{-1}}(-k\eta-1)z^{-k\eta-2}\tilde c^{2\eta}-\frac{2\eta-1}{\eta^2}\int_0^{\nu^{-1}} (\partial_z(c^\eta))^2\\
    &=\frac{k(k\eta+1)}{2}\int_0^{\nu^{-1}} (\tilde c^\eta)^2 z^{-k\eta-2}-\frac{2\eta-1}{\eta^2}\int_0^{\nu^{-1}} (\partial_z(\tilde c^\eta))^2z^{-k\eta},
\end{split}
\end{equation}
where we have used the vanishing boundary condition at $z=\nu^{-1}$ \eqref{sigma-one-nc} and \eqref{diffusive-power-fact-nc} to deduce the boundary terms vanish.
}{\begin{equation}\label{diffusive:integration-by-parts-nc}
\begin{split}
    \int_0^{\frac1\nu}  z^{-k\eta}\tilde c^{2\eta-1} \tilde c_{zz} dz&=\int_0^{\frac1\nu}  z^{-k\eta}\tilde c^{2\eta-1} \partial_z\tilde c_{z} dz\\
    &=\left(z^{-k\eta}\tilde c^{2\eta-1} \tilde c_z\right)_{z=0}^{z=\nu^{-1}}-\int_0^{\nu^{-1}}\tilde c_z \partial_z(z^{-k\eta}\tilde c^{2\eta-1})\\
    &=k\eta \int_0^{\nu^{-1}} z^{-k\eta-1}\tilde c^{2\eta-1}\tilde c_z -(2\eta-1)\int_0^{\nu^{-1}} \tilde c^{2\eta-2} z^{-k\eta}\tilde c_z^2\\
    &=\frac{k}{2} \int_0^{\nu^{-1}} z^{-k\eta-1} \partial_z (\tilde c^{2\eta}) -\frac{2\eta-1}{\eta^2}\int_0^{\nu^{-1}} (\partial_z(c^\eta))^2\\
    &=\left(\frac{k}{2}z^{-k\eta-1}\tilde c^{2\eta} \right)_{z=0}^{z=\nu^{-1}}-\frac{k}{2}\int_0^{\nu^{-1}}(-k\eta-1)z^{-k\eta-2}\tilde c^{2\eta}-\frac{2\eta-1}{\eta^2}\int_0^{\nu^{-1}} (\partial_z(c^\eta))^2\\
    &=\frac{k(k\eta+1)}{2}\int_0^{\nu^{-1}} (\tilde c^\eta)^2 z^{-k\eta-2}-\frac{2\eta-1}{\eta^2}\int_0^{\nu^{-1}} (\partial_z(\tilde c^\eta))^2z^{-k\eta},
\end{split}
\end{equation}
where we have used the vanishing boundary condition at $z=\nu^{-1}$ \eqref{sigma-one-nc} for $\sigma=1$ and we need

\begin{equation}\label{vanishing-boundary-condition-on-c:diffusive-nc}
   \tilde c_z z^{-k\eta}\tilde c^{2\eta-1}, \tilde c^{2\eta} z^{-k\eta-1}=o(1) \text{ as }z\xrightarrow{}0^+, \text{ if }\sigma=1,
\end{equation}
at the second and the last line respectively. Notice that this condition is weaker than
\begin{equation*}
\begin{split}
     &z^{-k\eta}\tilde c^{2\eta-1}, \tilde c^{2\eta} z^{-k\eta-1}=o(1) \\
     &\qquad\iff \tilde c=o(z^\frac{k\eta}{2\eta-1}) \text{ and } \tilde c=o(z^\frac{k\eta+1}{2\eta}) \text{ as }z\xrightarrow{}0^+, \text{ if }\sigma=1\\
     &\qquad \impliedby \tilde c=o(z^\frac{k\eta+1}{2\eta})\text{ if }\sigma=1\text{ and }k< 2-\frac{1}{\eta},
\end{split}
\end{equation*}
which is satisfied using \eqref{diffusive-power-fact-nc}.

Before applying Hardy's inequality, let us check whether the term
$$\int_0^{\nu^{-1}} (\partial_z(\tilde c^\eta))^2z^{-k\eta}$$
is indeed finite for each $s\geq s_0$. We observe that this condition can be satisfied via
\begin{equation}
    (\partial_z(\tilde c^\eta))^2 z^{-k\eta}\approx \tilde c_z^2 \tilde c^{2\eta-2} z^{-k\eta}=o(z^{-1+\epsilon})\text{ as }z\rightarrow 0^+,\,\text{for }0<\epsilon\ll 1,
\end{equation}
and is weaker than
\begin{equation}
     \tilde c=o(z^{\frac{k\eta-1}{2\eta-2}+\epsilon})\text{ as }z\rightarrow 0^+,\,\text{for some arbitrary small } \epsilon \text{ such that }0<\epsilon\ll 1,
\end{equation}
which is indeed satisfied by \eqref{diffusive-power-fact-nc}. Hardy's inequality \eqref{Corollary of Hardy's} says that
\begin{equation*}
    0\leq\int_0^{\nu^{-1}} (\tilde c^\eta)^2 z^{-k\eta-2}\leq (\frac{2}{1+k\eta})^2 \int_0^{\nu^{-1}} (\partial_z(\tilde c^\eta))^2z^{-k\eta}<4 \int_0^{\nu^{-1}} (\partial_z(\tilde c^\eta))^2z^{-k\eta}<\infty.
\end{equation*}
}

Thanks to the boundary condition for $\sigma=1$ ( recall \eqref{sigma-one-nc} ), we can now {apply the corollary of} Hardy's inequality \eqref{Corollary of Hardy's} on $\tilde c^{\eta}$, with parameter power $k\eta$ and $a=\nu^{-1}(s)$, to figure out the coefficient:
\ifthenelse{\boolean{isSimplified}}{\begin{equation*}
\begin{split}
    \int_0^{\frac1\nu}  z^{-k\eta}\tilde c^{2\eta-1} \tilde c_{zz} dz&=\frac{k(k\eta+1)}{2}\int_0^{\nu^{-1}} (\tilde c^\eta)^2 z^{-k\eta-2}-\frac{2\eta-1}{\eta^2}\int_0^{\nu^{-1}} (\partial_z(\tilde c^\eta))^2z^{-k\eta}\\
    &\leq (\frac{k(k\eta+1)}{2} (\frac{2}{1+k\eta})^2-\frac{2\eta-1}{\eta^2})  \int_0^{\nu^{-1}} (\partial_z(\tilde c^\eta))^2z^{-k\eta}\\
    &=\left(\frac{2k}{k\eta+1}-\frac{2\eta-1}{\eta^2}\right)\int_0^{\nu^{-1}} (\partial_z(\tilde c^\eta))^2z^{-k\eta}\leq 0\qquad\text{ if and only if }k\leq 2-\frac{1}{\eta}.
\end{split}
\end{equation*}
}{\begin{equation}
\begin{split}
    \int_0^{\frac1\nu}  z^{-k\eta}\tilde c^{2\eta-1} \tilde c_{zz} dz&=\frac{k(k\eta+1)}{2}\int_0^{\nu^{-1}} (\tilde c^\eta)^2 z^{-k\eta-2}-\frac{2\eta-1}{\eta^2}\int_0^{\nu^{-1}} (\partial_z(\tilde c^\eta))^2z^{-k\eta}\\
    &\leq (\frac{k(k\eta+1)}{2} (\frac{2}{1+k\eta})^2-\frac{2\eta-1}{\eta^2})  \int_0^{\nu^{-1}} (\partial_z(\tilde c^\eta))^2z^{-k\eta}\\
    &=\left(\frac{2k}{k\eta+1}-\frac{2\eta-1}{\eta^2}\right)\int_0^{\nu^{-1}} (\partial_z(\tilde c^\eta))^2z^{-k\eta}\leq 0\qquad\text{ if and only if }k\leq 2-\frac{1}{\eta},
\end{split}
\end{equation}
where we have observed that this coefficient can be bounded above by zero if and only if $0\leq k\leq 2-\frac{1}{\eta}$:
\begin{equation*}
    0\leq k\leq 2-\frac{1}{\eta} \iff \frac{1}{\eta}k=(2-\frac{2\eta-1}{2\eta})k\leq \frac{2\eta-1}{\eta^2}\iff 2k\leq \frac{2\eta-1}{\eta^2}(k\eta+1)\iff \frac{2k}{k\eta+1}\leq \frac{2\eta-1}{\eta^2}.
\end{equation*}}
Finally, we conclude that if $k<2-\frac{1}{\eta}$, then
\begin{equation*}
    \frac{d}{ds}\mathcal T_{k\eta,2\eta}^{2\eta} \leq  \mathcal T_{k\eta,2\eta}^{2\eta} (O(\eta C(z^*) s^{-\frac{h_a}{2}})-k\eta+1). 
\end{equation*}
\end{proof}


\subsection{Soltuions to Differential Inequalities-Choice of Parameters}\label{choice of parameters diffusive}
\ifthenelse{\boolean{isSimplified}}{Since the estimate of \eqref{interior-estimate-of-a-diffusive-nc} is similar to the non-diffusive case \eqref{interior-estimate-of-a-diffusive}, we need to further impose \begin{equation}\label{additional-requirement-on-alpha-nc}
\begin{split}
   \alpha>\alpha_0,\,h_a<2,
\end{split}
\end{equation}
where $\alpha_0\approx 1.88415$ satisfies \eqref{definition-of_alpha_knot}.}{Combining with the estimate of \eqref{interior-estimate-of-a-diffusive-nc}, the definition of trapedness \eqref{smooth:bd:etrap:diffusive-nc}, and the fact that all the linear terms \eqref{interior-estimate-of-a-diffusive-nc}  can be absorbed into the quadratic terms (recall the choice of parameters \eqref{initial-parameter-other-nc} and Young's inequality), this lemma suggests that, in order for the bootstrap argument to work, we need to further impose that
\begin{equation}\label{additional-requirement-on-alpha-nc}
\begin{split}
   &{-\alpha+1}+\frac{1}{\sqrt{\alpha+1}}+\frac{4}{(\alpha+3)}\sqrt{\frac{3}{8(\alpha+1)}}<0 \iff \alpha>\alpha_0,\\
    & (s^{-1})^2=o(s^{-h_a}) \iff h_a<2,
\end{split}
\end{equation}
where $\alpha_0\approx 1.88415$ is the unique positive solution of
\begin{equation}\label{definition-of_alpha_knot-nc}
     {-\alpha_0+1}+\frac{1}{\sqrt{\alpha_0+1}}+\frac{4}{(\alpha_0+3)}\sqrt{\frac{3}{8(\alpha_0+1)}}=0.
\end{equation}}
Similarly, in light of \eqref{smooth:lemma:interior:c:new idea:generalized-nc}, $T_{k\eta_0,2\eta_0}$ decays at most as fast as $e^{(-k/2+1/(2\eta_0))s}$ and, in order for the bootstrap argument \eqref{smooth:bd:etrap:diffusive-nc} to work, we need to further impose
\begin{equation}\label{additional-requirement-on-l-nc}
    e^{(-\frac{k}{2}+\frac{1}{2\eta_0})s}=o(e^{-ls/2})\iff l<k-\frac{1}{\eta_0},\, k>\frac{1}{\eta_0}.
\end{equation}

Summarizing all the conditions \eqref{initial-parameter-weights-nc}, \eqref{initial-parameter-other-nc}, \eqref{interior-parameter-a-nc}, \eqref{condition on k:diffusive-nc}, \eqref{additional-requirement-on-alpha-nc} and \eqref{additional-requirement-on-l-nc} on our parameters, we obtain the conditions
\ifthenelse{\boolean{isSimplified}}{\begin{alignat}{2}\label{final-parameter-conditions-diffusive-nc}
    &\alpha_0<\alpha<3-\frac{2}{\eta_0}\qquad
        && \begin{cases}
            \eta_0\geq 2,\, \eta\in \mathds{N}\\
            \alpha-1+\frac{1}{\eta_0}<k<2-\frac{1}{\eta_0}
        \end{cases}\nonumber\\
        & 1<h_a<2\qquad
        &&0<l< k-\frac{1}{\eta_0}\\
        & \frac{\alpha-1}{2}<\epsilon_a\leq \epsilon_c\qquad
        && \frac{k}{2}+\frac{1}{2\eta_0}<\epsilon_c<1\nonumber.
\end{alignat}}{\begin{subequations}\label{final-parameter-conditions-diffusive-nc}
    \begin{align}
        &\alpha_0<\alpha<3-\frac{2}{\eta_0}\\
        & \begin{cases}
            \eta_0\geq 2,\, \eta\in \mathds{N}\\
            \alpha-1+\frac{1}{\eta_0}<k<2-\frac{1}{\eta_0}
        \end{cases}\\
        & 1<h_a<2\\
        &0<l< k-\frac{1}{\eta_0},\\
        & \frac{\alpha-1}{2}<\epsilon_a\leq \epsilon_c\\
        & \frac{k}{2}+\frac{1}{2\eta_0}<\epsilon_c<1
    \end{align}
\end{subequations}
where $\alpha_0\approx 1.88415$ satisfies \eqref{definition-of_alpha_knot}.}
 One can check that
\begin{equation*}
    (\alpha,\eta_0,k,h_a,l,\epsilon_a,\epsilon_c)=(2,4,3/2,4/3,1,3/4,9/10),
\end{equation*}
satisfies the conditions above. We omit the following proof since it is similar to the non-diffusive case: Lemma \ref{smooth:lemma:interior:diffusive}.

\ifthenelse{\boolean{isSimplified}}{}{We now rigorously prove the previous observations, as well as partially finish the bootstrap argument.\begin{lemma}[Gronwall-type Differential Inequality with Source]\label{Gronwall-type Differential Inequality with Source}
    Let $L$ be a positive function and let $C_1,C_2,h,\delta,\epsilon>0, K\in \R$.
    If $L^2$ satisfies
    \begin{equation}\label{Gronwall-with-source}
        \frac{d}{ds} L^2 +C_1 L^2\leq C_2 s^{-h},\qquad L^2(s_0)\leq \delta^2 s_0^{-h+\epsilon},
    \end{equation}
    for $s\geq s_0$, then for $s_0$ large enough we have
    \begin{equation}
        L^2\leq 2\delta^2 s^{-h+\epsilon}.
    \end{equation}
\end{lemma}

\begin{proof}
   Wwe compute
    \begin{equation*}
        \begin{split}
            \frac{d}{ds}(s^h L^2 e^{C_1/2 s})&=s^h e^{C_1s/2}(\frac{d}{ds}L^2 +C_1/2 L^2+h s^{-1} L^2)\\
            &\leq s^h e^{C_1s/2}(\frac{d}{ds}L^2 +C_1 L^2)\\
            &\leq C_2 e^{C_1s/2},
        \end{split}
    \end{equation*}
    provided that $s$ is large enough. Integrate both hands gives
    \begin{equation*}
        \begin{split}
            L^2&\leq \frac{2C_2}{C_1}s^{-h}-\frac{2C_2}{C_1}s^{-h} e^\frac{C_1(s_0-s)}{2}+s_0^h s^{-h} L^2(s_0) e^\frac{C_1(s_0-s)}{2}\\
            &\leq \frac{2C_2}{C_1} s^{-h} +\frac{s_0^h}{s^h} L^2(s_0)\\
            &\leq \frac{2C_2}{C_1} s^{-h} +\delta^2 s^{-h} s^\epsilon\leq 2\delta^2 s^{-h+\epsilon},
        \end{split}
    \end{equation*}
    for $s$ large enough.\\
\end{proof}}

\begin{corollary}[Interior and Total Bootstrap Arguments]\label{smooth:lemma:interior:diffusive-nc}

Let $(\alpha,\eta_0,k,h_a,l,\epsilon_a,\epsilon_c)$ be in the range of \eqref{final-parameter-conditions-diffusive-nc}. For any $z^*\geq 1$ and $ \delta>0$, there exists a large self-similar time $s_0^*$ such that for any $s_0\geq s_0^*$, for any solution which is trapped on $[s_0,s_1]$, we have for $s\in [s_0,s_1]$:
\begin{eqnarray}\label{smooth:bd:int e1:diffusive-nc}
\mathcal I_a^2(s)  \leq 2\delta^2 s^{-h_a}, \quad   \mathcal T_{k\eta_0,2\eta_0}^{2\eta_0}(s) \leq \frac14 e^{-\eta_0 l s}.
\end{eqnarray}
\end{corollary}
\ifthenelse{\boolean{isSimplified}}{}{\begin{proof}
    Let $0<\kappa\ll 1$ be fixed, for which whose value will be chosen later. In light of \eqref{interior-estimate-of-a-diffusive-nc}, we observe that when $s$ is sufficiently large,
    \begin{equation}\label{trivial-condition-nc}
        s^{-1}+e^{(-\frac{h_b}{2}-1+\epsilon)s}+ e^{(-\frac{l}{2}+\epsilon) s}+ se^{(-1-\frac{h_b}{2})s}+ e^{-\frac{l}{2}s}s^{-1}\lesssim  s^{-1},
    \end{equation}
    where we have used the choice of parameters \eqref{final-parameter-conditions-diffusive-nc}. Choose $s_0$ large enough so that $O(s^{-\frac{h_a}{2}})<\frac{\kappa}{2}$ and that \eqref{trivial-condition-nc} is satisfied. Combining \eqref{trivial-condition-nc} and Young's inequality, we see the linear source term can be bounded by
    \begin{equation}
    \begin{split}
          &C(z^*) I_a(  s^{-1}+e^{(-\frac{h_b}{2}-1+\epsilon)s}+ e^{(-\frac{l}{2}+\epsilon) s}+ e^{-\frac{l}{2}s}s^{-1})\\
          &\qquad\lesssim C(z^*)s^{-1}\\
          &\qquad\leq C(z^*) (\frac{s^{-2}}{2\frac{\kappa}{C(z^*)}}+\frac{\frac{\kappa}{C(z^*)}}{2}I_a^2)\\
          &\qquad\leq C(z^*)s^{-2}+\kappa/2 I_a^2,
    \end{split}
    \end{equation}
    for $s_0$ large enough. Now \eqref{interior-estimate-of-a-diffusive-nc} is reduced to
    \begin{equation}
         \frac{d}{ds}I_a^2+I_a^2({\alpha-1}-\frac{1}{\sqrt{\alpha+1}}-\frac{4}{(\alpha+3)}\sqrt{\frac{3}{8(\alpha+1)}}-\kappa)\leq C(z^*)s^{-2},
    \end{equation}
    from which Gronwall-type \eqref{Gronwall-with-source} of lemma can be applied (Recall \eqref{final-parameter-conditions-diffusive-nc}), provided that $\kappa$ is small enough such that ${\alpha-1}-\frac{1}{\sqrt{\alpha+1}}-\frac{4}{(\alpha+3)}\sqrt{\frac{3}{8(\alpha+1)}}-\kappa>0$.

    For the second and third inequality, again, we let $0<\kappa\ll 1$ be fixed and choose $s_0$ large enough so that $O(s^{-\frac{h_a}{2}})<\frac{\kappa}{2}$. Now, for $s$ large enough, with choice of \eqref{final-parameter-conditions-diffusive-nc}, \eqref{smooth:lemma:interior:c:new idea:generalized-nc} reads
    \begin{equation}
      \frac{d}{ds}T_{k\eta_0,2\eta_0}^{2\eta_0} \leq  T_{k\eta_0,2\eta_0}^{2\eta_0}(\kappa+1-k\eta_0), 
    \end{equation}
    and  Gronwall's lemma \ref{Gronwall-type Differential Inequality with Source} and \eqref{smooth:bd:etrap:diffusive-nc} gives
    \begin{equation}
    \begin{split}
         & T_{k\eta_0,2\eta_0}^{2\eta_0}(s)\leq T_{k\eta_0,2\eta_0}^{2\eta_0}(s_0) e^{(-k\eta_0+1+\kappa)(s-s_0)}\leq \frac{1}{4}e^{-\eta_0 ls_0} e^{-\eta_0 l(s-s_0)}=\frac{1}{4}e^{-\eta_0 ls},
    \end{split}
    \end{equation}
    provided that  $1-k\eta_0+\kappa<-\eta_0 l$.
\end{proof}}

\subsection{Exterior Estimate}
\begin{lemma}[Exterior Estimate-Exterior Bootstrap Argument]\label{smooth:lemma:main-critical:diffusive-nc}
Let $(\alpha,\eta_0,k,h_a,l,\epsilon_a,\epsilon_c)$ be in the range of \eqref{final-parameter-conditions-diffusive-nc}. There exists $\bar z^*\geq 1$, and for any $z^*\geq \bar z^*$, a $\delta^*>0$ such that for $0<\delta\leq \delta^*$ the following holds true. There exists $s_0^*$ large enough such that if a solution is trapped on $[s_0,s_1]$ with $s_0\geq s_0^*$, for any time $s_0\leq s\leq s_1$ we have

\begin{eqnarray}\label{bd:int e-critical2:diffusive-nc}
\mathcal E_a^2(s)\leq \frac14 s^{-{h_a}}.
\end{eqnarray}

\end{lemma}

\begin{proof}
\ifthenelse{\boolean{isSimplified}}{Mimick the non-diffusive case: Lemma \ref{smooth:lemma:main-critical:diffusive}.}{The proof relies on the maximum principle. We rewrite \eqref{equation:epsilon-a} as
\begin{equation}\label{smooth:equation:epsilon-2-nc}
    \begin{split}
        &\tilde a_s +\mathcal L_a \tilde a    = F_a.
    \end{split}
    \end{equation}
where the transport operator $\mathcal L_a$ (note that it has nonlinear parts) and the source terms $F_a$ and $F_b$ are:
\begin{align*}
&\mathcal L_a v=  - \frac{\lambda_s}{\lambda}v - \frac{\nu_s}{\nu} z v_z - 2\phi v + \partial_z^{-1} \phi v_z    - \tilde a v + \partial_z^{-1} \tilde a v_z ,\\
&F_a = -\partial_z^{-1} \tilde a \phi' +(\frac{\lambda_s}{\lambda} +1)\phi +  \frac{\nu_s}{\nu}z\phi' - 2\nu \int_0^{\frac{1}{\nu}} (\phi+ \tilde{a})^2(z) dz
        \\
         &\qquad  + \lambda\left( - \frac{\nu}{\lambda}\partial_z^{-1} \tilde{c}  + \int_0^{\frac{1}{\nu}} \frac{\nu^2}{\lambda} \partial_z^{-1} \tilde{c}(z) dz\right),\\
\end{align*}

\noindent \textbf{Step 1}. \emph{A supersolution for $\pa_s+\mathcal L_a$ and $\pa_s+\mathcal L_b$ on $[z^*,\nu^{-1}]$}. We introduce
$$
f(s,z)= \frac 12 s^{-h_a/2}
$$
and claim that there exists $z^*$ large enough such that for $s_0$ large enough, for all $ s_0\leq s \leq s_1$ and $z\geq z^*$:
\be \label{smooth:exterior:id:supersolution:diffusive-nc}
(\pa_s +\mathcal L_a) f\geq \frac{1}{4} s^{-h_a/2}.
\ee
To prove \eqref{smooth:exterior:id:supersolution:diffusive-nc}, we compute using \eqref{smoothmodulationequations:diffusive-nc} and \eqref{asy on lambda derivative-nc}:
\be
\begin{split}
    (\pa_s +\mathcal L_a) f &= (-\frac{h_a}{4}s^{-h_a/2-1})+(-\frac{\lambda_s}{\lambda}-2e^{-z}-\tilde{a})\frac{1}{2}s^{-h_a/2}\\
    &\geq (-\frac{h_a}{2}s^{-1}+1-Cs^{-1}-2e^{-z^*}-C(z^*)s^{-h_a/2})\frac{1}{2}s^{-h_a/2}\\
    &\geq \frac{1}{4} s^{-h_a/2},
\end{split}
\ee

which implies \eqref{smooth:exterior:id:supersolution:diffusive-nc} upon first taking $z^*$ large enough and then $s_0^*$ large enough. \\

\noindent \textbf{Step 2}. \emph{Estimate for the source term}. For any $\kappa,\epsilon>0$, we claim that for $z^*$ large enough and then for $\delta$ small enough, for all $s_0\leq s \leq s_1$ and  $z\in [z^*,\nu^{-1}]$:
\be \label{smooth:exterior:bd:source:diffusive-nc}
|F_a(s,z)|\leq \kappa{s^{-\frac{h_a}{2}}}.
\ee
Let $0<\kappa$ be fixed. We inject the improved bootstrap bound \eqref{smooth:bd:int e1:diffusive-nc} in the computation \eqref{smooth:bd:varepsiloninterior:diffusive} and get:
\be \label{smooth:exterior:bd:interiorvarepsilon:diffusive-nc}
\begin{split}
    |\tilde a (z)|&\leq \frac{z^\frac{\alpha+1}{2}}{(\alpha+1)^\frac{1}{2}} I_a(s)\leq C \delta s^{-h_a/2} {z^*}^\frac{\alpha+1}{2} \quad \mbox{for }z\in [0,z^*],\\
\end{split}
\ee
where we have used $h_b<2$ and taken $s$ (may depend on $z^*$ and $\delta$) large enough.
Combining with 
\begin{equation}\label{a-b-exterior:diffusive-nc}
\begin{split}
    |\tilde a(z)|&\leq s^{-\frac{h_a}{2}} \qquad\text{for $z\in [z^*,\nu^{-1}]$}\\
\end{split}  
\end{equation}
using $\phi(z)=e^{-z}$:
\begin{equation} \label{smooth:exterior:bd:inter1:diffusive}
\begin{split}
    |\partial_z^{-1} \tilde a \phi'| &\leq (\|\tilde a\|_{L^\infty[0,z^*]}+\|\tilde a\|_{L^\infty[z^*,\nu^{-1}]}) z |\phi'(z)|\\
    &\leq \Big(({z^*}^\frac{\alpha+1}{2}) \delta s^{-\frac {h_a}{2}}+s^{-\frac {h_a}{2}} \Big)z^*e^{-z^*} \leq  \kappa s^{-\frac{h_a}{2}} \qquad\text{for $z\in [z^*,\nu^{-1}]$},
\end{split}
\end{equation}
where $\kappa>0$ and we have taken $z^*$ large enough such that all of $z^*e^{-z^*}<\kappa/2$, and that $z^{1}e^{-z}$ is decreasing for $z>z^*$ are satisfied, and $\delta$ small enough such that $\delta {z^*}^\frac{\alpha+1}{2}<1$. Next, using \eqref{smooth:bd:integral:diffusive-nc}, \eqref{estimate-on-primitive-of-c-diffusive-1-nc}, \eqref{smoothmodulationequations:diffusive-nc}, and \eqref{a-b-exterior:diffusive-nc}, for all $z\geq z^*$ and $z^*$ large enough:
\be \label{smooth:exterior:bd:inter2:diffusive-nc}
\begin{split}
   &\left|(\frac{\lambda_s}\lambda +1)\phi \right|  \leq Cs^{-1}\leq \kappa s^{-\frac{h_a}{2}},
   \\
   &\left|\frac{\nu_s}\nu z\phi'  \right|\leq Cs^{-1}\cdot1 \leq Cs^{-1}\leq \kappa s^{-\frac{h_a}{2}},
   \\
   &\left|- 2\nu\Big(\int_0^{\frac{1}{\nu}} (\phi+\tilde a)^2(z)dz\Big)\right|\leq Cs^{-1}\leq \kappa s^{-\frac{h_a}{2}},\\
   &\left|\nu \pa_z^{-1} \tilde c\right|\leq C s^1 e^{-\frac{l}{2}s}  \leq Cs^{-1}\leq \kappa s^{-\frac{h_a}{2}},
   \\
   &\left|\int_0^{\frac{1}{\nu}} \nu^2 \partial_z^{-1} \tilde{c}(z) dz\right|\leq e^{(-\frac{l}{2}+\epsilon) s} \leq Cs^{-1}\leq \kappa s^{-\frac{h_a}{2}},
   \\
\end{split}
\ee
where we have taken $s_0^*$ to be large enough that may depend on $z^*$.
Combining \eqref{smooth:exterior:bd:inter1:diffusive-nc} and \eqref{smooth:exterior:bd:inter2:diffusive-nc} shows  \eqref{smooth:exterior:bd:source:diffusive-nc}.\\

\noindent \textbf{Step 3}. \emph{End of the proof}. We introduce
\be \label{smoothexterior:def:fpm:diffusive-nc}
f^\pm= \pm \left( f-\tilde a \right)
\ee
and try to apply maximum principle on them.
Using \eqref{smooth:equation:epsilon-2-nc}, \eqref{smooth:exterior:id:supersolution:diffusive-nc} and \eqref{smooth:exterior:bd:source:diffusive-nc} one obtains that for $s_0\leq s\leq s_1$ and $z\in [z^*,\nu^{-1}]$:
\be \label{smooth:exterior:inter1:diffusive-nc}
\begin{split}
  &(\pa_s+\mathcal L_a)f^+= (\pa_s+\mathcal L_a)f - F_a\geq (\frac{1}{4}-\kappa) s^{-\frac{h_a}{2}}\geq 0 \quad \mbox{and similarly} \quad (\pa_s+\mathcal L)f^-\leq 0,
  \\
\end{split}
\ee
provided that $\kappa \leq 1/4$. We estimate
\begin{equation}
\begin{split}
    -\frac{\nu_s}{\nu}z^* + \int_0^{z^*} (\phi+\tilde a)(\tilde z) d\tilde z &\geq -\frac{\nu_s}{\nu^2}z^* \nu + 1-e^{-z^*}-\norm{\tilde{a}}_{L^\infty[0,\nu^{-1}]}z^*\\
    &\geq (1-C(z^*)s^{-h_a+1})z^*(Cs^{-1})+1-e^{-z^*}-C(z^*)z^*s^{-h_a/2}\\
    &\geq 1-e^{-z^*}-(Cs^{-1}-Cs^{-1} C(z^*) s^{-h_a+1}-C(z^*) s^{-h_a/2})z^*\\
    &\geq 1-e^{-z^*}-Cs^{-1}z^*\geq 0,
\end{split}  
\end{equation}
where we have first taken $z^*$ sufficiently large, then $s$ sufficiently large. From this we know that the particles are always moving from region $0\leq z\leq z^*$ to $z^*\leq z \leq \frac1\nu$.
At the boundary $z=z^*$ one has using \eqref{smooth:exterior:bd:interiorvarepsilon:diffusive-nc} that:
\be \label{smooth:exterior:inter2-nc}
\begin{split}
   &f^+(s,z^*)= \frac{s^{-\frac{h_a}{2}}}{2}-\tilde a (s,z^*)\geq (\frac{1}{2}-C \delta  {z^*}^\frac{\alpha+1}{2}) s^{-h_a/2}\geq 0 \quad \mbox{and similarly} \quad f^-(s,z^*)\leq 0, 
   \\
\end{split}
\ee
provided $\delta$ is small enough depending on $z^*$. At initial time $s=s_0$, we have using \eqref{bd:eini-beta=0:diffusive-nc} that for all $z\in [z^*,\nu_0^{-1}]$:
\begin{equation}\label{smooth:exterior:inter3:diffusive-nc}
\begin{split}
&f^+(s_0,z)\geq \frac{s_0^{-\frac{h_a}{2}}}{2}-\| \tilde a_0\|_{L^\infty [z^*,\nu_0^{-1}]}\geq \frac{s_0^{-\frac{h_a}{2}}}{2}-\frac{s_0^{-\frac{h_a}{2}}}{4}\geq 0,  \quad \mbox{and similarly} \quad f^-(s_0,z)\leq 0,
\\
\end{split}
\end{equation}
At the point $z(s)=\frac1{\nu(s)}$, thanks to the boundary condition \eqref{raw-nc}, the characteristics of the full transport field stays on
the boundary (or the characteristics curves passing any points in the interior of the domain cannot be issued from the boundary $z=\nu^{-1}(s)$, so it can only be issued from $(s_0,\nu^{-1}(s_0))$) since
\begin{equation}\label{partical-stay-label}
    \frac{d}{ds}\frac1{\nu(s)}= -\frac{\nu_s}{\nu^2} = -\frac{\nu_s}{\nu}\nu^{-1} + \underbrace{\int_0^{\frac1\nu} (\phi+\varepsilon) dz}_{0}=\text{transport coefficient}.
\end{equation}
This together with \eqref{smooth:exterior:inter3:diffusive-nc} imply that $f^+(s,\frac1\nu)\geq 0$ and $f^-(s,\frac1\nu)\leq 0$.
Therefore, in view of \eqref{smooth:exterior:bd:inter1:diffusive-nc}, \eqref{smooth:exterior:bd:inter2:diffusive-nc} and \eqref{smooth:exterior:inter3:diffusive-nc} one can apply the maximum principle and obtain that $f^+(s,z)\geq 0$ and $f^-(s,z)\leq 0$ for all $s_0\leq s \leq s_1$ and $z^*\leq z \leq \nu^{-1}$. By the definition \eqref{smoothexterior:def:fpm:diffusive-nc}  of $f^\pm$, this implies the desired estimate \eqref{bd:int e-critical2:diffusive-nc} and completes the proof of the Lemma.
}

\end{proof}

\subsection{Conclusion of Bootstrap Argument-Improvement on the Bounds of the Temperature}
\ifthenelse{\boolean{isSimplified}}{}{In view of Lemma \ref{lemma:smoothmodulation:diffusion-nc}, Lemma \ref{smooth:lemma:interior:diffusive-nc} and Lemma \ref{smooth:lemma:main-critical:diffusive-nc}, Definition \ref{def:crit-beta=0:diffusive-nc} and Definition \ref{def:ini crit-beta=0:diffusive-nc}, the bootstrap argument is completed.

\begin{proposition} \label{smooth:pr:bootstrap:diffusive-nc}

Let $(\alpha,\eta_0,k,h_a,l,\epsilon_a,\epsilon_c)$ be in the range of \eqref{final-parameter-conditions-diffusive-nc}. There exist universal constants $z^*\geq 1$, $\delta(z^*)>0$ and $s_0^*(z^*,\delta)\geq 0$ such that the following holds true. For all $s_0\geq s_0^*$, any solution of \eqref{raw-nc} which is initially close to the blowup profile in the sense of Definition \ref{def:ini crit-beta=0:diffusive-nc} is trapped on $[s_0,+\infty)$ in the sense of Definition \ref{def:crit-beta=0:diffusive-nc}.
\end{proposition}
\begin{proof}
Similar to the non-diffusive case.
\end{proof}}

For convenience, let us now fix
\begin{equation}\label{particular-parameter-choice-nc}
    (\eta_0,k,h_a,l,\epsilon_a,\epsilon_c)=(4,\frac{\alpha+1}{2},\frac{3}{2},\frac{2\alpha+1}{8},\frac{10\alpha+3}{32},\frac{2\alpha+11}{16}),\, \alpha_0<\alpha<\frac{5}{2},
\end{equation}
while we allow
\begin{equation}
    \eta\geq \eta_0=4 \text{ to vary.}
\end{equation}
We can check that for each $1<\alpha<\frac{5}{2}$, it is indeed in the range of \eqref{final-parameter-conditions-diffusive-nc}, since most values are conveniently taken to be midpoints of the lower and higher ends.\par
\begin{remark}\label{continuous remark}
    In this case, since $\epsilon_c >\frac{k}{2}$, we see that $\tilde c(s,z)z^{-\frac{k}{2}}=\tilde c(s,z)z^{-\frac{\alpha+1}{4}}$ is continuous at zero.
\end{remark}

\begin{proposition}\label{l infinity corollary nc}
   Let $\alpha_0<\alpha<\frac{5}{2}$ and $  (\eta_0,k,h_a,l,\epsilon_a,\epsilon_c)$ be defined in \eqref{particular-parameter-choice-nc}.  There exist universal constants $z^*\geq 1$, $\delta(z^*)>0$ and $s_0^*(z^*,\delta)\geq 0$ such that the following holds true. For all $s_0\geq s_0^*$, any solution of \eqref{raw-nc} which is initially close to the blowup profile in the sense of Definition \ref{def:ini crit-beta=0:diffusive-nc} is trapped on $[s_0,+\infty)$ in the sense of Definition \ref{def:crit-beta=0:diffusive-nc}. Furthermore, $\norm{\tilde c}_{L^\infty(0,\nu^{-1}(s))}$ decays exponentially in time, provided that the initial condition is chosen accordingly:
   \begin{equation}\label{initial maximum norm estimate nc}
        \norm{\frac{\tilde c(s_0,z) }{z^{\frac{\alpha+1}{4}}}}_{L^{\infty}(0,\nu^{-1}(s_0))}\leq e^{(-\frac{\alpha+1}{16})(s_0)} \implies  \norm{\tilde c(s,z) }_{L^{\infty}(0,\nu^{-1}(s))}\leq C e^{(-\frac{\alpha+1}{16})s} s^{\frac{\alpha+1}{4}},
   \end{equation}
   for some universal constant $C$.
\end{proposition}
\begin{proof}
\ifthenelse{\boolean{isSimplified}}{In a similar fashion as in Proposition \ref{smooth:pr:bootstrap:diffusive}, in view of Lemma \ref{lemma:smoothmodulation:diffusion-nc}, Lemma \ref{smooth:lemma:interior:diffusive-nc} and Lemma \ref{smooth:lemma:main-critical:diffusive-nc}, Definition \ref{def:crit-beta=0:diffusive-nc} and Definition \ref{def:ini crit-beta=0:diffusive-nc}, the first part of this proposition immediately follows.}{}

Applying Lemma \ref{differential inequality of c-nc} again with $\eta\geq \eta_0$:
\begin{equation}
      \frac{d}{ds}\mathcal T_{\frac{\alpha+1}{2}\eta,2\eta}\leq  \mathcal T_{\frac{\alpha+1}{2}\eta,2\eta} (O( C(z^*) s^{-\frac{1}{2}})-\frac{\frac{\alpha+1}{2}}{2}+\frac{1}{2\eta}),
\end{equation}
Choose $s$ large enough such that $1\gg\kappa>O(C(z^*) s^{-1/2})$, and this inequality gives
\begin{equation}\label{L 2eta estimate}
   \mathcal T_{\frac{\alpha+1}{2} \eta,2\eta}(s)\leq \mathcal T_{\frac{\alpha+1}{2}\eta,2\eta}(s_0) e^{(-\frac{\frac{\alpha+1}{2}}{2}+\frac{1}{2\eta}+\kappa)(s-s_0)}
\end{equation}
Therefore, since we can rewrite
\begin{equation}
    \mathcal T_{2\eta,\frac{\alpha+1}{2}\eta}=(\int_0^{\nu^{-1}(s)} \tilde c^{2\eta} z^{-\frac{\alpha+1}{2} \eta})^\frac{1}{2\eta}
    =\norm{\tilde c(s,\cdot) (\cdot)^{-\frac{\frac{\alpha+1}{2}}{2}}}_{L^{2\eta}(0,\nu^{-1}(s))}.
\end{equation}
Since $\tilde c(s,z)z^{-\frac{\alpha+1}{4}}$ is continuous at $z=0$ (Remark \ref{continuous remark}) and thus $\norm{\tilde c(s,z)z^{-\frac{\alpha+1}{4}}}_{L^\infty(0,\nu^{-1}(s))}$ exists for all $s\geq s_0$. Moreover, we observe that $\norm{\tilde c(s,z)z^{-\frac{\alpha+1}{4}}}_{L^{2\eta_0}(0,\nu^{-1}(s))}$ also exists for all $s\geq s_0$ and $e^{(-\frac{\frac{\alpha+1}{2}}{2}+\frac{1}{2\eta}+\kappa)(s-s_0)}\leq e^{(-\frac{\frac{\alpha+1}{2}}{2}+\frac{1}{8}+\kappa)(s-s_0)}$.  By results from measure theory, taking $\eta\rightarrow \infty$ in \eqref{L 2eta estimate} shows that
\begin{equation}
\resizebox{.93 \textwidth}{!}{$\displaystyle\norm{\tilde c(s,z) z^{-\frac{\alpha+1}{4}}}_{L^{\infty}(0,\nu^{-1}(s))}\leq \norm{\frac{\tilde c(s_0,z) }{z^{\frac{\alpha+1}{4}}}}_{L^{\infty}(0,\nu^{-1}(s_0))} e^{(-\frac{\alpha+1}{8}+\kappa)(s-s_0)}\leq \norm{\frac{\tilde c(s_0,z) }{z^{\frac{\alpha+1}{4}}}}_{L^{\infty}(0,\nu^{-1}(s_0))} e^{(-\frac{\alpha+1}{16})(s-s_0)}.$}
\end{equation}
Therefore, if initially
\begin{equation}
    \norm{\frac{\tilde c(s_0,z) }{z^{\frac{\alpha+1}{4}}}}_{L^{\infty}(0,\nu^{-1}(s_0))}\leq e^{(-\frac{\alpha+1}{16})(s_0)},
\end{equation}
then for $s\geq s_0$,
\begin{equation}
     \norm{\tilde c(s,z) z^{-\frac{\alpha+1}{4}}}_{L^{\infty}(0,\nu^{-1}(s))}\leq e^{(-\frac{\alpha+1}{16})s}\implies \norm{\tilde c(s,z) }_{L^{\infty}(0,\nu^{-1}(s))}\leq C e^{(-\frac{\alpha+1}{16})s} s^{\frac{\alpha+1}{4}}.
\end{equation}
\end{proof}

\section{Stability of Blowup Result}
\subsection{Re-decomposition of Zero-average Velocity and Temperature}
Given any initial condition that has the form of
\begin{equation}\label{arbitrary zero-average initial condition}
    \begin{split}
        a_0(Z)&=\frac{1}{\widetilde \lambda_0}\phi(\frac{Z}{\widetilde\nu_0})+\tilde a_0(Z)\\
        c_0(Z)&=\tilde c_0(Z),
    \end{split}
\end{equation}
{where the domain $Z\in[0,1]$, and $a_0$ has zero average on the whole domain:
\begin{equation}\label{definition of zero average}
    \int_0^1 a_0(Z)dZ=0,
\end{equation}
} we need to find a new decomposition $(\bar \lambda_0,\bar \nu_0,\bar a_0,\bar c_0)$ such that
\begin{equation}\label{redecomposition of initial condition}
\begin{split}
  &a_0(Z)=\frac{1}{\bar \lambda_0}\phi(\frac{Z}{\bar\nu_0})+\bar a_0(Z)\\
     & c_0(Z)=\frac{1}{\bar \lambda_0^{1+\sigma}} \bar c_0(Z)\\
     & Z=z\bar \nu_0,\, \bar \lambda_0=\bar s_0 e^{-\bar s_0}
\end{split}
\end{equation}
$\bar \lambda_0\bar a_0$ and $\bar c_0$ are initial close and \text{ satisfy Definition \ref{def:ini crit-beta=0:diffusive} or  Definition \ref{def:ini crit-beta=0:diffusive-nc} respectively,} and thus the propositions we established earlier can be applied to obtain some ``stability'' result.\par
To obtain such decomposition, in addition, we need to find an upper bound of $\tilde \lambda_0$, a range of $\tilde \nu_0$ depending on each $\tilde \lambda_0$, and suitable function spaces, from which the initial perturbations $\tilde a_0$ and $\tilde c_0$ are chosen. The proof of the following lemma is similar to the corresponding proof found in \cite{Collot2023}.

\begin{lemma}\label{decomposition lemma}
    For each fixed choices of parameters in the range of \eqref{final-parameter-conditions-diffusive} (or \eqref{particular-parameter-choice-nc}), let $z^*$, $\delta(z^*)$, $s_0^{**}(z^*,\delta(z^*))$ be given in Proposition \ref{smooth:pr:bootstrap:diffusive}  (or \ifthenelse{\boolean{isSimplified}}{Proposition}{Corollary} \ref{l infinity corollary nc}), where $s_0^{**}$ is the threshold self-similar time. There exists a $\lambda_0^*(z^*,\delta(z^*),s_0^{**}(z^*,\delta(z^*)))$ small enough, such that it can be written as $\lambda_0^*=s_0^* e^{-s_0^*}$ for some $s_0^*\geq \max(1,s_0^{**}(z^*,\delta(z^*)))$, and the following holds. For each $\tilde \lambda_0<\lambda_0^*/2$ and $\tilde \nu_0$ satisfies $\frac{3}{2N_0} \leq {\tilde \nu_0}\log(\tilde \lambda_0^{-1}) \leq \frac{1}{2} N_0$, if initial conditions is decomposed in \eqref{arbitrary zero-average initial condition} and satisfies \eqref{definition of zero average}, then there exists another  decomposition $(\bar \lambda_0,\bar \nu_0,\bar a_0,\bar c_0,\bar s_0)$ such that $\bar s_0\geq s_0^*\geq s_0^{**}$ and \eqref{redecomposition of initial condition} holds. Furthermore, there exists $\kappa^*=\kappa^*(\tilde \lambda_0,z^*,\delta(z^*))$, such that if $\kappa\leq \kappa^*$ and initial original perturbations satisfy 
    \begin{equation}
        \begin{split}   
  &\tilde a_0 \in X_{\tilde \epsilon_a,\kappa}=\set{f\in C^2[0,1]:{\norm{f}}_{C^{1,\tilde \epsilon_a}[0,1]}<\kappa,f'(Z=0)=0}\\
   &\tilde c_0 \in 
         X_{\tilde \epsilon_c,\kappa}=\set{f\in C^2[0,1]:\norm{f}_{C^{1,\tilde \epsilon_c}[0,1]}<\kappa,\, f(Z=0)=f'(Z=0)=0}\text{ if }\sigma=0\\
         &\left(\tilde c_0\in X_{\tilde \epsilon_c,\kappa}=\set{f\in C^3[0,1]:\norm{f}_{C^{0,\tilde \epsilon_c}[0,1]}<\kappa,\, f(Z=0)= f(Z=1)=0}\text{ if }\sigma=1\right)
\end{split}
    \end{equation}
    for some $1>\tilde \epsilon_a> \epsilon_a$ and $1>\tilde \epsilon_c> \epsilon_c$, then
    \begin{enumerate}
        \item The new decomposition $\bar\lambda_0\bar a_0,\bar c_0,\bar \lambda_0,\bar\nu_0$ satisfy the definition of initial closeness with additional parameters  $(z^*,\delta(z^*),\bar s_0;N_0)$: Definition \ref{def:crit-beta=0:diffusive} (Definition \ref{def:crit-beta=0:diffusive-nc}).
        \item The decaying speed conditions at $z=0$ Equation \eqref{decay-speed-condition} (or \eqref{decay speed condition nc}) are satisfied for $\bar\lambda_0\bar a_0,\bar c_0$ with parameters $\epsilon_a,\epsilon_c$.
        \item For diffusive case, the assumption of the implication of weighted $L^\infty$ estimate  \eqref{initial maximum norm estimate nc} is satisfied.
    \end{enumerate}
\end{lemma}
\begin{proof}
\noindent \textbf{Step 1}. \emph{Setup}.    Let $s_0^{**}\gg 1$ be the threshold self-similar time in Proposition \ref{smooth:pr:bootstrap:diffusive} or \ifthenelse{\boolean{isSimplified}}{Proposition}{Corollary} \ref{l infinity corollary nc} and parameters
$\widetilde\lambda_0,\tilde \nu_0$ satisfy
\begin{equation}\label{parameter-decomposition-restriction}
    s_0^*\geq \max(s_0^{**},1),\,e^{-1}>\lambda_0^*=s_0^* e^{-s_0^*}\leq s_0^{**} e^{-s_0^{**}},\,\widetilde\lambda_0\leq \frac{\lambda_0^*}{2},\, \frac{3/2}{N_0} \leq {\tilde \nu_0}\log(\tilde \lambda_0^{-1}) \leq N_0(1/2).
\end{equation} 
\ifthenelse{\boolean{isSimplified}}{}{We choose perturbation from the space
\begin{equation}
  \tilde a_0 \in X_{\tilde \epsilon_a,\kappa}=\set{f\in C^2[0,1]:\norm{f}_{C^{1,\tilde \epsilon_a}[0,1]}<\kappa,\, f'(Z=0)=0},\text{ for } 1>\tilde \epsilon_a> \epsilon_a.
\end{equation}
We shall find another decomposition of $a_0$, as well as $(\bar \lambda_0,\bar \nu_0, \bar a_0)$ such that
\begin{equation}
    a_0(Z)=\frac{1}{\bar \lambda_0}\phi(\frac{Z}{\bar\nu_0})+\bar a_0(Z),\text{ and }\bar \lambda_0\bar a_0\text{ satisfies Definition \ref{def:ini crit-beta=0:diffusive} or  Definition \ref{def:ini crit-beta=0:diffusive-nc} respectively.}
\end{equation}}
For simplicity, let us define $\bar \epsilon_0$ using the spatial self-similar coordinate system:
\begin{equation}
    \bar \epsilon_0 (z)=\bar \lambda_0\bar a_0(\bar \nu_0 z)=\bar \lambda_0\bar a_0(Z),\, Z=\bar \nu_0 z.
\end{equation}
\noindent \textbf{Step 2}. \emph{Another decomposition}. We first decompose $\bar \epsilon_0$ in two different ways: For $1<z\leq \frac{1}{\bar \nu_0}$,
\ifthenelse{\boolean{isSimplified}}{\begin{equation}
    \begin{split}
        \bar \epsilon_0(z)&=\bar \epsilon_1(z)+\bar \epsilon_2(z),\,\bar \epsilon_1(z)=\frac{\bar \lambda_0}{\widetilde \lambda_0}\phi(\frac{\bar \nu_0}{\widetilde \nu_0}z)-\phi(z),\,\bar \epsilon_2(z)=\bar \lambda_0 \tilde a_0(\bar \nu_0 z),
    \end{split}
\end{equation}}{\begin{equation}
    \begin{split}
        \bar \epsilon_0(z)&=\bar \epsilon_1(z)+\bar \epsilon_2(z)\\
        \bar \epsilon_1(z)&=\frac{\bar \lambda_0}{\widetilde \lambda_0}\phi(\frac{\bar \nu_0}{\widetilde \nu_0}z)-\phi(z)\\
        \bar \epsilon_2(z)&=\bar \lambda_0 \tilde a_0(\bar \nu_0 z),
    \end{split}
\end{equation}}
as well as for $0\leq z \leq 1$,
\ifthenelse{\boolean{isSimplified}}{\begin{alignat}{2}
     \bar \epsilon_0(z)&= \epsilon_1(z)+ \epsilon_2(z)+\epsilon_3(z)+ \epsilon_4(z)  \\
        \epsilon_1(z)&=\frac{\bar \lambda_0}{\widetilde \lambda_0}-1+\bar \lambda_0 \tilde a_0(0)+z(1-\frac{\bar \lambda_0}{\widetilde \lambda_0}\frac{\bar \nu_0}{\widetilde \nu_0})\quad&&\epsilon_2(z)=(\frac{\bar \lambda_0}{\widetilde \lambda_0}-1)(\phi(\frac{\bar \nu_0}{\widetilde \nu_0} z)-1+\frac{\bar \nu_0}{\widetilde \nu_0}z)\\
        \epsilon_3(z)&=\phi(\frac{\bar \nu_0}{\widetilde \nu_0} z)-\phi(z)+(\frac{\bar \nu_0}{\widetilde \nu_0}-1)z &&\epsilon_4(z)=\bar \lambda_0(\tilde a_0(z\bar \nu_0)-\tilde a_0(0)).
\end{alignat}}{\begin{equation}
    \begin{split}
        \bar \epsilon_0(z)&= \epsilon_1(z)+ \epsilon_2(z)+\epsilon_3(z)+ \epsilon_4(z)\\
        \epsilon_1(z)&=\frac{\bar \lambda_0}{\widetilde \lambda_0}-1+\bar \lambda_0 \tilde a_0(0)+z(1-\frac{\bar \lambda_0}{\widetilde \lambda_0}\frac{\bar \nu_0}{\widetilde \nu_0})\\
        \epsilon_2(z)&=(\frac{\bar \lambda_0}{\widetilde \lambda_0}-1)(\phi(\frac{\bar \nu_0}{\widetilde \nu_0} z)-1+\frac{\bar \nu_0}{\widetilde \nu_0}z)\\
        \epsilon_3(z)&=\phi(\frac{\bar \nu_0}{\widetilde \nu_0} z)-\phi(z)+(\frac{\bar \nu_0}{\widetilde \nu_0}-1)z\\
        \epsilon_4(z)&=\bar \lambda_0(\tilde a_0(z\bar \nu_0)-\tilde a_0(0)).
    \end{split}
\end{equation}}
Setting $z=0$ and use $\phi=e^{-z}$ in the previous two equations, we observe
\begin{equation*}
    \epsilon_1(0)= 0\iff \frac{\bar \lambda_0}{\widetilde \lambda_0}-1+\bar \lambda_0 \tilde a_0(0)=0.
\end{equation*}
Similarly, computing the derivatives of $\partial_z\epsilon_1(0), \partial_z\epsilon_2(0)$ and $\partial_z\epsilon_1(0)... \partial_z\epsilon_4(0)$ gives
\begin{equation*}
   \partial_z\epsilon_1(0)= 0\iff 1-\frac{\bar \lambda_0}{\tilde \lambda_0}\frac{\bar \nu_0}{\tilde \nu_0}=0.
\end{equation*}
Therefore, $\bar \lambda_0$ and $\bar \nu_0$ (and thus $\bar a_0$) are uniquely determined by
\begin{equation}\label{parameter-new-decomposition}
\begin{split}
     \bar \lambda_0&=(\tilde \lambda_0^{-1}+\tilde a_0(0))^{-1}=\tilde \lambda_0 (1+O(\kappa \tilde \lambda_0))\\
     \bar \nu_0&=\frac{\tilde \lambda_0}{\bar \lambda_0} \tilde \nu_0=\tilde \nu_0 (1+O(\kappa \tilde \lambda_0)),
\end{split}
\end{equation}
for $0<{\kappa\tilde \lambda_0 }\ll 1$, where we have used $(1+O(\kappa \tilde \lambda_0))^{-1}=1+O(\kappa \tilde \lambda_0) $. Notice that these two equations immediately imply that $\epsilon_1\equiv 0$. \par
\noindent \textbf{Step 3}. \emph{Computations and Estimates}. We now compute each term individually. For $z\geq 1$, 
\ifthenelse{\boolean{isSimplified}}{\begin{equation}\label{exterior decomposition}
    \begin{split}
          \abs{\bar \epsilon_1(z)}&\leq \abs{\frac{\bar \lambda_0}{\widetilde \lambda_0}\phi(\frac{\bar \nu_0}{\widetilde \nu_0}z)-\phi(\frac{\bar \nu_0}{\widetilde \nu_0}z)}+\abs{\phi(\frac{\bar \nu_0}{\widetilde \nu_0}z)-\phi(z)}\\
    &\leq \phi(\frac{\bar \nu_0}{\widetilde \nu_0}z) \abs{\frac{\bar \lambda_0}{\widetilde \lambda_0}-1}+\max_{z \wedge \frac{\bar \nu_0}{\widetilde \nu_0}z \leq x\leq z\vee \frac{\bar \nu_0}{\widetilde \nu_0}z} \abs{\phi'(x)}\cdot\abs{\frac{\bar \nu_0}{\widetilde \nu_0}-1}z\\
    &\leq O(\kappa \tilde \lambda_0)  \phi(\frac{\bar \nu_0}{\widetilde \nu_0}z)+ z\phi(z \wedge \frac{\bar \nu_0}{\widetilde \nu_0}z) O(\kappa \tilde \lambda_0) \\
    &\lesssim \kappa \tilde \lambda_0(1+e^{-1}\vee \frac{\tilde \nu_0}{\bar \nu_0} e^{-1})\lesssim \kappa \tilde \lambda_0 \\
     \abs{\partial_z\bar \epsilon_1(z)}&\leq \abs{-\frac{\bar \lambda_0}{\widetilde \lambda_0}\frac{\bar \nu_0}{\widetilde \nu_0}\phi(\frac{\bar \nu_0}{\widetilde \nu_0}z)+\phi(\frac{\bar \nu_0}{\widetilde \nu_0}z)}+\abs{\phi(z)-\phi(\frac{\bar \nu_0}{\widetilde \nu_0}z)}\\
     &\leq \abs{1-\frac{\bar \lambda_0}{\widetilde \lambda_0}\frac{\bar \nu_0}{\widetilde \nu_0}}\phi(\frac{\bar \nu_0}{\widetilde \nu_0}z)+ z\phi(z \wedge \frac{\bar \nu_0}{\widetilde \nu_0}z) O(\kappa \tilde \lambda_0)\\
     &=z\phi(z \wedge \frac{\bar \nu_0}{\widetilde \nu_0}z) O(\kappa \tilde \lambda_0)\text{ (by the definition of }\bar \nu_0)\lesssim \kappa \tilde \lambda_0\\ 
    \abs{\bar \epsilon_2(z)}&\leq \bar \lambda_0 \kappa=(1+O(\kappa \tilde\lambda_0))\tilde\lambda_0\kappa\lesssim \kappa \tilde \lambda_0 \\
    \abs{\partial_z\bar \epsilon_2(z)}&\leq \bar \lambda_0 \bar \nu_0 \kappa=(1+O(\kappa \tilde\lambda_0))\tilde\lambda_0(1+O(\kappa \tilde\lambda_0))\tilde\nu_0\kappa\lesssim \kappa \tilde \lambda_0 \tilde \nu_0 ,\\
     \abs{\bar \epsilon_0(z)}&\lesssim \kappa\tilde \lambda_0, \abs{\partial_z\bar \epsilon_0(z)}\lesssim \kappa \tilde \lambda_0(1+\tilde \nu_0),\text{ for }\kappa\tilde \lambda_0 \text{  small enough, and }1\leq z \leq \bar \nu_0^{-1}
    \end{split}
\end{equation}}{\begin{equation}\label{exterior decomposition}
\begin{split}
    \abs{\bar \epsilon_1(z)}&\leq \abs{\frac{\bar \lambda_0}{\widetilde \lambda_0}\phi(\frac{\bar \nu_0}{\widetilde \nu_0}z)-\phi(\frac{\bar \nu_0}{\widetilde \nu_0}z)}+\abs{\phi(\frac{\bar \nu_0}{\widetilde \nu_0}z)-\phi(z)}\\
    &\leq \phi(\frac{\bar \nu_0}{\widetilde \nu_0}z) \abs{\frac{\bar \lambda_0}{\widetilde \lambda_0}-1}+\max_{z \wedge \frac{\bar \nu_0}{\widetilde \nu_0}z \leq x\leq z\vee \frac{\bar \nu_0}{\widetilde \nu_0}z} \abs{\phi'(x)}\cdot\abs{\frac{\bar \nu_0}{\widetilde \nu_0}-1}z\\
    &\leq O(\kappa \tilde \lambda_0)  \phi(\frac{\bar \nu_0}{\widetilde \nu_0}z)+ z\phi(z \wedge \frac{\bar \nu_0}{\widetilde \nu_0}z) O(\kappa \tilde \lambda_0) \\
    &\lesssim \kappa \tilde \lambda_0(1+e^{-1}\vee \frac{\tilde \nu_0}{\bar \nu_0} e^{-1})\lesssim \kappa \tilde \lambda_0 \\
     \abs{\partial_z\bar \epsilon_1(z)}&\leq \abs{-\frac{\bar \lambda_0}{\widetilde \lambda_0}\frac{\bar \nu_0}{\widetilde \nu_0}\phi(\frac{\bar \nu_0}{\widetilde \nu_0}z)+\phi(\frac{\bar \nu_0}{\widetilde \nu_0}z)}+\abs{\phi(z)-\phi(\frac{\bar \nu_0}{\widetilde \nu_0}z)}\\
     &\leq \abs{1-\frac{\bar \lambda_0}{\widetilde \lambda_0}\frac{\bar \nu_0}{\widetilde \nu_0}}\phi(\frac{\bar \nu_0}{\widetilde \nu_0}z)+ z\phi(z \wedge \frac{\bar \nu_0}{\widetilde \nu_0}z) O(\kappa \tilde \lambda_0)\\
     &=z\phi(z \wedge \frac{\bar \nu_0}{\widetilde \nu_0}z) O(\kappa \tilde \lambda_0)\text{ (by the definition of }\bar \nu_0)\\ 
     &\lesssim \kappa \tilde \lambda_0\\
    \abs{\bar \epsilon_2(z)}&\leq \bar \lambda_0 \kappa=(1+O(\kappa \tilde\lambda_0))\tilde\lambda_0\kappa\lesssim \kappa \tilde \lambda_0 \\
    \abs{\partial_z\bar \epsilon_2(z)}&\leq \bar \lambda_0 \bar \nu_0 \kappa=(1+O(\kappa \tilde\lambda_0))\tilde\lambda_0(1+O(\kappa \tilde\lambda_0))\tilde\nu_0\kappa\lesssim \kappa \tilde \lambda_0 \tilde \nu_0 ,\\
     \abs{\bar \epsilon_0(z)}&\lesssim \kappa\tilde \lambda_0, \abs{\partial_z\bar \epsilon_0(z)}\lesssim \kappa \tilde \lambda_0(1+\tilde \nu_0),\text{ for }\kappa\tilde \lambda_0 \text{  small enough, and }1\leq z \leq \bar \nu_0^{-1}
\end{split}
\end{equation}}
where we have used $z\phi(z \wedge \frac{\bar \nu_0}{\widetilde \nu_0}z)$ is bounded on $\R^+$, mean-value theorem and \eqref{parameter-new-decomposition}. Next,\ifthenelse{\boolean{isSimplified}}{ in a similar manner, }{ }using $0\leq 1-e^{-x}\leq x $ on $[0,\infty)$ and $\phi'=-\phi$, we compute for all $0\leq z\leq 1$,
\ifthenelse{\boolean{isSimplified}}{\begin{equation}\label{interior decomposition}
    \abs{\epsilon_2 (z)}\lesssim \kappa \tilde \lambda_0 z,\quad\abs{\partial_z\epsilon_2 (z)}\lesssim \kappa \tilde \lambda_0 z,\quad \abs{\epsilon_3 (z)}\lesssim z \kappa \tilde\lambda_0\quad \abs{\partial_z\epsilon_3 (z)}\lesssim z \kappa \tilde\lambda_0,\text{ for }\kappa\tilde \lambda_0 \text{  small enough,}
\end{equation}}{\begin{equation}\label{interior decomposition}
    \begin{split}
        \abs{\epsilon_2 (z)}&\leq \abs{\frac{\bar \lambda_0}{\widetilde \lambda_0}-1} ( \abs{1-\phi(\frac{\bar \nu_0}{\widetilde \nu_0} z)}+\abs{\frac{\bar \nu_0}{\tilde \nu_0}}z)\\
        &\lesssim  \abs{\frac{\bar \lambda_0}{\widetilde \lambda_0}-1}\abs{\frac{\bar \nu_0}{\tilde \nu_0}}z\lesssim \kappa \tilde \lambda_0(1+\kappa \tilde \lambda_0)z\leq \kappa \tilde \lambda_0 z\\
        \abs{\partial_z\epsilon_2 (z)}&\leq \abs{\frac{\bar \lambda_0}{\widetilde \lambda_0}-1} \abs{\frac{\bar \nu_0}{\tilde \nu_0}}\abs{ \phi'(\frac{\bar \nu_0}{\tilde \nu_0}z)+1}\\
        &\leq \abs{\frac{\bar \lambda_0}{\widetilde \lambda_0}-1} \abs{\frac{\bar \nu_0}{\tilde \nu_0}}\abs{ -\phi(\frac{\bar \nu_0}{\tilde \nu_0}z)+1}\lesssim \kappa \tilde \lambda_0 z\\
        \abs{\epsilon_3 (z)}&\leq \abs{\phi(\frac{\bar \nu_0}{\widetilde \nu_0} z)-\phi(z)}+\abs{\frac{\bar \nu_0}{\widetilde \nu_0}-1} z\\
        &\leq z\phi(z \wedge \frac{\bar \nu_0}{\widetilde \nu_0}z) O(\kappa \tilde \lambda_0)+O(\kappa \tilde \lambda_0) z\lesssim z \kappa \tilde\lambda_0\\
        \abs{\partial_z\epsilon_3 (z)}&\leq \abs{\frac{\bar \nu_0}{\tilde \nu_0} \phi'(\frac{\bar \nu_0}{\tilde \nu_0} z)-\phi'(\frac{\bar \nu_0}{\tilde \nu_0} z)}+\abs{\phi'(\frac{\bar \nu_0}{\tilde \nu_0} z)-\phi'(z)}+\abs{\frac{\bar \nu_0}{\tilde \nu_0}-1}\\
        &=\abs{\frac{\bar \nu_0}{\tilde \nu_0}-1}\abs{1-\phi (\frac{\bar \nu_0}{\tilde \nu_0} z)}+\abs{\phi(\frac{\bar \nu_0}{\widetilde \nu_0}z)-\phi(z)}\\
        &\leq O(\kappa \tilde \lambda_0) \frac{\nu_0}{\tilde\nu_0}z+z\phi(z \wedge \frac{\bar \nu_0}{\widetilde \nu_0}z) O(\kappa \tilde \lambda_0)\lesssim z \kappa \tilde\lambda_0,\text{ for }\kappa\tilde \lambda_0 \text{  small enough,}
    \end{split}
\end{equation}}
where we have also used part of the estimate we did in \eqref{exterior decomposition}. \ifthenelse{\boolean{isSimplified}}{Next, combining with mean-value theorem and Hölder continuity, we observe 
\begin{alignat}{1}
    &\abs{\epsilon_4(z)}\leq \bar \lambda_0 {||\tilde a_0||}_{C^1[0,1]}\abs{z\bar \nu_0}\leq \tilde \lambda_0 (1+O(\kappa \tilde \lambda_0))\kappa \tilde \nu_0  (1+O(\kappa \tilde \lambda_0))z\lesssim  \kappa\tilde\lambda_0\tilde \nu_0 z,\label{special interior decomposition}\\
    &\abs{\partial_z\epsilon_4(z)}\leq \bar \lambda_0 \bar \nu_0 \kappa (z\bar \nu_0)^{\tilde \epsilon_a}\leq \tilde \lambda_0 (1+O(\kappa \tilde \lambda_0)) (\tilde\nu_0  (1+O(\kappa \tilde \lambda_0)))^{1+\tilde \epsilon_a} \kappa z^{\tilde\epsilon_a}\lesssim \kappa\tilde \lambda_0\tilde\nu_0^{1+\tilde \epsilon_a} z^{\tilde\epsilon_a}\label{special derivative interior decomposition}
\end{alignat}
}{Next, using mean-value theorem, we observe
\begin{equation}\label{special interior decomposition}
    \abs{\epsilon_4(z)}\leq \bar \lambda_0 \norm{\tilde a_0}_{C^1[0,1]}\abs{z\bar \nu_0}\leq \tilde \lambda_0 (1+O(\kappa \tilde \lambda_0))\kappa \tilde \nu_0  (1+O(\kappa \tilde \lambda_0))z\lesssim  \kappa\tilde\lambda_0\tilde \nu_0 z,
\end{equation}
by Hölder continuous property, we have
\begin{equation}\label{special derivative interior decomposition}
\begin{split}
      \abs{\partial_z\epsilon_4(z)}&\leq \bar \lambda_0 \bar \nu_0\abs{\tilde a_0'(z\bar \nu_0)-\tilde a_0'(0)}\\
      &\leq \bar \lambda_0 \bar \nu_0 \kappa (z\bar \nu_0)^{\tilde \epsilon_a}\\
      &\leq \tilde \lambda_0 (1+O(\kappa \tilde \lambda_0)) (\tilde\nu_0  (1+O(\kappa \tilde \lambda_0)))^{1+\tilde \epsilon_a} \kappa z^{\tilde\epsilon_a}\\
      &\lesssim \kappa\tilde \lambda_0\tilde\nu_0^{1+\tilde \epsilon_a} z^{\tilde\epsilon_a},
\end{split}
\end{equation}
for $0\leq z\leq 1$ and $\kappa \tilde \lambda_0$ small enough.} Since we can choose $\tilde\lambda_0<\lambda^*_0$  small enough such that
\begin{equation}\label{estimate on bar nu knot}
     {\tilde \nu_0}\log(\tilde \lambda_0^{-1}) \leq N_0\implies  {\tilde \nu_0}\leq N_0 \log(\tilde \lambda_0^{-1})^{-1}\leq N_0 \log( {\lambda_0^*}^{-1})^{-1}\lesssim 1.
\end{equation}
Combining this fact and\eqref{exterior decomposition}-\eqref{special derivative interior decomposition} we see that
\ifthenelse{\boolean{isSimplified}}{\begin{alignat}{2}\label{summary-of-a-decomposition}
    \abs{\bar \epsilon_0}&\lesssim \kappa \tilde \lambda_0 (1+\tilde \nu_0) z\lesssim \kappa \tilde \lambda_0 z \quad z\in [0,1]\quad &&\abs{\partial_z \bar \epsilon_0}\lesssim \kappa \tilde \lambda_0 (1+\tilde\nu_0^{1+\tilde \epsilon_a}) z^{\tilde\epsilon_a} \lesssim  \kappa \tilde \lambda_0  z^{\tilde\epsilon_a} \quad z\in [0,1]\\
        \abs{\partial_z \bar \epsilon_0}&\lesssim \kappa \tilde \lambda_0(1+\tilde \nu_0)\lesssim \kappa \tilde \lambda_0\quad z\in [1,z^*]\quad &&\abs{\bar \epsilon_0}\lesssim \kappa \tilde \lambda_0\quad z\in [z^*,\bar \nu_0^{-1}].
\end{alignat}}{\begin{equation}\label{summary-of-a-decomposition}
    \begin{split}
    \abs{\bar \epsilon_0}&\lesssim \kappa \tilde \lambda_0 (1+\tilde \nu_0) z\lesssim \kappa \tilde \lambda_0 z \quad z\in [0,1]\\
        \abs{\partial_z \bar \epsilon_0}&\lesssim \kappa \tilde \lambda_0 (1+\tilde\nu_0^{1+\tilde \epsilon_a}) z^{\tilde\epsilon_a} \lesssim  \kappa \tilde \lambda_0  z^{\tilde\epsilon_a} \quad z\in [0,1]\\
        \abs{\partial_z \bar \epsilon_0}&\lesssim \kappa \tilde \lambda_0(1+\tilde \nu_0)\lesssim \kappa \tilde \lambda_0\quad z\in [1,z^*]\\
         \abs{\bar \epsilon_0}&\lesssim \kappa \tilde \lambda_0\quad z\in [z^*,\bar \nu_0^{-1}].
    \end{split}
\end{equation}}

\noindent \textbf{Step 4}. \emph{Applying estimates to see conditions are satisfied for $\tilde a$}. We first define $\bar s_0\geq 1$ to be the larger root of
\begin{equation}
    \bar s_0 e^{-\bar s_0}=\bar \lambda_0=\tilde \lambda_0+O(\kappa \tilde \lambda_0).
\end{equation}
Combing \eqref{parameter-decomposition-restriction} and choose $\kappa\tilde \lambda_0$ small enough such that \begin{equation*}
    \bar s_0 e^{-\bar s_0} \leq \frac{4}{3}\tilde \lambda_0\leq \frac{2}{3}\lambda^*_0\leq \lambda^*_0=s_0^* e^{-s_0^*},
\end{equation*}
since the function $xe^{-x}$ is decreasing on $[1,\infty)$ and both $s_0^*,\bar s_0\geq 1$, we see $\bar s_0 \geq s_0^*$. Next, We check whether the definitions of initial trappedness is indeed partially satisfied for $\bar \lambda_0\bar a_0(z)=\bar \epsilon_0(z)$ and $\kappa$ is relatively small.
\begin{enumerate}
    \item The integral compatibility condition is automatically satisfied by zero-average assumption on $a_0$.
    \item The vanishing speed conditions of $\bar \epsilon_0$ are satisfied by the first and the second equation of \eqref{summary-of-a-decomposition} and the assumption of $1>\tilde \epsilon_a>\epsilon_a$.
     \item We show that the modular equations are satisfied for $\lambda_0^*$ small enough, as long as {$\kappa\tilde\lambda_0=o(1)$,} we compute
     \ifthenelse{\boolean{isSimplified}}{ \begin{equation}\begin{split}
    \bar \lambda_0&:=\bar s_0 e^{-\bar s_0}\implies \log(\bar \lambda_0^{-1})=\bar s_0-\log \bar s_0\simeq \bar s_0\implies \bar s_0=\log(\bar \lambda_0^{-1})+O(\log(\log(\bar\lambda_0^{-1})))\\
    \bar s_0&=\log(\tilde \lambda_0^{-1})+O(\log(\log(\tilde \lambda_0^{-1}))),\text{ for }\tilde\lambda_0\leq \tilde\lambda_0^*\text{ small enough.}\\
    {\bar \nu_0}{\bar s_0}&= \tilde \nu_0 (1+O(\kappa \tilde \lambda_0)) (\log(\tilde \lambda_0^{-1})+O(\log\log\tilde \lambda_0^{-1}))
    ={{\tilde \nu_0}\log(\tilde \lambda_0^{-1})(1+O(\kappa\tilde \lambda_0)+{O(\frac{\log\log \tilde \lambda_0^{-1}}{  \log \tilde \lambda_0^{-1}})})}.
    \end{split}
    \end{equation}
    where we have also used \eqref{parameter-decomposition-restriction} and \eqref{parameter-new-decomposition}. 
    }{\begin{equation}
    \begin{split}
    \bar \lambda_0&:=\bar s_0 e^{-\bar s_0}\implies \log(\bar \lambda_0^{-1})=\bar s_0-\log \bar s_0\simeq \bar s_0\implies \bar s_0=\log(\bar \lambda_0^{-1})+O(\log(\log(\bar\lambda_0^{-1})))\\
    \bar s_0&= \log(\tilde \lambda_0^{-1})+\log(1+O(\kappa \tilde \lambda_0))+O(\log(\log(\tilde \lambda_0^{-1})+\log(1+O(\kappa \tilde \lambda_0))))\\
    &=\log(\tilde \lambda_0^{-1})+O(\kappa \tilde \lambda_0)+O(\log(\log(\tilde \lambda_0^{-1})+O(\kappa \tilde \lambda_0)))\\
    &=\log(\tilde \lambda_0^{-1})+O(\log(\log(\tilde \lambda_0^{-1}))),\text{ for }\tilde\lambda_0\leq \tilde\lambda_0^*\text{ small enough.}\\
    \end{split}
    \end{equation}
    Therefore, since $ \frac{1+1/2}{N_0} \leq {\tilde \nu_0}\log(\tilde \lambda_0^{-1}) \leq N_0(1-1/2)$, we also have
    \begin{equation}
        \begin{split}
        {\bar \nu_0}{\bar s_0}=\frac{\tilde \lambda_0}{\bar \lambda_0} \tilde \nu_0\bar s_0&= \tilde \nu_0 (1+O(\kappa \tilde \lambda_0)) (\log(\tilde \lambda_0^{-1})+O(\log\log\tilde \lambda_0^{-1}))\\
        &={\tilde \nu_0}\log(\tilde \lambda_0^{-1})+O({\tilde \nu_0}\kappa \tilde \lambda_0 \log(\tilde \lambda_0^{-1}))\\
        &= {\tilde \nu_0}\log(\tilde \lambda_0^{-1})+O(\tilde \nu_0  \log \tilde \lambda_0^{-1}\frac{\log\log \tilde \lambda_0^{-1}}{  \log \tilde \lambda_0^{-1}})\\
        &= {\tilde \nu_0}\log(\tilde \lambda_0^{-1})+O(\frac{\log\log \tilde \lambda_0^{-1}}{  \log \tilde \lambda_0^{-1}}),
        \end{split}
    \end{equation}
    as $O(\frac{\log\log \tilde \lambda_0^{-1}}{  \log \tilde \lambda_0^{-1}})=o(1)$ as $\tilde \lambda_0^{-1}\rightarrow\infty$, we can choose $\lambda_0^*\ll 1$ small enough, such that applying the bound on ${\tilde \nu_0}\log(\tilde \lambda_0^{-1})$ guarantees
    \begin{equation}
      \frac{1}{N_0}  \leq \bar \nu_0 \bar s_0\leq N_0.
    \end{equation}}
    \item The weighted interior $L^2$ norm and the exterior $L^\infty$ norm can be estimated,
    \ifthenelse{\boolean{isSimplified}}{by using \eqref{summary-of-a-decomposition}\begin{equation}
        \begin{split}
           \mathcal I_a^2(s_0)&\lesssim \kappa^2  \tilde \lambda_0^2 (1+\tilde\nu_0^{1+\tilde \epsilon_a})^2 \int_0^1 z^{-\alpha+2\tilde \epsilon_a} dz+(z^*-1)\cdot 1\cdot \kappa^2  \tilde \lambda_0^2 (1+\tilde\nu_0)^2\lesssim \kappa^2  \tilde \lambda_0^2 (z^*+1)\\ 
          \mathcal  E_a^2(s_0)&= \max_{[z^*,\bar \nu_0^{-1}]} \bar \epsilon_0^2(z)\leq \kappa^2 \tilde \lambda_0^2,
        \end{split}
    \end{equation}}{\begin{equation}
        \begin{split}
           \mathcal I_a^2(s_0)&=\left(\int_0^1+\int_{1}^{z^*}\right) z^{-\alpha}\partial_z\bar \epsilon_0^2\, dz\\
            &\lesssim \kappa^2  \tilde \lambda_0^2 \int_0^1 z^{-\alpha+2\tilde \epsilon_a} dz+(z^*-1)\cdot 1\cdot \kappa^2  \tilde \lambda_0^2 (1+\tilde\nu_0)^2\\
            &\lesssim  \kappa^2  \tilde \lambda_0^2( (1+\tilde\nu_0^{1+\tilde \epsilon_a})^2+z^*    (1+\tilde\nu_0)^2),\\
            &\lesssim \kappa^2  \tilde \lambda_0^2 (z^*+1)\\ 
          \mathcal  E_a^2(s_0)&= \max_{[z^*,\bar \nu_0^{-1}]} \bar \epsilon_0^2(z)\leq \kappa^2 \tilde \lambda_0^2,
        \end{split}
    \end{equation}
    where we have used the fact that $-\alpha+2\epsilon_a>-1\implies -\alpha+2\tilde \epsilon_a>-1$ and \eqref{summary-of-a-decomposition}.}
    By the previous computations, we know that
    \begin{equation}\label{bar s_0 estimate}
    \begin{split}
             &{\log(\bar \lambda_0^{-1})=\bar s_0-\log \bar s_0 \geq \frac{1}{2}\bar s_0 \quad \forall \bar s_0>0}\\
             &\implies 4\log(\tilde \lambda_0^{-1})\geq \bar s_0 \implies \log(\tilde \lambda_0^{-1})^{-h_a}\lesssim \bar s_0^{-h_a}\\
    \end{split}
    \end{equation}
    therefore, we can choose $\kappa(z^*,\delta,\tilde\lambda_0)\lesssim \min(\delta (\frac{1}{z^*+1})^\frac{1}{2} \tilde \lambda_0^{-1} \log(\tilde \lambda_0^{-1})^{-h_a/2}, \log(\tilde \lambda_0^{-1})^{-h_a/2} \tilde\lambda_0^{-1})$ such that the definitions are satisfied.
   
\end{enumerate}
\noindent \textbf{Step 5}. \emph{Applying estimates to see conditions are satisfied for $\tilde c$. }\ifthenelse{\boolean{isSimplified}}{}{We define $\bar c_0$ such that
\begin{equation}
    \tilde c_0(Z)=\frac{1}{\bar \lambda_0^{1+\sigma}} \bar c_0(Z),\, Z=z\bar \nu_0
\end{equation}
where $\bar \lambda_0$ and $\bar \nu_0$ are defined in \eqref{parameter-new-decomposition}. We choose $\tilde c_0$ from the space:
\begin{equation}
     \tilde c_0 \in \begin{cases}
         X_{\tilde \epsilon_c,\kappa}=\set{f\in C^2[0,1]:\norm{f}_{C^{1,\tilde \epsilon_c}[0,1]}<\kappa,\, f(Z=0)=f'(Z=0)=0,\text{ for } 1>\tilde \epsilon_c> \epsilon_c}\text{ if }\sigma=0\\
         X_{\tilde \epsilon_c,\kappa}=\set{f\in C^3[0,1]:\norm{f}_{C^{0,\tilde \epsilon_c}[0,1]}<\kappa,\, f(Z=0)= f(Z=1)=0,\text{ for } 1>\tilde \epsilon_c> \epsilon_c}\text{ if }\sigma=1
     \end{cases}
\end{equation}} 
For simplicity, we introduce
\begin{equation}
    r(z)=\bar c_0(z\bar \nu_0)=\bar \lambda_0^{1+\sigma}\tilde c_0(z\bar\nu_0),
\end{equation}
and we shall show $r$ satisfies the definition of initial closeness.
\begin{enumerate}
    \item For non-diffusive case, the proof is similar to the case of $\tilde a$. 
    \ifthenelse{\boolean{isSimplified}}{}{\begin{enumerate}
        \item We compute, using \eqref{estimate on bar nu knot} and mean value theorem, that for each $0\leq z \leq \bar \nu_0^{-1}$
        \begin{equation}
            \begin{split}
               \abs{ r'(z)}&=\bar\lambda_0\bar\nu_0\abs{\tilde c_0'(z \bar \nu_0)-\tilde c_0'(0)}\leq\tilde \lambda_0 (1+O(\kappa \tilde \lambda_0)) \bar\nu_0\kappa (z\bar \nu_0)^{\tilde \epsilon_c}\lesssim \kappa z^{\tilde \epsilon_c}\tilde \lambda_0 \\
               \abs{r(z)}&=\bar\lambda_0 \abs{\tilde c_0(z \bar \nu_0)-\tilde c_0(0)}\lesssim z\bar \nu_0\bar\lambda_0\kappa\leq z\bar \nu_0\tilde\lambda_0\kappa,
            \end{split}
        \end{equation}
        therefore, we have $r(z)=o(z^{\epsilon_c})$ and $r'(z)=o(z^{\epsilon_c})$.
        \item From the same computation above, we have, by the choice of $\tilde \epsilon_c>\epsilon_c$,
        \begin{equation}
            \begin{split}
                I_c^2(\bar s_0)&=\int_0^{z^*} z^{-\gamma} r_z^2\lesssim \kappa^2\tilde \lambda_0^2\int_0^{z^*} z^{-\gamma+2\tilde\epsilon_c}\leq C(z^*)\kappa^2\tilde \lambda_0^2\\
                E_c^2(\bar s_0)&=\max_{[z^*,\bar \nu_0^{-1}]} r^2(z)\leq \max_{[z^*,\bar \nu_0^{-1}]} z^2\bar \nu_0^2\tilde\lambda_0^2\kappa^2 \lesssim \tilde\lambda_0^2\kappa^2,
            \end{split}
        \end{equation}
        and on the other hand, since \eqref{bar s_0 estimate}, $$e^{-h_c \bar s_0}\geq e^{-h_c 2 \log(\bar \lambda_0^{-1})}=\bar \lambda_0^{2h_c}=\tilde \lambda_0^{2h_c}(1+O(\kappa\tilde \lambda_0))\gtrsim \tilde\lambda_0^{2h_c},$$
        we see we can choose $\kappa(z^*,\delta,\tilde \lambda_0)$ small enough such that $I_c^2(\bar s_0), E_c^2(\bar s_0)\lesssim e^{-h_c \bar s_0}$.
    \end{enumerate}}
    \item For diffusive case, we proceed by a similar reasoning.
    \begin{enumerate}
        \item We compute, using \eqref{estimate on bar nu knot} and mean value theorem, that for each $0\leq z \leq \bar \nu_0^{-1}$
        \begin{equation}\label{estimate diffusive decomposition}
            \begin{split}
               \abs{r(z)}&=\bar\lambda_0^2 \abs{\tilde c_0(z \bar \nu_0)-\tilde c_0(0)}\lesssim (z\bar \nu_0)^{\tilde \epsilon_c}\bar\lambda_0^2\kappa\lesssim z^{\tilde \epsilon_c} \tilde \nu_0^{\tilde \epsilon_c}\tilde \lambda_0^2\kappa,
            \end{split}
        \end{equation}
        therefore, we have $r(z)=o(z^{\epsilon_c})$.
        \item From the same computation above, we have, by the choice of $\tilde \epsilon_c>\epsilon_c$ and $k,\eta_0$:
        \ifthenelse{\boolean{isSimplified}}{\begin{equation}
            \begin{split}
              \mathcal  T_{k\eta_0,2\eta_0}(\bar s_0)&=(\int_0^{z^*} z^{-k\eta_0} r^{2\eta_0})^{\frac{1}{2\eta_0}}\lesssim  \tilde \nu_0^{\tilde \epsilon_c}\tilde \lambda_0^2\kappa (\int_0^{\bar \nu_0^{-1}} z^{-k\eta_0+2\eta_0\tilde\epsilon_c})^\frac{1}{2\eta_0} \\
                &\simeq \tilde \nu_0^{\tilde \epsilon_c}\tilde \lambda_0^2\kappa (\bar\nu_0^{-1})^{\frac{1}{2\eta_0}-\frac{k}{2}+\tilde\epsilon_c}=\bar\nu_0^{\frac{k}{2}-\frac{1}{2\eta_0}}\tilde \lambda_0^2\kappa\lesssim \tilde \lambda_0^2\kappa
            \end{split}
        \end{equation}}{ \begin{equation}
            \begin{split}
              \mathcal  T_{k\eta_0,2\eta_0}(\bar s_0)&=(\int_0^{z^*} z^{-k\eta_0} r^{2\eta_0})^{\frac{1}{2\eta_0}}\\
               & \lesssim  \tilde \nu_0^{\tilde \epsilon_c}\tilde \lambda_0^2\kappa (\int_0^{\bar \nu_0^{-1}} z^{-k\eta_0+2\eta_0\tilde\epsilon_c})^\frac{1}{2\eta_0}\\
               &\simeq \tilde \nu_0^{\tilde \epsilon_c}\tilde \lambda_0^2\kappa (\bar\nu_0^{-1})^{\frac{1}{2\eta_0}-\frac{k}{2}+\tilde\epsilon_c} \\
               &=\bar\nu_0^{\frac{k}{2}-\frac{1}{2\eta_0}}\tilde \lambda_0^2\kappa\lesssim \tilde \lambda_0^2\kappa
            \end{split}
        \end{equation}}
        and on the other hand, by \eqref{bar s_0 estimate}, $$e^{-l\bar s_0/2}\geq e^{-l\log(\bar \lambda_0^{-1})}=\bar \lambda_0^{l}=\tilde \lambda_0^{l}(1+O(\kappa\tilde \lambda_0))\gtrsim \tilde\lambda_0^{l},$$
        we see we can choose $\kappa(\tilde \lambda_0)$ small such that $\mathcal I_c^2(\bar s_0),\mathcal E_c^2(\bar s_0)\lesssim e^{-h_c \bar s_0}$.
        \item Finally, we check whether we can choose the initial weighted maximum norm to be as small as possible. Let $0<\xi<\epsilon_c<\tilde \epsilon_c$ and $p>0$, upon using \eqref{estimate diffusive decomposition}, we see
        \begin{equation}
            \abs{r z^{-\xi}}\lesssim z^{\tilde \epsilon_c-\xi}  \bar \nu_0^{\tilde \epsilon_c}\bar\lambda_0^2\kappa\leq ((\bar \nu_0)^{-1})^{\tilde \epsilon_c-\xi}  \bar \nu_0^{\tilde \epsilon_c}\bar\lambda_0^2\kappa= \bar \nu_0^\xi \bar\lambda_0^2\kappa \lesssim \tilde\nu_0^\xi \tilde\lambda_0^2\kappa \lesssim  \tilde\lambda_0^2\kappa,
        \end{equation}
            and on the other hand, by \eqref{bar s_0 estimate},\ifthenelse{\boolean{isSimplified}}{by the same reasoning above, $e^{-p\bar s_0/2}\gtrsim \tilde\lambda_0^{p},$ }{$$e^{-p\bar s_0/2}\geq e^{-p\log(\bar \lambda_0^{-1})}=\bar \lambda_0^{p}=\tilde \lambda_0^{p}(1+O(\kappa\tilde \lambda_0))\gtrsim \tilde\lambda_0^{p},$$} 
            we see we can choose $\kappa(\tilde \lambda_0)$ small, such that $\norm{r z^{-\xi}}_{L^\infty[0,\bar\nu_0^{-1}]}<e^{-p\bar s_0/2}.$
    \end{enumerate}
\end{enumerate}
\end{proof}

\subsection{Conversion to the Original Coordinate System}
Assume that the initial condition has the form of
\begin{equation}\label{initial decomp convert}
\begin{split}
  &a_0(Z)=\frac{1}{\bar \lambda_0}\phi(\frac{Z}{\bar\nu_0})+\bar a_0(Z)=\frac{1}{\bar \lambda_0}(\phi(z)+\underbrace{\bar \lambda_0\bar a_0(z\bar\nu_0)}_{\bar \epsilon_0(z)})\\
     & c_0(Z)=\frac{1}{\bar \lambda_0^{1+\sigma}} \bar c_0(Z)=\frac{1}{\bar \lambda_0^{1+\sigma}} \bar c_0(z\bar\nu_0)\\
     & Z=z\bar \nu_0,\, \bar \lambda_0=\bar s_0 e^{-\bar s_0}
\end{split}
\end{equation}
and that it leads to solution 
\begin{equation}\label{decomp convert}
\begin{split}
  &a(t,Z)=\frac{1}{\bar \lambda(t)}\phi(\frac{Z}{\bar\nu(t)})+\bar a(t,Z)=\frac{1}{\bar \lambda(t)}(\phi(z)+\bar \lambda(t)\bar a(t,z\bar\nu(t)))=\frac{1}{\bar \lambda(s)}(\phi(z)+\underbrace{\bar \lambda(s)\bar a(s,z)}_{\bar\epsilon(s,z)})\\
     & c(t.Z)=\frac{1}{\bar \lambda(t)^{1+\sigma}} \bar c(t,Z)=\frac{1}{\bar \lambda(t)^{1+\sigma}} \bar c(t,z\bar\nu(t))=\frac{1}{\bar \lambda(s)^{1+\sigma}} \bar c(s,z)\\
     &Z=z\bar\nu(t)=z\bar\nu(s),\,\frac{d s}{dt}=\frac{1}{\bar\lambda(t)},\,  s(t=0)=\bar s_0.
\end{split}
\end{equation}
\ifthenelse{\boolean{isSimplified}}{\begin{corollary}\label{conversion corollary}
     For each fixed choices of parameters in the range of \eqref{final-parameter-conditions-diffusive} (or \eqref{particular-parameter-choice-nc}), let the initial condition and its corresponding solution be defined in \eqref{initial decomp convert} and \eqref{decomp convert}. Suppose that $\bar \epsilon,\bar c,\bar\nu,\bar\lambda, s$ is initially close (Definition \ref{def:ini crit-beta=0:diffusive} or \ref{def:ini crit-beta=0:diffusive-nc}), and trapped (Definition \ref{def:crit-beta=0:diffusive} or \ref{def:crit-beta=0:diffusive-nc}) for all time $\infty\geq s\geq s(0)=\bar s_0$, then there exists some time $T$ such that the following asymptotic relations hold as $t\rightarrow T^-$. (Notice that $1<h_a<2$, by choice \eqref{final-parameter-conditions-diffusive} and \eqref{particular-parameter-choice-nc}.)
        \begin{alignat}{2}
              &\bar\lambda=T-t+O((T-t)\abs{\log (T-t)}^{-h_a+1})\\
              &\bar\nu=\abs{\log (T-t)}^{-1}+O(\abs{\log (T-t)}^{-h_a})\\
            &\norm{\bar \epsilon(t,Z)}_{L^\infty[0,1]}=O\left(\abs{\log (T-t)}^{-\frac{h_a}{2}}\right).
        \end{alignat}
     For non-diffusive case, we have (Notice $h_c>0$)
     \begin{equation}
         \begin{split}
             \norm{\bar c(t,Z)}_{L^\infty[0,1]}\leq C(T-t)^{\frac{h_c}{2}}.
         \end{split}
     \end{equation}
     For diffusive case, if in addition, we have $\norm{{\bar c_0(Z) } (\frac{Z}{\bar \nu_0})^{-\xi}}_{L^{\infty}(0,1)}\leq e^{-p\bar s_0}$, then we also have
     \begin{equation}
         \begin{split}
             \norm{\bar c(t,Z)}_{L^\infty[0,1]}\leq (T-t)^{p} \abs{\log(T-t)}^\xi,
         \end{split}
     \end{equation}
     for  $\xi=\frac{\alpha+1}{4}, p=\frac{\alpha+1}{16}$.
\end{corollary}}{\begin{corollary}\label{conversion corollary}
     For each fixed choices of parameters in the range of \eqref{final-parameter-conditions-diffusive} (or \eqref{particular-parameter-choice-nc}), let the initial condition and its corresponding solution be defined in \eqref{initial decomp convert} and \eqref{decomp convert}. Suppose that $\bar \epsilon,\bar c,\bar\nu,\bar\lambda, s$ is initially close (Definition \ref{def:ini crit-beta=0:diffusive} or \ref{def:ini crit-beta=0:diffusive-nc}), and trapped (Definition \ref{def:crit-beta=0:diffusive} or \ref{def:crit-beta=0:diffusive-nc}) for all time $\infty\geq s\geq s(0)=\bar s_0$, then there exists some time $T$ such that the following asymptotic relations hold as $t\rightarrow T^-$. (Notice $1<h_a<2$)
     \begin{equation}
     \begin{split}
           a(t,Z)&=\frac{1}{\bar \lambda(t)}\phi(\frac{Z}{\bar \nu(t)})+\bar a(t,Z)=\frac{1}{\bar \lambda(t)}(\phi(\frac{Z}{\bar \nu(t)})+\bar \epsilon(t,Z)),\\
             \bar\lambda&=T-t+O((T-t)\abs{\log (T-t)}^{-h_a+1})\\
            \bar\nu&=\abs{\log (T-t)}^{-1}+O(\abs{\log (T-t)}^{-h_a})\\
            \norm{\bar a(t,Z)}_{L^\infty[0,1]}&=O\left(\frac{1}{T-t}\left(\abs{\log (T-t)}^{-\frac{h_a}{2}}\right)\right)\\
            \norm{\bar \epsilon(t,Z)}_{L^\infty[0,1]}&=O\left(\abs{\log (T-t)}^{-\frac{h_a}{2}}\right)\\
            a(t,Z)&=\frac{1}{T-t}\phi\left(Z(\abs{\log(T-t)})\right)+O\left(\frac{1}{T-t}\abs{\log (T-t)}^{-h_a+1}\right).
     \end{split}
     \end{equation}
     For non-diffusive case, we have (Notice $h_c>0$)
     \begin{equation}
         \begin{split}
             c(t,Z)=\frac{1}{\bar \lambda(t)}\bar c(t,Z),\quad \norm{\bar c(t,Z)}_{L^\infty[0,1]}\leq C(T-t)^{\frac{h_c}{2}}.
         \end{split}
     \end{equation}
     For diffusive case, if in addition, we have $\norm{{\bar c_0(Z) } (\frac{Z}{\bar \nu_0})^{-\xi}}_{L^{\infty}(0,1)}\leq e^{-p\bar s_0}$, then we also have
     \begin{equation}
         \begin{split}
             c(t,Z)=\frac{1}{\bar \lambda(t)^2}\bar c(t,Z),\quad \norm{\bar c(t,Z)}_{L^\infty[0,1]}\leq (T-t)^{p} \abs{\log(T-t)}^\xi,
         \end{split}
     \end{equation}
     for  $\xi=\frac{\alpha+1}{4}, p=\frac{\alpha+1}{16}$.
\end{corollary}
}
\begin{proof}
    We invert the self-similar transformation into the original coordinate system. Fix choices of parameters to be in the range of \eqref{final-parameter-conditions-diffusive} (or \eqref{particular-parameter-choice-nc}). By A-priori estimate Lemma \ref{lemma:smoothmodulation:diffusion} (or Lemma \ref{lemma:smoothmodulation:diffusion-nc}), \ifthenelse{\boolean{isSimplified}}{the term $T=\int_{\bar s_0}^\infty \bar\lambda(s)ds$ is well-defined and}{we are able to define
\begin{equation}
    T=\int_{\bar s_0}^\infty \bar\lambda(s)ds,
\end{equation}
and we can express $t$ in terms of self-similar time $s$:} 
\ifthenelse{\boolean{isSimplified}}{\begin{equation}
\begin{split}
      t(s)&=\int_{\bar s_0}^s \lambda(s)ds=T-\int_{ s}^\infty  se^{-s}(\tilde\lambda_\infty+O(s^{-h_a+1}) ) ds=T-\tilde\lambda_\infty s e^{-s}+O(s^{-h_a+2} e^{-s}),
\end{split}
\end{equation}}{\begin{equation}
\begin{split}
      t(s)&=\int_{\bar s_0}^s \lambda(s)ds\\
      &=\left(\int_{\bar s_0}^\infty-\int_{ s}^\infty\right) \lambda(s)ds\\
      &=T-\int_{ s}^\infty  se^{-s}(\tilde\lambda_\infty+O(s^{-h_a+1}) ) ds\\
      &=T-\tilde\lambda_\infty s e^{-s}+\tilde\lambda_\infty e^{-s}+O(\Gamma(-h_a+3,s))\\
      &=T-\tilde\lambda_\infty s e^{-s}+O(s^{-h_a+2} e^{-s}),
\end{split}
\end{equation}}
where we have used $1<h_a<2$. Therefore, we see
\begin{equation}
\resizebox{.93 \textwidth}{!}{$\displaystyle
    T-t=\tilde\lambda_\infty s e^{-s}(1+O(s^{-h_a+1} ))\implies \log (T-t)=\log(\tilde\lambda_\infty)+\log(s)-s+O(s^{-h_a+1})\implies s\simeq \abs{\log (T-t)},$}
\end{equation}
and plugging this back to the first equation above,
\begin{equation}
    \tilde\lambda_\infty s e^{-s}\ifthenelse{\boolean{isSimplified}}{}{=(T-t)(1+O(s^{-h_a+1} ))^{-1}=(T-t)(1+O(s^{-h_a+1} ))}=T-t+O((T-t)\abs{\log (T-t)}^{-h_a+1}).
\end{equation}
Therefore, combining the previous two equations and Lemma \ref{lemma:smoothmodulation:diffusion} (or Lemma \ref{lemma:smoothmodulation:diffusion-nc}), we see
\begin{equation}
    \begin{split}
        \bar\lambda&= s e^{-s}\big(\tilde \lambda_\infty + O(s^{-h_a+1}) \big)=T-t+O((T-t)\abs{\log (T-t)}^{-h_a+1})\\
\bar\nu&=\frac{1}{s}+O(s^{-h_a})=\abs{\log (T-t)}^{-1}+O(\abs{\log (T-t)}^{-h_a}).
    \end{split}
\end{equation}
For non-diffusive case, applying Lemma \ref{lemma:prelim-est:diffusive} gives
\begin{equation}
    \begin{split}
       & \norm{\bar \epsilon(s,z)}_{L^\infty(0,\nu^{-1}(s))}\leq C(z^*) s^{-\frac{h_a}{2}}\lesssim_{z^*} \abs{\log (T-t)}^{-\frac{h_a}{2}}\\
       & \norm{\bar c(s,z)}_{L^\infty(0,\nu^{-1}(s))}\leq C(z^*) e^{-\frac{h_c}{2}s}\lesssim_{z^*} (T-t)^{\frac{h_c}{2}}.
    \end{split}
\end{equation}
For diffusive case, similarly, applying Lemma \ref{lemma:prelim-est:diffusive-nc} gives
\begin{equation}
    \begin{split}
      & \norm{\bar \epsilon(s,z)}_{L^\infty(0,\nu^{-1}(s))}\leq C(z^*) s^{-\frac{h_a}{2}}\lesssim_{z^*} \abs{\log (T-t)}^{-\frac{h_a}{2}}\\
    \end{split}
\end{equation}
and \ifthenelse{\boolean{isSimplified}}{the second part of Proposition }{Corollary }\ref{l infinity corollary nc} reads
\begin{equation}
       \norm{\frac{\bar c(\bar s_0,z) }{z^{\xi}}}_{L^{\infty}(0,\nu^{-1}(s_0))}\leq e^{-p\bar s_0}\implies \norm{\bar c(s,z) }_{L^{\infty}(0,\nu^{-1}(s))}\leq C e^{-ps} s^{\xi},\\
\end{equation}
where $\xi=\frac{\alpha+1}{4}, p=\frac{\alpha+1}{16}$, which can be converted as follows:
\begin{equation}
    \begin{split}
         &\norm{\frac{\bar c(\bar s_0,z) }{z^{\xi}}}_{L^{\infty}(0,\nu^{-1}(\bar s_0))}=\norm{{\bar c_0(Z) } (\frac{Z}{\bar \nu_0})^{-\xi}}_{L^{\infty}(0,1)}\leq e^{-p\bar s_0}\\
         &\implies\norm{{\bar c(s,z) }}_{L^{\infty}(0,\nu^{-1}(s))}=\norm{{\bar c(t,Z) }}_{L^{\infty}(0,1)}\leq (T-t)^{p} \abs{\log(T-t)}^\xi.
    \end{split}
\end{equation}
\ifthenelse{\boolean{isSimplified}}{}{Now, since $\phi(x+y)=\phi(x)\phi(y)$ and $1<h_a<2$, we have
     \begin{equation}
     \begin{split}
          a(t,Z)&=\frac{1}{\bar \lambda(t)}\phi(\frac{Z}{\bar \nu(t)})+\bar a(t,Z)\\
          &=\frac{1}{T-t}\left(1+O(\abs{\log (T-t)}^{-h_a+1})\right)\phi\left(Z(\abs{\log(T-t)}+O(\abs{\log (T-t)}^{-h_a+1}))\right)\\
          &\qquad+O\left(\frac{1}{T-t}\left(\abs{\log (T-t)}^{-\frac{h_a}{2}}+O(\abs{\log (T-t)}^{-\frac{3}{2}h_a+1})\right)\right)\\
          &=\frac{1}{T-t}\left(1+O(\abs{\log (T-t)}^{-h_a+1})\right)\phi\left(Z(\abs{\log(T-t)})\right)\underbrace{\phi(O(\abs{\log (T-t)}^{-h_a+1})Z)}_{1+ O(\abs{\log (T-t)}^{-h_a+1})Z}\\
          &\qquad+O\left(\frac{1}{T-t}\left(\abs{\log (T-t)}^{-\frac{h_a}{2}}\right)\right)\\
          &=\frac{1}{T-t}\phi\left(Z(\abs{\log(T-t)})\right)+O\left(\frac{1}{T-t}\abs{\log (T-t)}^{-h_a+1}\right),
     \end{split}
     \end{equation}}
\end{proof}
\ifthenelse{\boolean{isSimplified}}{\begin{remark}
    For both cases, $a(t,Z)$ can be more explicitly written as
    \begin{equation}
        a(t,Z)=(T-t)^{-1+Z}+O\left(\frac{1}{T-t}\abs{\log (T-t)}^{-h_a+1}\right).
    \end{equation}
\end{remark}}{}
Finally, in view of Proposition \ref{smooth:pr:bootstrap:diffusive} or Proposition \ref{l infinity corollary nc}, Corollary \ref{conversion corollary}, and Lemma \ref{decomposition lemma}, Theorem \ref{theorem:critical} and Theorem \ref{theorem:critical:diffusive} are proved.

\section*{Data availability statements}
Data sharing is not applicable to this article, as no datasets were generated or analyzed during the current study.

{
\section*{Acknowledgment}
S. I. and L. Q are partially supported by NSERC grant No. 371637-2025. E.S.T. have benefited from the inspiring environment of the CRC 1114 “Scaling Cascades in Complex Systems”, Project Number 235221301, Project C09, funded
by the Deutsche Forschungsgemeinschaft (DFG). Moreover, this work was also supported in part by the DFG Research Unit FOR 5528 on Geophysical Flows and was supported by the National Science Center (Poland), project 2021/43/B/ST1/02851. Q.L. was partially supported by the Simons Foundation (SFI-MPS-TSM-00013384).
}
\bibliographystyle{plain}   
\bibliography{manuscript}
\end{document}